\PassOptionsToPackage{vlined,linesnumbered}{algorithm2e}
\documentclass[]{alt2027}

\usepackage{times}
\usepackage{mathtools,enumitem,booktabs,microtype}
\hypersetup{hidelinks}
\renewenvironment{algorithm}[1][htbp]
  {\begin{algorithm2e}[#1]}{\end{algorithm2e}}
\SetKwInput{KwInput}{Input}
\SetKwInput{KwOutput}{Output}

\numberwithin{theorem}{section}
\newtheorem{assumption}[theorem]{Assumption}

\renewenvironment{proof}[1][\proofname]
  {\par\noindent{\bfseries\upshape #1\ }}
  {\jmlrQED}

\newcommand{\E}{\mathbb E}
\newcommand{\Prob}{\mathbb P}
\newcommand{\dd}{\,\mathrm d}
\newcommand{\cF}{\mathcal F}
\newcommand{\cP}{\mathcal P}
\newcommand{\eps}{\varepsilon}
\newcommand{\Ber}{\mathrm{Ber}}
\newcommand{\KL}{\mathrm{KL}}
\newcommand{\gft}{\mathrm{gft}}

\newcommand{\Xinfo}{\mathcal X_{\mathrm{info}}}
\newcommand{\Xminus}{\mathcal X_-}
\newcommand{\Xplus}{\mathcal X_+}

\newcommand{\CL}{C_{\lambda,L}}

\title{From Fairness to Efficiency in Repeated Bilateral Trade}

\begin{document}
\maketitle


\vspace{-1cm}

\begin{center}
Fran\c{c}ois Bachoc$^1$, Tommaso Cesari$^2$
		and Roberto Colomboni$^3$ \\
		$^1$University of Lille and Institut Universitaire de France (IUF), France\\
		$^2$EECS, University of Ottawa, Canada \\
		$^3$ School of Mathematics, University of Bristol, United Kingdom
\end{center}

\begin{abstract}
Recently, \cite{bachoc2026repeated} proposed an axiomatic approach to fairness for repeated bilateral trade. Their resulting objective is a H\"older mean of seller and buyer utilities with non-positive order $\lambda$, interpolating between Rawlsian fairness at $\lambda = -\infty$ and Nash fairness at $\lambda = 0$. Here we study this family of H\"older means
in the complementary setting of orders $\lambda \in (0,1]$, 
which constitutes a tradeoff between fairness and efficiency, with the weight of efficiency increasing with $\lambda$, and where $\lambda=1$ 
yields the existing gain from trade criterion, focusing solely on efficiency.
We obtain an axiomatic characterization of these H\"older means, relying on a Pigou-Dalton principle of transferring from rich to poor. 
From an online learning perspective, we show that, contrary to the case $\lambda \le 0$, the problem is unlearnable in full generality, even under seller's and buyer's valuation independence. We focus on seller's and buyer's valuations with  distributions having bounded densities, in which case we obtain matching lower and upper bounds on the cumulative regret, of order $\widetilde{\Theta}(T^{2/3})$ with respect to the time horizon $T$.
The upper bounds rely on a reconstruction identity where the expected objective is written as a gain-from-trade term plus a two-dimensional integral involving the seller and buyer distribution and survival functions, and a kernel function. For $0 < \lambda < 1/2$, the strong singularity of the kernel is addressed by two main ingredients.
First, we propose a new zooming algorithm that sequentially rejects or refines layers of the set of prices. Second, we propose
a tailored decomposition of the two-dimensional integral, enabling us to truncate it with a controlled impact on the truncation error. For $1/2 \le \lambda \le 1$, the singularity of the kernel is mild enough, so that the existing scouting bandit algorithm, already proposed for the case $\lambda=1$, can be extended.
\end{abstract}

\section{Introduction}

Consider repeated bilateral trade, where at round $t \in \mathbb{N}$ a learner posts a price $P_t \in [0,1]$, a seller with valuation $S_t \in [0,1]$ and a buyer with valuation $B_t \in [0,1]$ arrive independently, and trade happens if and only if $S_t \le P_t \le B_t$. If trade occurs, the seller's gain is $P_t - S_t$ and the buyer's is $B_t - P_t$. The learner does not observe $S_t$ or $B_t$ but does observe the bits $Y_t := \mathbb{I}\{S_t \le P_t\}$ and $Z_t := \mathbb{I}\{P_t \le B_t\}$. The classical objective in this setting, the \emph{gain from trade}, measures the learner's performance according to the sum of the participants' gains $(P_t - S_t) + (B_t - P_t) = B_t - S_t$, whenever trade happens. Our goal is instead to optimize a notion of fairness, taking into account how evenly this gain sum is split between buyer and seller. The reference \cite{bachoc2026repeated} considered a family of objectives measuring fairness with H\"older means, parameterized by their H\"older order $\lambda \in (-\infty, 0]$: for
$x,y \ge 0$, for $\lambda < 0$,
\begin{equation} \label{eq:holder:means}
M_\lambda(x, y) := \left(\frac{x^\lambda + y^\lambda}{2}\right)^{1/\lambda},
\end{equation}
$M_{-\infty}(x,y):=\min(x,y)$
and $M_0(x,y):=\sqrt{xy}$.
In this setting, the learner's fairness performance is $M_\lambda(P_t - S_t, B_t - P_t)$, called \emph{$\lambda$-fair gain from trade}, whenever trade occurs.
Note that  $M_{-\infty}$ corresponds to Rawlsian fairness, that $M_0$ is associated with Nash's bargaining solution and that fairness becomes less stringent as $\lambda$ increases. Furthermore, for any $\lambda \in [- \infty,0]$,
from the convention $1 / \infty = 0$, we have
$M_{\lambda}(x,0) = M_{\lambda}(0,y) = 0$, that is, a division of the total gain of the form ``one takes it all'' receives fairness score zero.

\subsection{Contributions}
\label{subsection:contributions}

In this paper, we consider the H\"older means $M_{\lambda}$ in \eqref{eq:holder:means}, but with $\lambda \in (0,1]$ and again, the learner's fairness performance is $M_\lambda(P_t - S_t, B_t - P_t)$, whenever trade occurs.
For this range of $\lambda$, for $x,y>0$, we have $M_{\lambda}(x,0) >0$ and  $M_{\lambda}(0,y) > 0$, so that completely unfair trades are given non-zero score. 
Since also $M_1(x,y) = (x+y)/2$ is half of the gain from trade, the family $M_{\lambda}, \lambda \in (0,1]$, interpolates between fairness at $\lambda=0$ and efficiency at $\lambda=1$.
Note that $M_1$ was studied in \cite{cesa2024bilateralMOR}.
We provide an axiomatic characterization of this interpolation in Appendix~\ref{app:axioms}.

We then focus on the cumulative regret, and provide algorithms, their upper bounds, and matching lower bounds for all positive H\"older orders $\lambda \in (0, 1]$.
Together with the results of \cite{bachoc2026repeated} for $\lambda \le 0$, we exhibit a phase transition at $\lambda = 0$. For $\lambda \le 0$, sublinear regret can be obtained for arbitrary distributions of $S_t$ and $B_t$ (that are assumed to be independent), while for any $\lambda \in (0,1]$ we prove a lower bound with linear regret for arbitrary distributions, see Section \ref{section:linear:lower}.
For this lower bound, we use a ``needle in a haystack'' scheme as in \cite{cesa2024bilateralMOR}.

Then, for any fixed $\lambda \in (0,1]$, we assume that $S_t$ and $B_t$ have bounded densities. We obtain upper and lower bounds on the cumulative regret, of orders $\widetilde{\Theta}(T^{2/3})$. The lower bound
is provided in Section \ref{section:LB:bounded:dens} and its
proof follows a revealing-action structure, extending \cite{bachoc2024fairNeurIPS,cesa2024bilateralMOR,bachoc2026repeated}. Regarding the algorithms and upper bounds, we rely on a representation lemma of the expected fair gain from trade, in the form of a gain-from-trade term plus a two-dimensional integral involving a kernel function and the distribution and survival functions of $S_t$ and $B_t$, see Section \ref{section:representation:lemma}.

Then, we treat two ranges of $\lambda$ separately. 
For $\lambda \in (0, 1/2)$, addressed in Section \ref{sec:hard-alg-main}, 
the kernel is strongly singular. As a consequence, we need to construct a decomposition of it that enables us to truncate it, to filter out the singularity, with sufficient control on the resulting bias. Also, we provide a zooming algorithm that successfully refines or discards regions of the set of possible prices, $[0,1]$. 
Consider next $\lambda \in [1/2, 1]$. For $\lambda =1$ we simply refer to \cite{cesa2024bilateralMOR}. For $\lambda \in [1/2,1)$, the previous decomposition and zooming algorithm are not needed, and we can more simply adapt the scouting bandit algorithm of \cite{cesa2024bilateralMOR}, see Section \ref{subsection:lambda:half:one}.

\subsection{Techniques and challenges}

The main challenge for the upper bounds is the setting $0 < \lambda < 1/2$, where as discussed above, the kernel has a strong singularity. This singularity obstructs a direct extension of the global estimation analysis of \cite{bachoc2026repeated} with the concentration rate required for our regret bound. The zooming algorithm we suggest proceeds in epochs, where at each epoch $m$ ($m=0,1\dots$), an estimation accuracy of order $\Delta_m = 2^{-m}$ is aimed for, with sufficiently high probability. At epoch $m$, this accuracy is aimed only for an active layer, which is a subset of $[0,1]$, and whose complement consists of layers that were eliminated during previous epochs. The estimation procedure at epoch $m$ relies on two main original ingredients. First, we sample exploration prices independently on the entire price set $[0,1]$, but with a non-uniform density that is larger in regions where we currently have stronger confidence that they contain close-to-optimal prices for the fair gain from trade. By precisely tuning the non-uniformity, we can post $\widetilde{O}(\Delta_m^{-3})$ prices but paying only  $\widetilde{O}(\Delta_m^{-2})$ regret, and achieving accuracy of order $\Delta_m$. Second, given these non-uniform samples, the estimation procedure estimates the double integral stemming from the representation lemma discussed above.
Similarly as in \cite{bachoc2026repeated}, we use a double sum but, in contrast, we also need to introduce a new involved decomposition that enables us to lower the bias caused by a truncation step.

Our lower bounds under bounded densities rely on constructing a smoothed revealing-action instance, similarly as in \cite{bachoc2026repeated}. 
We construct environments with distributions that differ by towers of fixed width that have mass shifts sufficient to reverse which of two separated reward peaks is larger. Crucially the only prices from which one can gain information to distinguish environments come with a constant regret gap, thus separating reward and information. 
For a shift of mass $\varepsilon$,
each informative sample has KL divergence $O(\varepsilon^2)$, so one must balance $\Omega(\varepsilon^{-2})$ costly informative samples against $\Omega(T\varepsilon)$ regret in at least one environment from choosing the wrong side, leading to an $\Omega(T^{2/3})$ lower bound, by tuning $\varepsilon=\Theta(T^{-1/3})$. 
The technical arguments are however significantly more challenging than in \cite{bachoc2026repeated}, since our distributions for $S_t$ and $B_t$, in both environments, must have bounded densities, with a fixed bound.

\subsection{Model and objectives}
\label{sec:model-objectives}

The learner's task is the following. At each round $t$, it chooses a price $P_t \in [0, 1]$ to post. Then a pair of seller and buyer valuations $(S_t, B_t)$ is drawn from a product\footnote{Similarly as in \cite{cesa2024bilateralMOR,bachoc2026repeated}, the dependence between buyer and seller valuations would make the problem unlearnable.} distribution, $(S,B) \sim \nu = \nu_S \otimes \nu_B$, independently from round to round. The learner observes whether or not a trade takes place according to the feedback
\begin{equation}
\label{eq:Yt:Zt}
Y_t := \mathbb{I}\{S_t \le P_t\}
\qquad\text{and}\qquad
Z_t := \mathbb{I}\{P_t \le B_t\}.
\end{equation}

In particular, the learner never observes the actual valuations of sellers or buyers. For a distribution $\nu = \nu_S \otimes \nu_B$ of product form, we shall denote $F_S$ for the cdf of the distribution of the seller, and write $F_B^\circ$ for the survival function of the distribution of the buyer: $F_B^\circ(p) = \Prob(B \ge p)$. For $1 \le L < \infty$, we write $\mathcal D_L$ for the class of distributions $\nu \in \mathcal P([0, 1]^2)$ of product form $\nu = \nu_S \otimes \nu_B$ for which both $S$ and $B$ take their values in $[0, 1]$ and admit a density that is bounded by $L$. If $\nu \in \mathcal D_L$, then, for all $0 \le x \le y \le 1$:
\[
0 \le F_S(y) - F_S(x) \le L(y-x),
\qquad
0 \le F_B^\circ(x) - F_B^\circ(y) \le L(y-x).
\]

For $\lambda \in [0, 1]$, recall the definition of $M_{\lambda}$ in Section 
\ref{subsection:contributions} and in \eqref{eq:holder:means}.
The $\lambda$-fair gain from trade is
\[
\gft_\lambda(p,s,b)
:=
\mathbb{I}\{s\le p\le b\}M_\lambda(p-s,b-p),
~ ~ ~ ~
\lambda,p,s,b \in [0,1].
\]

For all $\nu \in \mathcal P([0, 1]^2)$, define the mean of the gain from trade, parameterized by $\lambda \in [0, 1]$, as the function $G_\lambda : [0, 1] \to \mathbb R$ defined by:
\[
G_\lambda(p) = \E[\gft_\lambda(p,S,B)],
\]
where $(S, B) \sim \nu$. We shall write
\[
G_\lambda^\star := \sup_{p\in[0,1]} G_\lambda(p).
\]
Given a learner strategy $\alpha$, its regret when playing against $\nu$, at horizon $T \in \mathbb{N}$, is
\[
R_\lambda^\nu(\alpha,T)
=
\E_\nu\left[
\sum_{t=1}^T
\bigl(G_\lambda^\star-G_\lambda(P_t)\bigr)
\right],
\]
where $P_t$ denotes the price posted by $\alpha$ at time $t$. Our goal is to understand the dependence of
\[
\inf_\alpha \sup_{\nu\in\mathcal D_L} R_\lambda^\nu(\alpha,T),
\]
for fixed $\lambda$, with respect to $T$.

\subsection{Related work}

The closest references to our work address fair bilateral trade.
\cite{bachoc2024fairNeurIPS} considered repeated bilateral trade from a fairness perspective by optimizing $M_{-\infty}$. \cite{bachoc2026repeated} extended this study to the entire range $M_{\lambda}$, $\lambda \le 0$. Also, 
\cite{cesari2025volumeICLR} studied bilateral trade with the goal of maximizing trading volume, that is the probability of an accepted trade, irrespective of its magnitude.

There is also a significant literature focusing on bilateral trade from a pure efficiency viewpoint, starting with \cite{cesa2021bilateralEC,cesa2024bilateralMOR}, in the stochastic setting, similar to our work. 
Several extensions were then carried out, including the adversarial setting with smoothed adversaries~\citep{cesa2023smoothT34COLT,cesa2024smoothT34JMLR}, global budget balance~\citep{bernasconi2024globalSTOC}, adversarial notion of $\alpha$-regret~\citep{azar2022alphaNeurIPS}, limited information~\citep{bacchiocchi2025NeurIPS}, and a tight characterization of performance achievable by fixed-pricing~\citep{chen2026tightICALP}. 

Related to bilateral trade,
there are multiple related online learning problems stemming from economic models.
The brokerage problem was first addressed in~\cite{bolic2024onlineAAMAS,bachoc2025online} and then extended to the contextual setting, with parametric~\citep{bachoc2025parcnxtICML} and nonparametric~\citep{bachoc2025nonparcnxtAISTATS} structures. Another related model is the repeated mediated newsvendor one~\citep{bolic2025newsvendorNeurIPS}.
Adaptive social-welfare maximization~\citep{Cesa2025econometrica} considers a more abstract version of efficiency where the social-welfare is unknown and inferred from observations. Market making~\citep{cesa2025marketCOLT} studies the economic aspects of information elicitation and the price of information.

From a learning theory viewpoint, the techniques used in our work are related to several references.
Learning from incomplete observations has been considered in the budget-dependent zeroth-order convex optimization problem~\citep{bachoc2024zerothTMLR}, partial monitoring~\citep{bartok2014partial}, and continuum-armed bandits~\citep{bubeck2011xarmed}. 
Our problem of determining where to spend sampling budget is reminiscent of active learning~\citep{cohn1996active} and connects to the Horvitz--Thompson estimator~\citep{horvitz1952generalization} and adaptive importance sampling~\citep{he2024adaptive}. Our zooming algorithm and its successive region eliminations are related to  \cite{munos2011optimistic,de2012exponential,bouttier2020regret,bachoc2021sample,bachoc2021instance}.

\section{Main regret results under bounded densities}

\subsection{Representation lemma}
\label{section:representation:lemma}

The following lemma extends the representation lemma of
\cite[Lemma~2]{bachoc2026repeated} from $\lambda\leq0$ to $\lambda\in(0,1]$.

\begin{lemma}[Representation lemma]
\label{lem:main-reconstruction}
Fix $\lambda\in(0,1]$ and set
\[
H_\lambda(x,y)
=
\frac{1-\lambda}{4}
M_\lambda(x,y)^{1-2\lambda}
x^{\lambda-1}y^{\lambda-1},
\qquad x,y>0.
\]
Then
\[
M_\lambda(x,y)
=
2^{-1/\lambda}(x+y)
+
\int_0^x\int_0^y
H_\lambda(u,v)\dd v\dd u,
\qquad x,y\ge0,
\]
where the double integral on the right-hand side is finite. If in addition $S$ and $B$ are independent, then
\[
G_\lambda(p)
=
2^{1-1/\lambda}G_1(p)
+
\mathcal I_\lambda(p),
\]
where
\begin{equation}
\label{eq:Gun}
G_1(p)
=
\frac12
\left[
F_B^\circ(p)\int_0^pF_S(v)\dd v
+
F_S(p)\int_p^1F_B^\circ(w)\dd w
\right],
\end{equation}
and
\[
\mathcal I_\lambda(p)
=
\int_0^p\int_p^1
H_\lambda(p-v,w-p)
F_S(v)F_B^\circ(w)
\dd w\dd v,
\]
for any $p\in[0,1]$.
\end{lemma}
The proof of Lemma~\ref{lem:main-reconstruction} is postponed to Appendix~\ref{app:kernel-proof}.

\subsection{The case \texorpdfstring{$0<\lambda<1/2$}{0<lambda<1/2}}
\label{sec:hard-alg-main}

\subsubsection{Preliminaries}

For $0<\lambda<1/2$, we shall use a large constant $\CL$ that depends on $L$ and $\lambda$ but not $T$. We set
\begin{equation}
\CL
:=
\frac{2^{28}(1+L)^2}{\lambda^4},
\label{eq:CL:Def}
\end{equation}
and do not try to minimize its dependency on these two parameters. For $\Delta\in(0,1]$ and $\delta\in(0,1)$, we write
\[
\ell(\Delta,\delta)
:=
\log\left(\frac{64}{\Delta\delta}\right),
\qquad
q_\Delta(\delta)
:=
\left\lceil
\frac{\ell(\Delta,\delta)}{2\Delta^2}
\right\rceil.
\]

We use the following standard notation for a tree structure on the dyadic subsets of $[0,1]$. For $m\in\mathbb N\cup\{0\}$, write $
\Delta_m:=2^{-m}
$
and define the dyadic partition of $[0,1]$ to be
\[
\mathcal D_m
:=
\left\{
\left[\frac{k-1}{2^m},\frac{k}{2^m}\right):
k=1,\ldots,2^m-1
\right\}
\cup
\left\{
\left[\frac{2^m-1}{2^m},1\right]
\right\}.
\]
For $a<b$, if $I=[a,b]$ or $I=[a,b)$, we define its center to be $c(I):=(a+b)/2$. We define the two children of $[a,b]$ as
$
\left[a,\frac{a+b}{2}\right)
\quad\text{and}\quad
\left[\frac{a+b}{2},b\right]$
and the two children of $[a,b)$ as $
\left[a,\frac{a+b}{2}\right)
\quad\text{and}\quad
\left[\frac{a+b}{2},b\right)$.

\subsubsection{The algorithm}
\label{subsubsection:algorithm}
For simplicity, we only define our algorithm, see Algorithm~\ref{alg:main}, in the case where $T$ is large enough such that
\begin{equation}
\label{eq:T:large:enough}
\left(
\frac{4\CL\log(64T^5)}{T}
\right)^{1/3}
\le1,
\end{equation}
that is, $T$ is larger than some constant that may depend on $L$ and $\lambda$.

We divide the rounds into a sequence of epochs indexed by $m=0,1,\dots$. During each epoch, the algorithm keeps track of a collection of dyadic intervals
\[
\mathsf A_m\subseteq\mathcal D_m,
\]
the goal being that this collection contains those prices that are optimal at scale $\Delta_m$ for the expected objective $G_\lambda$, with high probability. We take
$
\mathsf A_0:=\{[0,1]\}$.
In the $m$-th epoch, the algorithm obtains an estimator for $G_\lambda(p)$ uniformly for all midpoints
\begin{equation}
\label{eq:Nm}
N_m:=\{c(I):I\in\mathsf A_m\}.
\end{equation}
This step (Line~\ref{line:estimation} of Algorithm~\ref{alg:main}) is achieved by implementing Algorithm~\ref{alg:estimate}\footnote{When $m=0$, Algorithm~\ref{alg:estimate} has
$
(\mathsf A_0,N_0,0,\delta_0,1)
$
as input and skips Lines~\ref{line:loop:rejected} to~\ref{sampling:cell:R}.}, which is described in Section~\ref{subsubsection:estimation}.
There $\delta_m$ is a confidence level.
During this step, Algorithm~\ref{alg:estimate} posts prices and observes the corresponding two-bit feedback. 
This estimation
enables the algorithm to identify a collection of dyadic intervals
\[
\mathsf R_m\subseteq\mathsf A_m
\]
whose prices are not optimal at scale $\Delta_{m+1}$, with high probability (Line~\ref{line:reject}).
These intervals are thus rejected at epoch $m$ (Line~\ref{line:reject:bis}).
The algorithm then splits each dyadic interval in $\mathsf A_m \backslash \mathsf R_m$  into its two dyadic children (hence its name, Zooming), which constitute $\mathsf A_{m+1}\subseteq\mathcal D_{m+1}$ in epoch $m+1$ (Line~\ref{line:reject:three}).

More precisely, in each epoch $m$, the algorithm maintains
\[
\mathsf R_0,\ldots,\mathsf R_{m-1},
\qquad
\mathsf R_j\subseteq\mathcal D_j,
\]
where $\mathsf R_0,\ldots,\mathsf R_{m-1}$ are the collections of dyadic intervals that were previously determined to be sub-optimal, and
$
\mathsf A_m\subseteq\mathcal D_m
$
is the collection of dyadic intervals still deemed relevant. We define
\[
A_m:=\bigcup_{I\in\mathsf A_m}I,
\qquad
R_j:=\bigcup_{I\in\mathsf R_j}I.
\]
Observe that
$
R_0,\ldots,R_{m-1},A_m
$
are pairwise disjoint and form a partition of $[0,1]$.
%
 The last epoch $M_T$ is a function of $\lambda,L,T$ and thus is known with certainty ahead of time, see Line~\ref{line:MT}. After $M_T$, the algorithm enters an exploitation stage. For the remaining times $t$ after epoch $M_T$ is completed, it chooses an arbitrary price in $A_{M_T+1}$ and posts it, see Lines~\ref{line:exploit:while} and~\ref{line:exploit}. By Lemma~\ref{lem:exploration-budget-regret}, the $M_T$-th epoch ends by time $T/2$.

\begin{algorithm}[!htb]
\caption{\textsc{Zooming algorithm for $0 < \lambda < \frac{1}{2}$}}
\label{alg:main}
\KwInput{Horizon $T\ge2$, order $0<\lambda<1/2$, density bound $L\ge1$ such that \eqref{eq:T:large:enough} holds.}
\KwOutput{Posted prices $P_1,\ldots,P_T$.}
Set
\[
        \overline\Delta_T
        \leftarrow
        \left(\frac{4\CL\log(64T^5)}{T}\right)^{1/3},
        \qquad t\leftarrow1.
\]

Let $M_T$ be the largest integer $m\ge0$ such that
$\Delta_m\ge\overline\Delta_T$\;
\label{line:MT}
Initialize $\mathsf A_0\leftarrow\{[0,1]\}$\;
\For{$m=0,1,\ldots,M_T$}{
        Set $\delta_m\leftarrow1/(T^2(m+1)^2)$ and
        $N_m\leftarrow\{c(I):I\in\mathsf A_m\}$\;
        \label{line:delta:m}
        Run Algorithm~\ref{alg:estimate} with
        $(\mathsf R_0,\ldots,\mathsf R_{m-1},\mathsf A_m,N_m,m,\delta_m,t)$,
        obtaining
        $(\widehat G_m(p))_{p\in N_m}$ and the updated counter $t$\;
            \label{line:estimation}
        Set $\widehat M_m\leftarrow\max_{p\in N_m}\widehat G_m(p)$\;
        Set
        \[
              \mathsf A_m^+
              \leftarrow
              \{I\in\mathsf A_m:
              \widehat G_m(c(I))
              \ge
              \widehat M_m-\CL\Delta_m\}.
        \]
        \label{line:reject}
        
        Set $\mathsf R_m\leftarrow\mathsf A_m\setminus\mathsf A_m^+$\;
        \label{line:reject:bis}
        Set $\mathsf A_{m+1}$ to be the collection of the two children of each
        $I\in\mathsf A_m^+$\;
        \label{line:reject:three}
        
}
\While{$t\le T$}{
\label{line:exploit:while}
        Post any price in $A_{M_T+1}$, observe the feedback, and set $t\leftarrow t+1$\;
        \label{line:exploit}
}
\end{algorithm}

\subsubsection{The estimation procedure}
\label{subsubsection:estimation}

Algorithm~\ref{alg:estimate} takes the rejected layers $\mathsf R_0,\ldots,\mathsf R_{m-1}$, the active dyadic intervals in $\mathsf A_m$, the active midpoint set $N_m$, the epoch $m$, the confidence parameter $\delta = \delta_m$, and the current round counter $t$ as inputs. Its goal is to construct estimates $(\widehat G_m(p))_{p\in N_m}$ such that, when the current partition is accurate,
\[
\sup_{p\in N_m}
\left|
\widehat G_m(p)-G_\lambda(p)
\right|
\le
\frac{\CL}{16}\Delta_m
\]
with high conditional probability (see Lemma \ref{lem:estimation} and \eqref{eq:hard-def-Bm}). 

Algorithm~\ref{alg:estimate} divides each interval in
$
\mathsf R_0,\ldots,\mathsf R_{m-1},\mathsf A_m
$
into $n_m$ equal sampling cells (Lines~\ref{line:Nm}, \ref{sampling:cell:A} and~\ref{sampling:cell:R}), where, with $\Delta=\Delta_m$,
\[
n_m
=
\left\lceil
\frac{\CL\ell(\Delta,\delta)}{16\Delta^2}
\right\rceil.
\]
Thus, $n_m$ depends on the confidence parameter $\delta$ through the logarithmic factor $\ell(\Delta,\delta)$. The procedure posts a single price uniformly distributed on each sampling cell and records the corresponding two-bit feedback (Line~\ref{line:unif:price}). It also posts each price $p\in N_m$ a total of $q_\Delta(\delta)$ times and records the corresponding two-bit feedback (Line~\ref{line:Ypa:Zpa}).

The feedback $(Y_J,Z_J)$ collected when posting $U_J$ uniformly on $J$ in Line~\ref{line:unif:price} can be used simultaneously in the estimation of $G_\lambda(p)$ for all $p\in N_m$. There is the same number $n_m$ of sampling cells per interval $I\in\mathsf R_0\cup\dots\cup\mathsf R_{m-1}\cup\mathsf A_m$, while the length of $I\in\mathsf R_j$ is $2^{-j}$ and the length of an interval in $\mathsf A_m$ is $2^{-m}$. Hence, the sampling density is higher in $R_j$ for larger $j$ and is highest in $A_m$. 
The prices in $R_j$ have a sub-optimality gap of order $\Delta_j$ for $G_\lambda$ and the prices in $A_m$ are optimal up to order $\Delta_m$ for $G_\lambda$, with high probability (see Lemma~\ref{lem:induction}). Thus, intuitively, Algorithm~\ref{alg:estimate} focuses its limited sampling budget on regions containing prices with larger values of $G_\lambda$.

\begin{algorithm}[!htb]
\caption{\textsc{Estimation procedure for $0 < \lambda < \frac{1}{2}$}}
\label{alg:estimate}
\KwInput{Rejected layers $\mathsf R_0,\ldots,\mathsf R_{m-1}$, active dyadic intervals
$\mathsf A_m$, active midpoint set $N_m$, epoch $m$, confidence
$\delta$, current round counter $t$.}
\KwOutput{Estimates $(\widehat G_m(p))_{p\in N_m}$ and updated counter $t$.}
Set
\[
        \Delta\leftarrow\Delta_m,
        \qquad
        n_m\leftarrow
        \left\lceil{\frac{\CL\ell(\Delta,\delta)}{16\Delta^2}}\right\rceil.
\]
Initialize $\cP_m\leftarrow\emptyset$\;
\label{line:Nm}
\ForEach{active dyadic interval $I\in\mathsf A_m$, with $|I|=\Delta_m$}{
        Split $I$ into $n_m$ equal cells and add them to $\cP_m$\;
}
\label{sampling:cell:A}
\For{$j=0,1,\ldots,m-1$}{
\label{line:loop:rejected}
        \ForEach{rejected dyadic interval $I\in\mathsf R_j$, with $|I|=\Delta_j$}{
                Split $I$ into $n_m$ equal cells and add them to $\cP_m$\;
        }
        \label{sampling:cell:R}
}
\ForEach{cell $J\in\cP_m$}{
       Draw $U_J\sim\mathrm{Unif}(J)$, post $P_{t}=U_J$, observe $(Y_{t},Z_{t})$, set $(Y_J,Z_J)\leftarrow(Y_{t},Z_{t})$, and set $t\leftarrow t+1$\;
        \label{line:unif:price}
}
Set $q\leftarrow q_\Delta(\delta)$\;
\ForEach{$p\in N_m$}{
        \For{$a=1,\ldots,q$}{
                Post $P_{t}=p$, observe $(Y_{t},Z_{t})$, set $Y_{p,a}\leftarrow Y_{t}$, set $Z_{p,a}\leftarrow Z_{t}$, and set $t\leftarrow t+1$\;
        \label{line:Ypa:Zpa}
        }
        Form $\widehat G_m(p)$ according to \eqref{eq:Ghat-hard-estimator}\;
}
\Return $(\widehat G_m(p))_{p\in N_m}$ and $t$\;
\label{lines:hatG}
\end{algorithm}

Given the collected data, the estimators $(\widehat G_m(p))_{p\in N_m}$ are returned at Line~\ref{lines:hatG}, and their computation is defined next. Fix $p\in N_m$. The starting point is the decomposition of $G_\lambda(p)$ in Lemma~\ref{lem:main-reconstruction}:
\[
G_\lambda(p)
=
2^{1-1/\lambda}G_1(p)
+
\mathcal I_\lambda(p).
\]

For the estimation of $G_1(p)$, we set
\begin{equation}
\label{eq:hatFS:FB}
\widehat F_S(p)
:=
\frac1q\sum_{a=1}^qY_{p,a},
\qquad
\widehat F_B^\circ(p)
:=
\frac1q\sum_{a=1}^qZ_{p,a},
\end{equation}
\begin{equation}
\label{eq:hatAS}
\widehat A^S(p)
:=
\sum_{J\in\cP_m}
|J|\mathbb{I}\{U_J\le p\}Y_J,
\end{equation}
\begin{equation}
\label{eq:hatAB}
\widehat A^B(p)
:=
\sum_{J\in\cP_m}
|J|\mathbb{I}\{U_J\ge p\}Z_J,
\end{equation}
and
\begin{equation}
\label{eq:hat:Gone}
\widehat G_1(p)
:=
\frac12
\left(
\widehat F_B^\circ(p)\widehat A^S(p)
+
\widehat F_S(p)\widehat A^B(p)
\right).
\end{equation}

Note that \eqref{eq:hat:Gone} matches \eqref{eq:Gun} in Lemma~\ref{lem:main-reconstruction}, that $Y_{p,a},Z_{p,a}$ are collected at Line~\ref{line:Ypa:Zpa}, and that $U_J,Y_J,Z_J$ are collected at Line~\ref{line:unif:price}.

Let us now turn to the estimation of $\mathcal I_\lambda(p)$. In Lemma~\ref{lem:centered-decomposition}, we decompose it in the form
\begin{equation}
\label{eq:deco:I:lambda}
\mathcal I_\lambda(p)
=
F_S(p)F_B^\circ(p)J(p)
+
F_B^\circ(p)L^S(p)
+
F_S(p)L^B(p)
+
C(p),
\end{equation}
where
\begin{equation}
\label{eq:Jp}
J(p)
:=
\int_0^p\int_p^1
H_\lambda(p-v,w-p)\,\dd w\,\dd v
\end{equation}
and $L^S(p)$, $L^B(p)$, and $C(p)$ are defined in Lemma~\ref{lem:centered-decomposition}.

Then in Lemma~\ref{lem:truncation}, we show that $L^S(p)$, $L^B(p)$, and $C(p)$ are well-approximated by quantities written $L^S_\Delta(p)$, $L^B_\Delta(p)$, and $C_\Delta(p)$, for which we provide the following estimators. We set
\begin{equation}
\label{eq:KpS}
K_p^S(u)
:=
\int_p^1H_\lambda(u,w-p)\,\dd w,
\qquad
0<u\le p,
\end{equation}
\begin{equation}
\label{eq:KpB}
K_p^B(z)
:=
\int_0^pH_\lambda(p-v,z)\,\dd v,
\qquad
0<z\le1-p.
\end{equation}
We then define
\begin{equation}
\widehat L_\Delta^S(p)
:=
\begin{cases}
\displaystyle
\sum_{J\in\cP_m}
|J|\mathbb{I}\!\left\{U_J\le p-\Delta^{\frac{1}{1+\lambda}}\right\}
\left(Y_J-\widehat F_S(p)\right)
K_p^S(p-U_J),
&\text{if }p\ge\Delta^{\frac{1}{1+\lambda}},\smallskip\\
0,
&\text{if }p<\Delta^{\frac{1}{1+\lambda}}
\end{cases}
\label{eq:hard-def-LShat}
\end{equation}
and
\begin{equation}
\widehat L_\Delta^B(p)
:=
\begin{cases}
\displaystyle
\sum_{J\in\cP_m}
|J|\mathbb{I}\!\left\{U_J\ge p+\Delta^{\frac{1}{1+\lambda}}\right\}
\left(Z_J-\widehat F_B^\circ(p)\right)
K_p^B(U_J-p),
&\text{if }p\le1-\Delta^{\frac{1}{1+\lambda}},\smallskip\\
0,
&\text{if }p>1-\Delta^{\frac{1}{1+\lambda}}.
\end{cases}
\label{eq:hard-def-LBhat}
\end{equation}
We set $\widehat C_\Delta(p):=0$ unless
$
\Delta^{\frac{1}{1+\lambda}}\le\frac12
\quad\text{and}\quad
\Delta^{\frac{1}{1+\lambda}}
\le p
\le
1-\Delta^{\frac{1}{1+\lambda}}.
$
If these two conditions hold, we set
\begin{equation}
\begin{aligned}
\widehat C_\Delta(p)
:=
\sum_{J,K\in\cP_m}
&|J||K|
\mathbb{I}\!\left\{U_J\le p-\Delta^{\frac{1}{1+\lambda}}\right\}
\mathbb{I}\!\left\{U_K\ge p+\Delta^{\frac{1}{1+\lambda}}\right\}\\
&\times
\left(Y_J-\widehat F_S(p)\right)
\left(Z_K-\widehat F_B^\circ(p)\right)
H_\lambda(p-U_J,U_K-p).
\end{aligned}
\label{eq:hard-def-Chat}
\end{equation}

Then, we set the estimator counterpart to \eqref{eq:deco:I:lambda},
\[
\widehat{\mathcal I}_{\lambda,\Delta}(p)
:=
\widehat F_S(p)\widehat F_B^\circ(p)J(p)
+
\widehat F_B^\circ(p)\widehat L_\Delta^S(p)
+
\widehat F_S(p)\widehat L_\Delta^B(p)
+
\widehat C_\Delta(p).
\]
Finally, we set the estimator counterpart to
$G_\lambda(p)=2^{1-1/\lambda}G_1(p)+\mathcal I_\lambda(p)$,
\begin{equation}
\widehat G_m(p)
:=
2^{1-1/\lambda}\widehat G_1(p)
+
\widehat{\mathcal I}_{\lambda,\Delta}(p).
\label{eq:Ghat-hard-estimator}
\end{equation}

\subsubsection{Regret upper bound}

Recall that $\mathcal D_L$ consists of product distributions under which $S$ and $B$ have probability density functions bounded by $L$.
\begin{theorem}[Regret upper bound, $0 < \lambda < \frac{1}{2}$; proved in Appendix~\ref{app:hard-upper}]
\label{thm:hard-main-statement}
Fix $0<\lambda<1/2$ and $L\ge1$.
Consider an unknown distribution $\nu \in \mathcal D_L$.
Then, when the algorithm $\alpha$ is Algorithm~\ref{alg:main}, we have
\[
        R_\lambda^\nu(\alpha,T)\le
        \left[
        \frac{2^{19}}{3}\CL^2(4\CL)^{-2/3}
        +3\CL(4\CL)^{1/3}
        \right]
        T^{2/3}\bigl(\log(64T^5)\bigr)^{1/3}+\frac{\pi^2}{6T},
\]
whenever $((4\CL\log(64T^5))/T)^{1/3}\le1$, where $\CL$ is defined in \eqref{eq:CL:Def}. In particular,
\[
        R_\lambda^\nu(\alpha,T)=\widetilde O_{\lambda,L}(T^{2/3}).
\]
\end{theorem}

\subsection{The case \texorpdfstring{$1/2 \le \lambda < 1$}{1/2 <= lambda < 1}}
\label{subsection:lambda:half:one}

For $1/2 \le \lambda < 1$, the singularity of the kernel $H_{\lambda}$ in Lemma~\ref{lem:main-reconstruction}, that is driven by the quantities $x^{\lambda-1}y^{\lambda-1}$ as $x \to 0, y \to 0$, becomes sufficiently mild that
a more direct adaptation of the scouting bandit algorithm of \cite{cesa2024bilateralMOR} (suggested for $\lambda = 1$) can be carried out.
This adaptation is provided in Algorithm~\ref{alg:easy-scout}.
The idea is to divide the times $t=1,\dots,T$ into two phases, a scouting phase and a bandit phase (the latter starting at Line~\ref{line:start:bandit}). In the scouting phase of length
 $
n:=\left\lceil T^{2/3}\log(eT)\right\rceil$,
the mechanism samples one price $U_J$ uniformly in each cell $J$ of the regular partition
\begin{equation}
        \mathcal P_n
        :=
        \left\{
        \left[\frac{j-1}{n},\frac{j}{n}\right):
        j=1,\ldots,n-1
        \right\}
        \cup
        \left\{
        \left[\frac{n-1}{n},1\right]
        \right\}.
\label{eq:easy-partition}
\end{equation}
It records the resulting seller-side bit $Y_J$ and buyer-side bit $Z_J$ (Line~\ref{line:scouting:n:feedback}). These observations can be used to estimate
$
\int_0^p F_S(v)\dd v$,
$\int_p^1 F_B^\circ(w)\dd w$,
and
$\mathcal I_\lambda(p)$ (recall Lemma~\ref{lem:main-reconstruction}),
for prices $p\in\mathcal G_K$, where
\begin{equation}
        \mathcal G_K
        :=
        \left\{\frac{k}{K+1}:k=1,\ldots,K\right\},
\label{eq:easy-grid}
\end{equation}
is the price grid with $K:=\left\lceil T^{1/3}\right\rceil$. For every $q\in\mathcal G_K$, the first two integrals are estimated by
\begin{equation}
        \widehat A^S(q)
        :=
        \sum_{J\in\mathcal P_n}
        |J|\mathbb I\{U_J\le q\}Y_J
\label{eq:easy-AShat}
\end{equation}
and
\begin{equation}
        \widehat A^B(q)
        :=
        \sum_{J\in\mathcal P_n}
        |J|\mathbb I\{U_J\ge q\}Z_J.
\label{eq:easy-ABhat}
\end{equation}

The more challenging quantity to estimate is $\mathcal I_\lambda(q)$, which is estimated as follows. Let $1/n$ denote the gap used for the truncation. Then, let
\begin{equation}
\widehat{\mathcal I}_{\lambda,n}(q)
:=
\begin{cases}
\displaystyle
\sum_{J,J'\in\mathcal P_n}
|J||J'|
\mathbb I\left\{U_J\le q-\frac1n\right\}
\mathbb I\left\{U_{J'}\ge q+\frac1n\right\}
Y_JZ_{J'}
\\
\displaystyle\hspace{4cm}\times
H_\lambda(q-U_J,U_{J'}-q),
&
\text{if }\dfrac1n\le q\le1-\dfrac1n,
\smallskip\\
0,
&
\text{otherwise,}
\end{cases}
\label{eq:easy-Ihat}
\end{equation}
The $1/n$ truncation gap introduces a controlled bias in exchange for being able to control the estimator better, as shown in Lemma~\ref{lem:easy-partition-bounds}. This allows us to manage the singularity of $H_\lambda$ at zero. We also let
\begin{equation}
        \widetilde{\mathcal I}_{\lambda,n}(q)
        :=
        \min\left\{
        1,
        \widehat{\mathcal I}_{\lambda,n}(q)
        \right\}.
\label{eq:easy-Ihat-projected}
\end{equation}

In the bandit phase we apply an off-the-shelf bandit algorithm with approximate expected rewards
\[
2^{1-1/\lambda}\,
\frac12
\left[
F_B^\circ(q)\widehat A^S(q)
+
F_S(q)\widehat A^B(q)
\right]
+
\widetilde{\mathcal I}_{\lambda,n}(q),
\]
and use the estimator
\begin{equation}
\begin{aligned}
        \widehat G_{\lambda,t}(q)
        &:=
        2^{1-1/\lambda}\,
        \frac12
        \left[
        Z_t\widehat A^S(q)
        +
        Y_t\widehat A^B(q)
        \right]
        +
        \widetilde{\mathcal I}_{\lambda,n}(q)
\end{aligned}
\label{eq:easy-Ghat}
\end{equation}
as bandit feedback, when posting price $q$ at time $t$. Its conditional expectation (given the scouting phase) equals the approximate reward.

Specifically, we apply the following algorithm $\beta_K$ for $K$ arms: Algorithm $\beta_K$ (Line~\ref{line:apply:bandit}) is the well-known Exp3 algorithm, applied to rewards divided by $2$ and tuned as in~\cite[Corollary~3.2]{auer2002nonstochastic} with $g=T-n$. After $T-n$ rounds, Algorithm $\beta_K$ has expected regret of at most
\begin{equation}
        6\sqrt{K (T-n) \log(K)}.
\label{eq:easy-bandit-guarantee}
\end{equation}
Note that our rewards in \eqref{eq:easy-Ghat} are known to lie in $[0,2]$, so we apply $\beta_K$ to scaled rewards in $[0,1]$.

\begin{algorithm}
\caption{\textsc{Scouting bandits for $1/2\le\lambda < 1$}}
\label{alg:easy-scout}
\KwInput{Horizon $T$
(large enough such that  $n :=\left\lceil T^{2/3}\log(eT)\right\rceil < T$)
, order $\lambda\in[1/2,1)$, density
bound $L \ge 1$.}
\KwOutput{Posted prices $P_1,\ldots,P_T$.}
Form
$n,\mathcal P_n$
and
$K, \mathcal G_K$
according to
\eqref{eq:easy-partition}
and
\eqref{eq:easy-grid}\;
Set $t\leftarrow1$\;
\ForEach{$J\in\mathcal P_n$}{
\label{line:scouting:n}
        draw $U_J\sim{\mathrm{Unif}}(J)$, post $P_t=U_J$, observe
        $(Y_t,Z_t)$, set $(Y_J,Z_J)\leftarrow(Y_t,Z_t)$, and set
        $t\leftarrow t+1$\;
        \label{line:scouting:n:feedback}
}
For every $q\in\mathcal G_K$, compute
$\widehat A^S(q)$, $\widehat A^B(q)$, and
$\widetilde{\mathcal I}_{\lambda,n}(q)$ according to
\eqref{eq:easy-AShat}, \eqref{eq:easy-ABhat},
\eqref{eq:easy-Ihat} and \eqref{eq:easy-Ihat-projected}\;
Initialize the $K$-armed bandit algorithm $\beta_K$\;
\For{$t=n+1,\ldots,T$}{
\label{line:start:bandit}
        obtain an arm $I_t\in\{1,\ldots,K\}$ from $\beta_K$, and
        set $q_t\leftarrow I_t/(K+1)$\;
        \label{line:apply:bandit}
        post $P_t=q_t$ and observe $(Y_t,Z_t)$\;
        \label{line:bandit:feedback}
        compute $\widehat G_{\lambda,t}(q_t)$ according to
        \eqref{eq:easy-Ghat}, and give this reward to $\beta_K$\;
}
\end{algorithm}

\begin{theorem}[Regret upper bound, $1/2\le\lambda < 1$;
proved in Appendix~\ref{app:easy-upper}]
\label{thm:easy-upper}
Fix $1/2\le\lambda < 1$ and $L\ge1$, and consider an unknown
distribution $\nu\in\mathcal D_L$. Let $\alpha$ be
Algorithm~\ref{alg:easy-scout}. If $n < T$, then
\begin{equation}
        R_\lambda^\nu(\alpha,T)
        \le
        2^8(1+L)T^{2/3}\log(eT).
\label{eq:easy-regret-bound}
\end{equation}
In particular,
\[
        R_\lambda^\nu(\alpha,T)
        =
        \widetilde O_L(T^{2/3})
\]
uniformly over $\lambda\in[1/2,1)$.
\end{theorem}

\begin{remark}[Orders $\lambda\ge1$]
The scouting arguments above can be extended to any fixed $\lambda\ge1$ (for $\lambda>1$, we use the same formula for $M_\lambda$ and the same definitions of the objective and regret as in Section~\ref{sec:model-objectives}). The case $\lambda=1$ is considered in detail in~\cite{cesa2024bilateralMOR}.
The case $\lambda > 1$ rewards unequal splits of surplus. More precisely, for any $G>0$, $[0,G] \ni x \mapsto M_{\lambda} (x,G-x)$ is maximal at $0$ and $G$. Hence the case $\lambda > 1$ is not relevant to fairness. 
\end{remark}

\subsection{Lower bound under bounded densities for \texorpdfstring{$\lambda\in[0,1]$}{lambda in [0,1]}}
\label{section:LB:bounded:dens}

Next, we provide a lower bound of order $T^{2/3}$, under bounded density, establishing the optimality of our upper bounds. 

\begin{theorem}[Lower bound under bounded densities; proved in Appendix~\ref{app:smooth-lower}]\label{thm:smooth-lower-main}
For every $L\ge20$, every $\lambda\in[0,1]$, every $T\ge32$, and every randomized posted-price algorithm $\alpha$, there exists a distribution $\nu\in\mathcal D_L$ such that
\[
        R_\lambda^\nu(\alpha,T)\ge \frac{1}{12800} T^{2/3}.
\]
\end{theorem}

\section{Linear lower bound without bounded density}
\label{section:linear:lower}

In \cite{bachoc2026repeated}, for $\lambda \leq 0$, upper bounds on the regret of order $\widetilde O(T^{2/3})$ are proved without assuming a bounded density of $S$ and $B$ (that remain independent). 
For $\lambda \in (0,1]$, we show next that the situation is fundamentally different, since a linear worst case lower bound on the regret holds, when the bounded density assumption is removed. The reason for this difference is that for $\lambda \le 0$ and $s,b \in [0,1]$, the function $[0,1] \ni p \mapsto \gft_\lambda(p,s,b)$ is continuous, while for $\lambda \in (0,1]$, this function can be discontinuous at $s$ and $b$. This discontinuity makes it possible to create a ``needle in a haystack'' situation, extending \cite{cesa2024bilateralMOR} that addressed the case $\lambda = 1$. 
More precisely, we can construct instances where the expected gain from trade has an isolated maximum, that is impossible to find in finite time, under binary feedback.

\begin{theorem}[Linear regret without bounded density; proved in Appendix~\ref{app:no-density}]\label{thm:linear-arbitrary}
For every fixed $\lambda \in (0,1]$, there is an explicit constant $c_\lambda>0$ such that, for every horizon $T$ and every possibly randomized posted-price algorithm $\alpha$, there exists a distribution $\nu$ for $(S,B)$ (under which they are independent) such that
\[
        R_\lambda^\nu(\alpha,T)\ge c_\lambda T.
\]
One may take $c_\lambda=\gamma_\lambda$, where $\gamma_\lambda$ is defined in Appendix~\ref{app:no-density}.
\end{theorem}

\section{Conclusions}

We studied bilateral trade with the goal of maximizing the expected $\lambda$-fair gain from trade, with the range $\lambda \in (0,1]$, that corresponds to an interpolation between fairness and efficiency. Our economic motivation for this objective is supported by an axiomatic characterization. From an online learning point of view, the bounded-density assumption is not merely a technical convenience: it cannot simply be dropped from our sublinear-regret guarantees. Unlike the case $\lambda\le0$, allowing arbitrary independent seller and buyer distributions makes linear worst-case regret unavoidable for every fixed $\lambda>0$, as our ``needle in a haystack'' construction shows. For $\lambda \in (0,1/2)$, the kernel appearing in our reconstruction lemma is strongly singular. This causes a technical challenge, that we solve thanks to two main original arguments. First, we introduce a Zooming algorithm, where we sample exploration prices non-uniformly, to cover more densely the regions that are currently estimated as more promising. Second, we use an involved decomposition of the double integral appearing in our reconstruction lemma, in order to reduce the bias caused by the truncation step. In the end, we obtain matching lower and upper bounds on the regret of order $\widetilde{\Theta}(T^{2/3})$.

\newpage

\section*{AI Disclosure}


AI (ChatGPT 5 and 6) was used as a support tool in tasks such as brainstorming ideas, polishing the writing, and translating notes into \LaTeX. The authors personally checked the entirety of the final manuscript and take full responsibility for it.

\appendix

\section{Axiomatic characterization of the relaxed fairness family}
\label{app:axioms}

We start from a similar axiomatization as Appendix~A of \cite{bachoc2026repeated}: the scalar fair-gain score $\phi:\mathbb R_+^2\to\mathbb R_+$ is assumed to satisfy a common-value scale (CWS) and scale covariance (SC). Replacing the zero-fairness axiom for one-takes-all splits, ZF(OTI), by Pigou-Dalton inequality aversion, we want to characterize the H\"older means of finite orders $\lambda\le1$.

\begin{assumption}[Common-value scale]
\label{ax:common-value-scale}
There exists a strictly increasing continuous function
$u:\mathbb R_{++}\to\mathbb R$ such that, for every pair of positive gains
$x,y>0$,
\[
        \phi(x,y)=
        u^{-1}\left(\frac{u(x)+u(y)}{2}\right).
\]
Furthermore, $\phi$ is continuous on the whole quadrant $\mathbb R_+^2$.
\end{assumption}

\begin{assumption}[Scale covariance]
\label{ax:scale-covariance}
For every $t>0$ and every $x,y\ge0$,
\[
        \phi(tx,ty)=t\phi(x,y).
\]
\end{assumption}

For completeness, we spell out the axiom of Pigou-Dalton transfers:

\begin{assumption}[Pigou-Dalton inequality aversion]
\label{ax:pigou-dalton}
For every total surplus $z>0$, every $x\in[0,z/2]$, and every
$\eta\in[0,z/2-x]$,
\[
        \phi(x+\eta,z-x-\eta)\ge \phi(x,z-x).
\]
In words, if we hold the total surplus fixed, then transferring some surplus from the ``better-off'' to the ``worse-off'' side does not decrease the score until the two sides cross.
\end{assumption}

\begin{theorem}[Relaxed fairness characterization]\label{thm:axiom-relaxed-holder}
Suppose that $\phi:\mathbb R_+^2\to\mathbb R_+$ satisfies
Assumptions~\ref{ax:common-value-scale} and~\ref{ax:scale-covariance}.  Then
$\phi$ satisfies Assumption~\ref{ax:pigou-dalton} if and only if
\[
        \phi=M_\lambda
        \qquad\text{for some finite }\lambda\le1.
\]
If pointwise limits of such scores are included, the closed characterized
family is
\[
        \{M_\lambda:\lambda\in[-\infty,1]\},
        \qquad M_{-\infty}(x,y)=\min\{x,y\}.
\]
Moreover, if Assumption~\ref{ax:pigou-dalton} is replaced by the stronger
one-takes-it-all zero-fairness axiom
\[
        \phi(1,0)=\phi(0,1)=0,
\]
then $\phi$ must be a H\"older mean with a finite order $\lambda\le0$, and every such mean satisfies the three axioms. We can also include the Rawlsian score $M_{-\infty}(x,y)=\min\{x,y\}$ as the pointwise limit of these means as $\lambda\to-\infty$. With this limiting case included, the family is indexed by $\lambda\in[-\infty,0]$.
\end{theorem}

\begin{proof}
We use the characterization step from Appendix~A of \cite{bachoc2026repeated} as a black box: the common-value-scale axiom
makes $\phi$ a quasi-arithmetic mean on $\mathbb R_{++}^2$, and scale
covariance forces this quasi-arithmetic mean to be homogeneous.  The standard
Hardy-Littlewood-Polya characterization of homogeneous quasi-arithmetic
means then implies that, on $\mathbb R_{++}^2$,
\[
        \phi(x,y)=M_\lambda(x,y)
\]
for some finite $\lambda\in\mathbb R$, where the logarithmic generator
corresponds to $\lambda=0$.  Thus, since we are assuming continuity on the whole quadrant, it also holds that $\phi(x,y)=M_\lambda(x,y)$ when one or both of $x,y$ are zero.

It remains only to identify which H\"older orders satisfy
Pigou-Dalton.  Fix $z>0$ and set
\[
        f_\lambda(x)=M_\lambda(x,z-x),
        \qquad x\in[0,z].
\]
For a finite order $\lambda\ne0$, we differentiate $f_\lambda(x)=M_\lambda(x,z-x)$ with respect to $x$, keeping the total surplus $z$ fixed. Increasing $x$ therefore increases the first gain and decreases the second by the same amount. For $0<x<z$, both gains are positive, and differentiating the formula for $M_\lambda$ gives
\[
        f_\lambda'(x)
        =\frac{1}{2}M_\lambda(x,z-x)^{1-\lambda}
        \bigl(x^{\lambda-1}-(z-x)^{\lambda-1}\bigr).
\]
For $\lambda=0$, $f_0(x)=\sqrt{x(z-x)}$.  Hence $f_0$ increases on
$[0,z/2]$.  For $\lambda<1$ with $\lambda\ne0$, the exponent $\lambda-1$ is negative. Thus, whenever $0<x<z/2$, we have $x<z-x$ and hence $x^{\lambda-1}>(z-x)^{\lambda-1}$. The derivative above is therefore positive on $(0,z/2)$, so $f_\lambda$ increases as the split becomes more equal. We have so far considered splits with $0<x<z/2$. At $x=0$, one participant receives nothing, while at $x=z/2$, the two gains are equal. Since $f_\lambda$ is continuous, its values at these endpoints are the limits of its values inside the interval. Its strict increase therefore also holds for transfers that start with a zero gain or end with equal gains. Thus, for these orders, transferring any positive amount from the participant with the larger gain to the one with the smaller gain increases the score, provided the transfer stops at or before equality.

For $\lambda=1$, the score is simply half the total surplus: $f_1(x)=(x+(z-x))/2=z/2$. It is therefore the same for every way of dividing $z$ between the two participants, including the equal split. Pigou-Dalton allows this: it requires that a transfer towards equality not lower the score, but it does not insist that the score increase. For $\lambda>1$, however, the derivative is negative on $(0,z/2)$. A transfer towards equality then decreases the score, contrary to Pigou-Dalton. Together with the case $\lambda=0$ considered above, this shows that the axiom holds precisely for the finite orders $\lambda\le1$.

The pointwise limit $M_\lambda\to M_{-\infty}=\min\{x,y\}$ as
$\lambda\to-\infty$ clearly satisfies Pigou-Dalton, giving the closed family
$[-\infty,1]$.  Finally,
\[
        M_\lambda(1,0)=0\quad\text{for }\lambda\le0,
        \qquad
        M_\lambda(1,0)=2^{-1/\lambda}>0\quad\text{for }\lambda>0.
\]
As above, ZF(OTI) singles out $\lambda\le0$, with Rawlsian $\lambda=-\infty$ as a limit.
\end{proof}

\section{Proof of Lemma~\ref{lem:main-reconstruction}}
\label{app:kernel-proof}

\begin{proof}[Proof of Lemma~\ref{lem:main-reconstruction}]
Fix $\lambda\in(0,1]$. For $x,y>0$, we differentiate $M_\lambda(x,y)$ once with respect to $x$ and once with respect to $y$, obtaining
\[
        \partial_{xy}^2M_\lambda(x,y)
        =
        \frac{1-\lambda}{4}
        M_\lambda(x,y)^{1-2\lambda}x^{\lambda-1}y^{\lambda-1}
        =H_\lambda(x,y).
\]
Moreover,
\[
        M_\lambda(x,0)=2^{-1/\lambda} x,
        \qquad
        M_\lambda(0,y)=2^{-1/\lambda} y,
        \qquad M_\lambda(0,0)=0.
\]
For $\varepsilon\in(0,\min\{x,y\})$, the ordinary fundamental theorem of calculus on $[\varepsilon,x]\times[\varepsilon,y]$ gives
\[
\begin{split}
        \int_\varepsilon^x\int_\varepsilon^yH_\lambda(u,v)\dd v\dd u
        &=
        M_\lambda(x,y)-M_\lambda(\varepsilon,y)
        -M_\lambda(x,\varepsilon)+M_\lambda(\varepsilon,\varepsilon).
\end{split}
\]
We now let $\varepsilon$ decrease to zero. The integration rectangles expand, and since $H_\lambda$ is nonnegative, the integrals on the left converge to the integral with both lower limits equal to zero. On the right, continuity of $M_\lambda$ gives $M_\lambda(\varepsilon,y)\to M_\lambda(0,y)$, $M_\lambda(x,\varepsilon)\to M_\lambda(x,0)$, and $M_\lambda(\varepsilon,\varepsilon)\to0$. Substituting the values on the axes computed above, we obtain
\[
        \int_0^x\int_0^yH_\lambda(u,v)\dd v\dd u
        =
        M_\lambda(x,y)-2^{-1/\lambda} x-2^{-1/\lambda} y.
\]
Although $H_\lambda$ can become unbounded near the axes, the preceding calculation shows that its integral is finite: the truncated integrals converge to the finite expression on the right-hand side. Rearranging this equality gives the claimed formula for $M_\lambda(x,y)$ when $x,y>0$. It remains to check the cases where one or both gains are zero. If $y=0$, the inner integral has equal lower and upper limits, so the double integral is zero. The claimed formula then reduces to $M_\lambda(x,0)=2^{-1/\lambda}x$, which we have already established. If $x=0$, the outer integration interval has length zero, and the formula similarly reduces to $M_\lambda(0,y)=2^{-1/\lambda}y$. When both gains are zero, both sides are zero. Thus the formula holds for all $x,y\ge0$.

To derive the reconstruction formula, compute the expectation over the event of trade. Let $p\in[0,1]$. On this event, put $x=p-S$ and $y=B-p$. The boundary term in the pointwise identity gives
\[
        2^{-1/\lambda}\E\!\left[\mathbb{I}\{S\le p\le B\}(B-S)\right]
        =2^{1-1/\lambda} G_1(p),
\]
because our notation $G_1$ uses the H\"older mean $M_1(x,y)=(x+y)/2$, i.e.,
one half of the classical gain from trade. Taking expectations under independence gives
\[
\begin{split}
        \E[(p-S)\mathbb{I}\{S\le p\le B\}]
        &=F_B^\circ(p)\E[(p-S)\mathbb{I}\{S\le p\}]\\
        &=F_B^\circ(p)\int_0^pF_S(v)\dd v,
\end{split}
\]
and similarly
\[
        \E[(B-p)\mathbb{I}\{S\le p\le B\}]
        =
        F_S(p)\int_p^1F_B^\circ(w)\dd w.
\]
This gives~\eqref{eq:Gun}.

We now take the expectation of the double-integral term. We first apply Fubini's theorem with the integration variables bounded away from zero, which allows us to interchange expectation and integration. Since the integrand is nonnegative, we can then let the truncation tend to zero by monotone convergence. Using the independence of $S$ and $B$ to factor the probability of the two threshold conditions, we obtain
\[
\begin{split}
&\E\left[
\mathbb{I}\{S\le p\le B\}
\int_0^{p-S}\int_0^{B-p}H_\lambda(u,z)\dd z\dd u
\right]\\
&\qquad =
\int_0^p\int_0^{1-p}
H_\lambda(u,z)\Prob(S\le p-u)\Prob(B\ge p+z)\dd z\dd u\\
&\qquad =
\int_0^p\int_p^1
H_\lambda(p-v,w-p)F_S(v)F_B^\circ(w)\dd w\dd v.
\end{split}
\]
The last double integral is exactly $\mathcal I_\lambda(p)$, while the contribution from the values of $M_\lambda$ on the axes was $2^{1-1/\lambda}G_1(p)$. Adding these two terms gives the claimed expression for $G_\lambda(p)$ and completes the proof.
\end{proof}

\section{Proofs for Section \ref{sec:hard-alg-main}}\label{app:hard-upper}

In this appendix, we provide the missing details in Section~\ref{sec:hard-alg-main}. More precisely, we proceed in the following manner. First, Section~\ref{subsection:basic:general} is dedicated to some properties about the expected gain from trade $G_\lambda$: we will decompose and approximate it, see Lemmas~\ref{lem:centered-decomposition} and~\ref{lem:truncation}. Second, we define an event on which the estimation procedure in Algorithm~\ref{alg:estimate} provides a good enough estimation in all epochs of Algorithm~\ref{alg:main} (event $\mathcal E_{M_T}$, see \eqref{eq:hard-def-Em}), and show that $\mathcal E_{M_T}$ occurs with high probability, see Lemma~\ref{lem:induction}. Third, under $\mathcal E_{M_T}$, we give a deterministic bound on the realized (pseudo-) regret suffered by Algorithm~\ref{alg:main}, see Lemma~\ref{lem:total-good-regret}. Finally, the proof of Theorem~\ref{thm:hard-main-statement} will follow easily from the preceding ingredients.

\subsection{Basic properties, decomposition and approximation of  \texorpdfstring{$G_{\lambda}$}{G\_lambda}}
\label{subsection:basic:general}

Throughout this section, recall
\[
        0<\lambda<\frac12,
        \qquad
        L\ge1,
        \qquad
        \Delta_m=2^{-m},
        \qquad
        G_\lambda^\star=\max_{p\in[0,1]}G_\lambda(p).
\]

The next lemma provides some convenient inequalities related to $\CL$, $\ell(\Delta,\delta)$ and $q_\Delta(\delta)$.

\begin{lemma}[Consequences of the choice of $\CL$]
\label{lem:constant-domination}
For $\CL$ defined in \eqref{eq:CL:Def},
\begin{equation}
\begin{gathered}
        \CL\ge2^{34},
        \qquad
        64\CL^{1/2}\le \frac{\CL}{16},
        \qquad
        L\left(2+\frac{1}{\lambda}\right)\le \frac{\CL}{16},
        \\
        \frac{2}{\lambda}\le\CL^{1/2},
        \qquad
        \frac{4}{\lambda^2}\le\CL^{1/2},
        \qquad
        \frac{4L}{\lambda(1+\lambda)}\le\CL^{1/2},
        \\
        16\left(2+\frac{1}{2\lambda}\right)
        \le\CL^{1/2}.
\end{gathered}
\label{eq:hard-constant-bounds}
\end{equation}
Moreover, for every $\Delta\in(0,1]$ and $\delta\in(0,1)$,
\begin{equation}
        q_\Delta(\delta)
        \le\Delta^{-2}\ell(\Delta,\delta).
\label{eq:hard-rq-upper-bounds}
\end{equation}
\end{lemma}

\begin{proof}
Since $1+L\ge2$ and $\lambda^{-1}>2$,
\[
\begin{aligned}
        \CL
        &=
        \frac{2^{28}(1+L)^2}{\lambda^4}
        \ge
        2^{28}2^2 2^4
        =
        2^{34},                                                    \\
        \CL^{1/2}
        &
        \ge
        2^{17}.
\end{aligned}
\]
In particular,
\[
        64\CL^{1/2}
        \le \frac{\CL}{16}.
\]

Moreover, since $L\le1+L$, $1+\lambda\ge1$, and
$2+1/\lambda\le2/\lambda$,
\[
\begin{aligned}
        L\left(2+{\frac{1}{\lambda}}\right)
        &\le
        {\frac{2(1+L)}{\lambda}}
        \le
        {\frac{\CL}{16}},                                             \\
        {\frac{2}{\lambda}}
        &\le
        \CL^{1/2},                                                 \\
        {\frac{4}{\lambda^2}}
        &\le
        \CL^{1/2},                                                 \\
        {\frac{4L}{\lambda(1+\lambda)}}
        &\le
        {\frac{4(1+L)}{\lambda}}
        \le
        \CL^{1/2},                                                 \\
        16\left(
        2+{\frac{1}{2\lambda}}
        \right)
        &\le
        {\frac{24}{\lambda}}
        \le
        \CL^{1/2}.
\end{aligned}
\]
This proves \eqref{eq:hard-constant-bounds}.

Finally, since
$\ell(\Delta,\delta)\ge\log64$ and $\Delta\le1$,
\[
        {\frac{\ell(\Delta,\delta)}{2\Delta^2}}
        \ge
        {\frac{\log64}{2}}
        \ge1.
\]
Therefore,
\[
\begin{aligned}
        q_\Delta(\delta)
        &\le
        2{\frac{\ell(\Delta,\delta)}{2\Delta^2}}
        =
        \Delta^{-2}\ell(\Delta,\delta).
\end{aligned}
\]
\end{proof}

The next lemma provides some inequalities for the kernel $H_{\lambda}$, defined in Lemma \ref{lem:main-reconstruction}. These inequalities will be used repeatedly in the sequel. 

\begin{lemma}[Kernel bounds]
\label{lem:hard-kernel-bounds}
For all $u,z\in(0,1]$,
\begin{equation}
        0\le H_\lambda(u,z)
        \le u^{\lambda-1}z^{\lambda-1},
\label{eq:hard-kernel-pointwise}
\end{equation}
\begin{equation}
        \int_0^1H_\lambda(u,z)\,\dd z
        \le {\frac{1}{\lambda}}u^{\lambda-1},
        \qquad
        \int_0^1H_\lambda(u,z)\,\dd u
        \le {\frac{1}{\lambda}}z^{\lambda-1},
\label{eq:hard-kernel-one-sided}
\end{equation}
and
\begin{equation}
        \int_0^1\int_0^1H_\lambda(u,z)\,\dd z\,\dd u
        \le {\frac{1}{\lambda^2}}.
\label{eq:hard-kernel-double}
\end{equation}
\end{lemma}

\begin{proof}
Since $1-2\lambda>0$, $M_\lambda(u,z)\le1$ on $[0,1]^2$, and $(1-\lambda)/4\le1$,
\[
\begin{aligned}
        H_\lambda(u,z)
        &={\frac{1-\lambda}{4}}
        M_\lambda(u,z)^{1-2\lambda}
        u^{\lambda-1}z^{\lambda-1}                                \\
        &\le u^{\lambda-1}z^{\lambda-1}.
\end{aligned}
\]
Integrating this inequality once and twice gives
\[
\begin{aligned}
        \int_0^1H_\lambda(u,z)\,\dd z
        &\le u^{\lambda-1}\int_0^1z^{\lambda-1}\,\dd z
        ={\frac{1}{\lambda}}u^{\lambda-1},                            \\
        \int_0^1H_\lambda(u,z)\,\dd u
        &\le z^{\lambda-1}\int_0^1u^{\lambda-1}\,\dd u
        ={\frac{1}{\lambda}}z^{\lambda-1},                            \\
        \int_0^1\int_0^1H_\lambda(u,z)\,\dd z\,\dd u
        &\le
        \left(\int_0^1u^{\lambda-1}\,\dd u\right)
        \left(\int_0^1z^{\lambda-1}\,\dd z\right)
        ={\frac{1}{\lambda^2}}.
\end{aligned}
\]
\end{proof}

In the next lemma, we show that $G_\lambda (p)$ is Lipschitz in $p$ (and bounded). 

\begin{lemma}[Lipschitzness and boundedness of $G_\lambda$]
\label{lem:hard-lipschitz}
For all $p,q\in[0,1]$,
\begin{equation}
        |G_\lambda(p)-G_\lambda(q)|
        \le L\left(2+{\frac{1}{\lambda}}\right)|p-q|
        \le {\frac{\CL}{16}}|p-q|.
\label{eq:hard-G-lipschitz}
\end{equation}
Moreover,
\begin{equation}
        0\le G_\lambda(p)\le1
        \qquad
        \text{for every }p\in[0,1].
\label{eq:hard-G-bounded}
\end{equation}
\end{lemma}

\begin{proof}
The bound \eqref{eq:hard-G-bounded} follows from
$0\le M_\lambda(x,y)\le1$ on $[0,1]^2$.

Let $0\le p<q\le1$. By direct differentiation, for $x,y>0$,
\begin{equation}
        \partial_xM_\lambda(x,y)
        ={\frac{1}{2}}M_\lambda(x,y)^{1-\lambda}x^{\lambda-1},
        \qquad
        \partial_yM_\lambda(x,y)
        ={\frac{1}{2}}M_\lambda(x,y)^{1-\lambda}y^{\lambda-1}.
\label{eq:hard-M-derivatives}
\end{equation}
For almost every pair $(S,B)$ such that $S\le p$ and $B\ge q$,
the fundamental theorem of calculus and
\eqref{eq:hard-M-derivatives} give
\begin{align}
        &\left|
        M_\lambda(q-S,B-q)
        -M_\lambda(p-S,B-p)
        \right|                                                     \notag\\
        &\qquad=
        \left|
        \int_0^{q-p}
        \Bigl[
        \partial_xM_\lambda(p+t-S,B-p-t)
        -
        \partial_yM_\lambda(p+t-S,B-p-t)
        \Bigr]
        \,\dd t
        \right|                                                     \notag\\
        &\qquad\le
        {\frac{1}{2}}
        \int_0^{q-p}
        \left[
        (p+t-S)^{\lambda-1}
        +(B-p-t)^{\lambda-1}
        \right]
        \,\dd t,
\label{eq:hard-common-trade-pointwise}
\end{align}
where we used $M_\lambda\le1$.

Let $f_S$ and $f_B$ denote the marginal densities of $S$ and
$B$. Recall that $0\le f_S\le L$ and $0\le f_B\le L$ almost everywhere.
By \eqref{eq:hard-common-trade-pointwise} and
Tonelli's theorem,
\begin{align}
        &\E\left[
        \mathbb{I}\{S\le p,\ B\ge q\}
        \left|
        M_\lambda(q-S,B-q)
        -M_\lambda(p-S,B-p)
        \right|
        \right]                                                     \notag\\
        &\qquad\le
        {\frac{1}{2}}
        \int_0^{q-p}
        \E\left[
        \mathbb{I}\{S\le p,\ B\ge q\}
        \left(
        (p+t-S)^{\lambda-1}
        +(B-p-t)^{\lambda-1}
        \right)
        \right]
        \,\dd t                                                     \notag\\
        &\qquad=
        {\frac{1}{2}}
        \int_0^{q-p}
        \int_0^p
        \int_q^1
        \left[
        (p+t-s)^{\lambda-1}
        +(b-p-t)^{\lambda-1}
        \right]
        f_S(s)f_B(b)
        \,\dd b\,\dd s\,\dd t                                     \notag\\
        &\qquad=
        {\frac{1}{2}}
        \int_0^{q-p}
        \Bigg[
        \left(
        \int_0^p
        (p+t-s)^{\lambda-1}f_S(s)\,\dd s
        \right)
        \left(
        \int_q^1 f_B(b)\,\dd b
        \right)                                                     \notag\\
        &\hspace{4.5cm}
        +
        \left(
        \int_0^p f_S(s)\,\dd s
        \right)
        \left(
        \int_q^1
        (b-p-t)^{\lambda-1}f_B(b)\,\dd b
        \right)
        \Bigg]
        \,\dd t                                                     \notag\\
        &\qquad\le
        {\frac{L}{2}}
        \int_0^{q-p}
        \left[
        \int_0^p
        (p+t-s)^{\lambda-1}\,\dd s
        +
        \int_q^1
        (b-p-t)^{\lambda-1}\,\dd b
        \right]
        \,\dd t                                                     \notag\\
        &\qquad=
        {\frac{L}{2\lambda}}
        \int_0^{q-p}
        \left[
        (p+t)^\lambda-t^\lambda
        +(1-p-t)^\lambda-(q-p-t)^\lambda
        \right]
        \,\dd t                                                     \notag\\
        &\qquad\le
        {\frac{L}{\lambda}}(q-p).
\label{eq:hard-common-trade-bound}
\end{align}

Moreover,
\begin{align}
        \Prob(S\le p,\ p\le B<q)
        &\le
        \Prob(p\le B<q)
        \le L(q-p),                                                 \label{eq:hard-buyer-strip-bound}\\
        \Prob(p<S\le q,\ B\ge q)
        &\le
        \Prob(p<S\le q)
        \le L(q-p).
\label{eq:hard-seller-strip-bound}
\end{align}
Combining
\eqref{eq:hard-common-trade-bound}-\eqref{eq:hard-seller-strip-bound},
we obtain
\begin{align}
        |G_\lambda(q)-G_\lambda(p)|
        & = 
\Big|
  \E\left[
        \mathbb{I}\{S\le p,\ B\ge q\}
        \left(
        M_\lambda(q-S,B-q)
        -M_\lambda(p-S,B-p)
        \right)
        \right]          
       \notag \\
& ~ -
\E\left[
\mathbb{I}\{S\le p,\ p\le B<q\}
M_\lambda(p-S,B-p)
\right]
\\
& ~ +
\E\left[
\mathbb{I}\{p<S\le q,\ B\ge q\}
M_\lambda(q-S,B-q)
\right]
\Big|
\notag
\\
        &\le
        {\frac{L}{\lambda}}(q-p)+2L(q-p)                               \notag\\
        &=L\left(2+{\frac{1}{\lambda}}\right)(q-p)                     \notag\\
        &\le{\frac{\CL}{16}}(q-p),
\end{align}
where the last
inequality follows from \eqref{eq:hard-constant-bounds}.
\end{proof}

Recalling the notation
\[
   \mathcal I_\lambda(p)=
        \int_0^p\int_p^1H_\lambda(p-v,w-p)F_S(v)F_B^\circ(w)\dd w\dd v
\]
introduced in Lemma~\ref{lem:main-reconstruction}, the difficulty lies in the term $\mathcal I_\lambda(p)$ as $H_\lambda$ is singular when $v \to p$ or $w \to p$. The solution will consist in decomposing $\mathcal I_\lambda(p)$ in four terms, some of which replace $F_S(v)$ and $F_B^\circ(w)$ by $F_S(v)-F_S(p)$ and $F_B^\circ(w)-F_B^\circ(p)$, see Lemma~\ref{lem:centered-decomposition}. These latter two new quantities are $O_L(p-v)$ and $O_L(w-p)$, respectively, when $v \to p$ or $w \to p$, which counter-balance the singularity of $H_\lambda$. We start with some definitions.

For $p\in[0,1]$, define
\begin{equation}
        A^S(p):=\int_0^pF_S(v)\,\dd v,
        \qquad
        A^B(p):=\int_p^1F_B^\circ(w)\,\dd w,
\label{eq:hard-def-AS-AB}
\end{equation}
and recall \eqref{eq:KpS} and \eqref{eq:KpB} for $K_p^S$ and $K_p^B$, respectively,
\begin{equation}
        K_p^S(u):=\int_p^1H_\lambda(u,w-p)\,\dd w,
        \quad 0<u\le p,
        \quad
        K_p^B(z):=\int_0^pH_\lambda(p-v,z)\,\dd v,
        \quad 0<z\le1-p.
\label{eq:hard-def-KS-KB}
\end{equation}
as well as \eqref{eq:Jp} for $J(p)$,
\begin{equation}
        J(p):=\int_0^p\int_p^1
        H_\lambda(p-v,w-p)\,\dd w\,\dd v.
\label{eq:hard-def-J}
\end{equation}
We are ready for our decomposition lemma.

\begin{lemma}[Centered decomposition]
\label{lem:centered-decomposition}
For every $p\in[0,1]$,
\begin{equation}
\begin{aligned}
        \mathcal I_\lambda(p)
        ={}&F_S(p)F_B^\circ(p)J(p)
        +F_B^\circ(p)L^S(p)
        +F_S(p)L^B(p)
        +C(p),                                                       \\
        L^S(p)
        :={}&\int_0^p
        (F_S(v)-F_S(p))K_p^S(p-v)\,\dd v,                           \\
        L^B(p)
        :={}&\int_p^1
        (F_B^\circ(w)-F_B^\circ(p))K_p^B(w-p)\,\dd w,              \\
        C(p)
        :={}&\int_0^p\int_p^1
        (F_S(v)-F_S(p))(F_B^\circ(w)-F_B^\circ(p))                 \\
        &\hspace{35mm}\times
        H_\lambda(p-v,w-p)\,\dd w\,\dd v.
\end{aligned}
\label{eq:hard-centered-decomposition}
\end{equation}
\end{lemma}

\begin{proof}
For $v\le p\le w$,
\[
\begin{aligned}
        F_S(v)F_B^\circ(w)
        ={}&F_S(p)F_B^\circ(p)
        +F_B^\circ(p)(F_S(v)-F_S(p))                                \\
        &+F_S(p)(F_B^\circ(w)-F_B^\circ(p))                         \\
        &+(F_S(v)-F_S(p))(F_B^\circ(w)-F_B^\circ(p)).
\end{aligned}
\]
Multiplying by $H_\lambda(p-v,w-p)$ and integrating over
$[0,p]\times[p,1]$ gives \eqref{eq:hard-centered-decomposition}.
\end{proof}

In order to see how to deal with $L^S(p)$, $L^B(p)$ and $C(p)$, the strategy will consist in approximating each of them by truncating a small portion $\Delta^{1/(1+\lambda)}$ (defined below) of the range of integration in the one-dimensional integrations. Setting the truncation to $\Delta^{1/(1+\lambda)}$, the resulting bias will be of order $\Delta$ (Lemma~\ref{lem:truncation}). On the other hand, the resulting approximation of $L^S(p)$, $L^B(p)$ and $C(p)$ will suffer from an estimation fluctuation of order $\Delta$ (Lemma~\ref{lem:algorithmic-concentration}). Importantly, thanks to Lemma~\ref{lem:centered-decomposition}, the truncation will be of order $\Delta^{1/(1+\lambda)}$, which is of larger order than $\Delta$.

For $\Delta\in(0,1]$, define
\begin{equation}
L^S_\Delta(p):=
\begin{cases}
\displaystyle
\int_0^{p-\Delta^{\frac{1}{1+\lambda}}}
(F_S(v)-F_S(p))K_p^S(p-v)\,\dd v,
&\text{if }p\ge\Delta^{\frac{1}{1+\lambda}},\smallskip\\
0,
&\text{if }p<\Delta^{\frac{1}{1+\lambda}},
\end{cases}
\label{eq:hard-def-LS-Delta}
\end{equation}
\begin{equation}
L^B_\Delta(p):=
\begin{cases}
\displaystyle
\int_{p+\Delta^{\frac{1}{1+\lambda}}}^1
(F_B^\circ(w)-F_B^\circ(p))K_p^B(w-p)\,\dd w,
&\text{if }p\le1-\Delta^{\frac{1}{1+\lambda}},\smallskip\\
0,
&\text{if }p>1-\Delta^{\frac{1}{1+\lambda}},
\end{cases}
\label{eq:hard-def-LB-Delta}
\end{equation}
\begin{equation}
C_\Delta(p):=
\begin{cases}
\displaystyle
\begin{aligned}
&\int_0^{p-\Delta^{\frac{1}{1+\lambda}}}
\int_{p+\Delta^{\frac{1}{1+\lambda}}}^1
(F_S(v)-F_S(p))(F_B^\circ(w)-F_B^\circ(p))\\
&\hspace{25mm}\times
H_\lambda(p-v,w-p)\,\dd w\,\dd v,
\end{aligned}
&\begin{gathered}
\text{if }\Delta^{\frac{1}{1+\lambda}}\le\frac12\text{ and}\\[-1mm]
\Delta^{\frac{1}{1+\lambda}}\le p
\le1-\Delta^{\frac{1}{1+\lambda}},
\end{gathered}
\smallskip\\
0,
&\text{otherwise.}
\end{cases}
\label{eq:hard-def-C-Delta}
\end{equation}
Finally, set
\begin{equation}
        \mathcal I_{\lambda,\Delta}(p)
        :=F_S(p)F_B^\circ(p)J(p)
        +F_B^\circ(p)L^S_\Delta(p)
        +F_S(p)L^B_\Delta(p)
        +C_\Delta(p).
\label{eq:hard-def-I-Delta}
\end{equation}

\begin{lemma}[Truncation bias]
\label{lem:truncation}
For every $\Delta\in(0,1]$ and $p\in[0,1]$,
\begin{equation}
        |L^S(p)-L^S_\Delta(p)|
        +|L^B(p)-L^B_\Delta(p)|
        +|C(p)-C_\Delta(p)|
        \le\CL^{1/2}\Delta,
\label{eq:hard-truncation-components}
\end{equation}
and
\begin{equation}
        |\mathcal I_\lambda(p)-\mathcal I_{\lambda,\Delta}(p)|
        \le\CL^{1/2}\Delta.
\label{eq:hard-truncation-I}
\end{equation}
\end{lemma}

\begin{proof}
By \eqref{eq:hard-kernel-one-sided} and the $L$-Lipschitzness of $F_S$,
\[
\begin{aligned}
        |L^S(p)-L^S_\Delta(p)|
        &\le
        \begin{cases}
        \displaystyle
        \int_{p-\Delta^{\frac{1}{1+\lambda}}}^{p}
        |F_S(v)-F_S(p)|K_p^S(p-v)\,\dd v,
        & p\ge\Delta^{\frac{1}{1+\lambda}},\smallskip\\
        \displaystyle
        \int_0^p|F_S(v)-F_S(p)|K_p^S(p-v)\,\dd v,
        & p<\Delta^{\frac{1}{1+\lambda}},
        \end{cases}                                                \\
        &\le
        {\frac{L}{\lambda}}
        \int_0^{\Delta^{\frac{1}{1+\lambda}}}u^\lambda\,\dd u
        ={\frac{L}{\lambda(1+\lambda)}}\Delta.
\end{aligned}
\]
Similarly,
\[
\begin{aligned}
        |L^B(p)-L^B_\Delta(p)|
        &\le
        \begin{cases}
        \displaystyle
        \int_p^{p+\Delta^{\frac{1}{1+\lambda}}}
        |F_B^\circ(w)-F_B^\circ(p)|K_p^B(w-p)\,\dd w,
        & p\le1-\Delta^{\frac{1}{1+\lambda}},\smallskip\\
        \displaystyle
        \int_p^1|F_B^\circ(w)-F_B^\circ(p)|K_p^B(w-p)\,\dd w,
        & p>1-\Delta^{\frac{1}{1+\lambda}},
        \end{cases}                                                \\
        &\le
        {\frac{L}{\lambda}}
        \int_0^{\Delta^{\frac{1}{1+\lambda}}}z^\lambda\,\dd z
        ={\frac{L}{\lambda(1+\lambda)}}\Delta.
\end{aligned}
\]
Let us finally address $C(p)-C_\Delta(p)$. 
We use \eqref{eq:hard-kernel-pointwise},
$|F_S(x)-F_S(y)|\le L|x-y|$, and
$|F_B^\circ(x)-F_B^\circ(y)|\le L |x-y|$. We distinguish three cases.

If $p<\Delta^{\frac{1}{1+\lambda}}$, then $C_\Delta(p)=0$, and
\[
\begin{aligned}
        |C(p)-C_\Delta(p)|
        &=|C(p)|                                                   \le
        \int_0^p\int_0^{1-p}
        Lu\,H_\lambda(u,z)\,\dd z\,\dd u                         \le
        {\frac{L}{\lambda}}\int_0^p u^\lambda\,\dd u               \le {\frac{L}{\lambda(1+\lambda)}}\Delta.
\end{aligned}
\]
If $p\ge\Delta^{\frac{1}{1+\lambda}}$ and
$p>1-\Delta^{\frac{1}{1+\lambda}}$, then $C_\Delta(p)=0$, and
\[
\begin{aligned}
        |C(p)-C_\Delta(p)|
        &=|C(p)|                                                   \le
        \int_0^p\int_0^{1-p}
        Lz\,H_\lambda(u,z)\,\dd z\,\dd u                         \le
        {\frac{L}{\lambda}}\int_0^{1-p} z^\lambda\,\dd z           \le {\frac{L}{\lambda(1+\lambda)}}\Delta.
\end{aligned}
\]
Finally, if
$\Delta^{\frac{1}{1+\lambda}}\le p
\le1-\Delta^{\frac{1}{1+\lambda}}$, then
\[
\begin{aligned}
        |C(p)-C_\Delta(p)|
        &\le
        \int_0^{\Delta^{\frac{1}{1+\lambda}}}
        \int_0^{1-p}
        Lu\,H_\lambda(u,z)\,\dd z\,\dd u                         \\
        &\quad+
        \int_0^p
        \int_0^{\Delta^{\frac{1}{1+\lambda}}}
        Lz\,H_\lambda(u,z)\,\dd z\,\dd u                         \\
        &\le
        {\frac{L}{\lambda}}
        \int_0^{\Delta^{\frac{1}{1+\lambda}}}u^\lambda\,\dd u
        +{\frac{L}{\lambda}}
        \int_0^{\Delta^{\frac{1}{1+\lambda}}}z^\lambda\,\dd z   \\
        &= {\frac{2L}{\lambda(1+\lambda)}}\Delta.
\end{aligned}
\]
Therefore,
\[
\begin{aligned}
        &|L^S(p)-L^S_\Delta(p)|
        +|L^B(p)-L^B_\Delta(p)|
        +|C(p)-C_\Delta(p)|                                        \\
        &\qquad\le
        {\frac{4L}{\lambda(1+\lambda)}}\Delta
        \le\CL^{1/2}\Delta,
\end{aligned}
\]
by \eqref{eq:hard-constant-bounds}. Equation~\eqref{eq:hard-truncation-I} follows from \eqref{eq:hard-centered-decomposition}, \eqref{eq:hard-def-I-Delta}, and $F_S(p),F_B^\circ(p)\in[0,1]$.
\end{proof}

\subsection{The good event}
\label{section:good:event}

For $m\in\{0,\ldots,M_T+1\}$, we will make use of the following two properties
\begin{equation*}
        p^\star\in A_m
        \quad
        \text{for every }p^\star\in\underset{p\in[0,1]}{\operatorname{argmax}}\,G_\lambda(p),
        \qquad
        G_\lambda^\star-G_\lambda(p)
        \le3\CL\Delta_m
        \quad\text{for every }p\in A_m,
        \tag{Act$_m$}
\label{eq:hard-active-invariant}
\end{equation*}
and
\begin{equation*}
        {\frac{\CL}{2}}\Delta_j
        \le G_\lambda^\star-G_\lambda(x)
        \le3\CL\Delta_j
        \quad
        \text{for every }x\in R_j,
        \quad j=0,\ldots,m-1.
        \tag{Rej$_{<m}$}
\label{eq:hard-rejected-invariant}
\end{equation*}
We denote by $H_m$ the event that \eqref{eq:hard-active-invariant} and \eqref{eq:hard-rejected-invariant} are verified at the beginning of the $m$-th epoch.

Intuitively, when $H_m$ occurs, Algorithm~\ref{alg:main} partitions its action space $[0,1]$ into the sets $R_0,\ldots,R_{m-1},A_m$ that capture the various levels of optimality of the different prices. Let us introduce some notations and events.

\begin{itemize}
\item
For $m=0,\ldots,M_T$, the $t_m$ random time denotes the last round in the $m$-th epoch, and we define its history vector up to (and including) $t_m$ by
\begin{equation}
        \cF_m
        :=(P_1,\ldots,P_{t_m},Y_1,\ldots,Y_{t_m},Z_1,\ldots,Z_{t_m}),
\label{eq:hard-def-Fm}
\end{equation}
recalling \eqref{eq:Yt:Zt}. We also denote $t_{-1}=0$, and let $\cF_{-1}$ be the empty history.

\item
For $m=0,\ldots,M_T$, define $B_m$ as the event
\begin{equation}
        \sup_{p\in N_m}
        |\widehat G_m(p)-G_\lambda(p)|
        \le {\frac{\CL}{16}}\Delta_m,
\label{eq:hard-def-Bm}
\end{equation}
We introduce
\begin{equation}
        \mathcal E_m:=\bigcap_{k=0}^mB_k,
        \qquad
        \mathcal E_{-1}:=\Omega.
\label{eq:hard-def-Em}
\end{equation}
\end{itemize}

By construction of the underlying probability space $\Omega$, the event $\mathcal E_{-1}$ is always true. Intuitively, event $B_m$ ensures that our estimation procedure in Algorithm~\ref{alg:estimate} enjoys an accuracy of order $\Delta_m$ in the $m$-th epoch. We will refer to $\mathcal E_{M_T}$ as the "good event".

Conditioned on $\cF_{m-1}$, the active and rejected sets, the action space partition $\cP_m$ as well as the number of rounds used in the $m$-th epoch are deterministic. Moreover, conditioned on $\cF_{m-1}$, all valuation pairs drawn during the $m$-th epoch are i.i.d. with distribution $\nu_S\otimes\nu_B$. The auxiliary random variables $U_J$, $J\in\cP_m$, uniformly sample their respective cells: by construction, conditioned on $\cF_{m-1}$, they are independent of one another and of the valuation pairs. Moreover, for all possible realizations of our algorithm,
\begin{equation}
        R_0,\ldots,R_{m-1},A_m
        \quad\text{are pairwise disjoint and form a partition of }[0,1].
\label{eq:hard-layer-partition}
\end{equation}

Intuitively, under $H_m$, for $p\in A_m$ and $x\in R_j$ with $j<m$, when $x$ is close to $p$, $\Delta_j$ is small and thus Algorithm~\ref{alg:main} explores the $p$ neighborhood densely. This ensures that the region close to $p$ is sampled sufficiently densely.

\begin{lemma}[Layer comparison]
\label{lem:layer-comparison}
Let $m \in \{ 0 ,\dots, M_T\}$.
Assume that $H_m$ holds. For every $p\in A_m$,
\begin{equation}
        x\in A_m
        \quad\Longrightarrow\quad
        \Delta_m\le7(\Delta_m+|p-x|),
\label{eq:hard-layer-control-active}
\end{equation}
and, for every $j<m$,
\begin{equation}
        x\in R_j
        \quad\Longrightarrow\quad
        \Delta_j\le7(\Delta_m+|p-x|).
\label{eq:hard-layer-control-rejected}
\end{equation}
\end{lemma}

\begin{proof}
Equation~\eqref{eq:hard-layer-control-active} is immediate. If $x\in R_j$, then $H_m$ and \eqref{eq:hard-G-lipschitz} give
\[
\begin{aligned}
        {\frac{\CL}{2}}\Delta_j
        &\le G_\lambda^\star-G_\lambda(x)                           \\
        &=G_\lambda^\star-G_\lambda(p)+G_\lambda(p)-G_\lambda(x)  \\
        &\le3\CL\Delta_m+{\frac{\CL}{16}}|p-x|,
\end{aligned}
\]
and hence
\[
        \Delta_j
        \le6\Delta_m+{\frac{1}{8}}|p-x|
        \le7(\Delta_m+|p-x|).
\]
\end{proof}

The next lemma provides estimates of the length of the sampling cells in Algorithm~\ref{alg:estimate}.

\begin{lemma}[Sampling-cell geometry]
\label{lem:partition-geometry}
Let $m\in\{0,\ldots,M_T\}$, let $\delta\in(0,1)$, set
\[
        \Delta=\Delta_m,
        \qquad
        \ell=\ell(\Delta,\delta),
\]
and let $\cP_m$ be the sampling partition constructed by
Algorithm~\ref{alg:estimate}. Every $J\in\cP_m$ satisfies
\begin{align}
        J\subset I,\ I\in\mathsf R_j
        &\quad\Longrightarrow\quad
        {\frac{8\Delta^2\Delta_j}{\CL\ell}}
        \le|J|\le
        {\frac{16\Delta^2\Delta_j}{\CL\ell}}
        \le2^{-12}\Delta^2\Delta_j,
\label{eq:hard-cell-size-rejected}\\
        J\subset I,\ I\in\mathsf A_m
        &\quad\Longrightarrow\quad
        {\frac{8\Delta^3}{\CL\ell}}
        \le|J|\le
        {\frac{16\Delta^3}{\CL\ell}}
        \le2^{-12}\Delta^3.
\label{eq:hard-cell-size-active}
\end{align}
Moreover,
\begin{equation}
        |\cP_m|\le{\frac{\CL\ell}{8\Delta^3}}.
\label{eq:hard-number-cells}
\end{equation}
\end{lemma}

\begin{proof}
Set
\[
        x:={\frac{\CL\ell}{16\Delta^2}}.
\]
By \eqref{eq:hard-constant-bounds} and $\ell\ge\log64$, we have
$x\ge1$. Hence,
\[
        x\le n_m=\lceil x\rceil\le2x.
\]
If $J\subset I\in\mathsf R_j$, then
\[
\begin{aligned}
        |J|
        &={\frac{\Delta_j}{n_m}},                                    \\
        {\frac{8\Delta^2\Delta_j}{\CL\ell}}
        &\le |J|\le
        {\frac{16\Delta^2\Delta_j}{\CL\ell}}.
\end{aligned}
\]
Since $\CL\ge2^{34}$ and $\ell\ge1$,
\[
        {\frac{16}{\CL\ell}}\le2^{-12},
\]
which proves \eqref{eq:hard-cell-size-rejected}. If
$J\subset I\in\mathsf A_m$, the same computation with
$|I|=\Delta$ proves \eqref{eq:hard-cell-size-active}.

Since $\Delta_j\ge\Delta$ for $j<m$, every cell in $\cP_m$
has length at least $8\Delta^3/(\CL\ell)$. Using
\eqref{eq:hard-layer-partition},
\[
        1=\sum_{J\in\cP_m}|J|
        \ge|\cP_m|{\frac{8\Delta^3}{\CL\ell}},
\]
which proves \eqref{eq:hard-number-cells}.
\end{proof}

Lemma~\ref{lem:algorithmic-concentration} below gives the concentration of each component in our estimator $\widehat G_m(p)$. The proof is obtained by a repeated use of Hoeffding's and McDiarmid's concentration inequalities. The proof of Lemma~\ref{lem:algorithmic-concentration} will require bounding the sums of squared ranges appearing when applying Hoeffding's concentration inequality and the sums of squared variation ranges appearing when applying McDiarmid's concentration inequality. These quantities will be controlled in Lemmas~\ref{lem:discrete-mass} and~\ref{lem:quadratic-envelope}.

\begin{lemma}[Bounds on sums over $\cP_m$]
\label{lem:discrete-mass}
Let $m \in \{ 0 ,\dots, M_T\}$ and let $\delta\in(0,1)$. Assume that $H_m$ holds, set $\Delta=\Delta_m$, let $\cP_m$ be the sampling partition constructed by Algorithm~\ref{alg:estimate} with confidence $\delta$, and choose arbitrary points $U_J\in J$ for $J\in\cP_m$. For every $p\in A_m$, the following inequalities hold.

If $p\ge\Delta^{\frac{1}{1+\lambda}}$ and
$0\le x\le p-\Delta^{\frac{1}{1+\lambda}}$, then
\begin{equation}
        \sum_{\substack{K\in\cP_m\\
        U_K\ge p+\Delta^{\frac{1}{1+\lambda}}}}
        |K|H_\lambda(p-x,U_K-p)
        \le\CL^{1/2}(p-x)^{\lambda-1}.
\label{eq:hard-discrete-mass-right}
\end{equation}
If $p\le1-\Delta^{\frac{1}{1+\lambda}}$ and
$p+\Delta^{\frac{1}{1+\lambda}}\le y\le1$, then
\begin{equation}
        \sum_{\substack{J\in\cP_m\\
        U_J\le p-\Delta^{\frac{1}{1+\lambda}}}}
        |J|H_\lambda(p-U_J,y-p)
        \le\CL^{1/2}(y-p)^{\lambda-1}.
\label{eq:hard-discrete-mass-left}
\end{equation}
If $\Delta^{\frac{1}{1+\lambda}}\le1/2$ and
$\Delta^{\frac{1}{1+\lambda}}\le p
\le1-\Delta^{\frac{1}{1+\lambda}}$, then
\begin{equation}
\begin{aligned}
        &\sum_{\substack{J,K\in\cP_m\\
        U_J\le p-\Delta^{\frac{1}{1+\lambda}}\\
        U_K\ge p+\Delta^{\frac{1}{1+\lambda}}}}
        |J||K|H_\lambda(p-U_J,U_K-p)
        \le\CL^{1/2}.
\end{aligned}
\label{eq:hard-discrete-rectangular-mass}
\end{equation}
Finally, for $p\ge\Delta^{\frac{1}{1+\lambda}}$,
\begin{equation}
        \sum_{\substack{J\in\cP_m\\
        U_J\le p-\Delta^{\frac{1}{1+\lambda}}}}
        |J|K_p^S(p-U_J)
        \le\CL^{1/2},
\label{eq:hard-discrete-KS-mass}
\end{equation}
and, for  $p\le1-\Delta^{\frac{1}{1+\lambda}}$,
\begin{equation}
        \sum_{\substack{K\in\cP_m\\
        U_K\ge p+\Delta^{\frac{1}{1+\lambda}}}}
        |K|K_p^B(U_K-p)
        \le\CL^{1/2}.
\label{eq:hard-discrete-KB-mass}
\end{equation}
\end{lemma}

\begin{proof}
Consider $K\in\cP_m$ such that
$U_K\ge p+\Delta^{\frac{1}{1+\lambda}}$. By
Lemmas~\ref{lem:layer-comparison} and~\ref{lem:partition-geometry},
\[
\begin{aligned}
        |K|
        &\le
        2^{-12}\Delta^2
        \begin{cases}
        \Delta_j,&K\subset R_j,\ j<m,\\
        \Delta,&K\subset A_m,
        \end{cases}                                                \\
        &\le7\cdot2^{-12}\Delta^2(\Delta+U_K-p)                   \\
        &\le14\cdot2^{-12}\Delta^2(U_K-p)
        \le{\frac{U_K-p}{64}},
\end{aligned}
\]
because $U_K-p\ge\Delta^{\frac{1}{1+\lambda}}\ge\Delta$. Thus, for every $z\in K$,
\[
\begin{aligned}
        z-p
        &\ge U_K-p-|K|
        \ge{\frac{63}{64}}(U_K-p)
        \ge{\frac{1}{2}}(U_K-p),                                      \\
        z-p
        &\le U_K-p+|K|
        \le{\frac{65}{64}}(U_K-p)
        \le2(U_K-p),                                               \\
        (z-p)^{\lambda-1}
        &\ge2^{\lambda-1}(U_K-p)^{\lambda-1}
        \ge{\frac{1}{2}}(U_K-p)^{\lambda-1}.
\end{aligned}
\]
Consequently,
\begin{equation}
\begin{aligned}
        &\sum_{\substack{K\in\cP_m\\
        U_K\ge p+\Delta^{\frac{1}{1+\lambda}}}}
        |K|(U_K-p)^{\lambda-1}                       
        \qquad\le
        2\sum_{\substack{K\in\cP_m\\
        U_K\ge p+\Delta^{\frac{1}{1+\lambda}}}}
        \int_K(z-p)^{\lambda-1}\,\dd z                            \\
        &\qquad\le
        2\int_{\min\{1,p+\frac12\Delta^{\frac{1}{1+\lambda}}\}}^1
        (z-p)^{\lambda-1}\,\dd z
        \le{\frac{2}{\lambda}}.
\end{aligned}
\label{eq:hard-right-power-sum}
\end{equation}
Using \eqref{eq:hard-kernel-pointwise} and
\eqref{eq:hard-right-power-sum},
\[
\begin{aligned}
        &\sum_{\substack{K\in\cP_m\\
        U_K\ge p+\Delta^{\frac{1}{1+\lambda}}}}
        |K|H_\lambda(p-x,U_K-p)                                   \qquad\le
        (p-x)^{\lambda-1}
        \sum_{\substack{K\in\cP_m\\
        U_K\ge p+\Delta^{\frac{1}{1+\lambda}}}}
        |K|(U_K-p)^{\lambda-1}                                    \\
        &\qquad\le{\frac{2}{\lambda}}(p-x)^{\lambda-1}
        \le\CL^{1/2}(p-x)^{\lambda-1},
\end{aligned}
\]
which proves \eqref{eq:hard-discrete-mass-right}.

For $J\in\cP_m$ such that
$U_J\le p-\Delta^{\frac{1}{1+\lambda}}$, Lemmas~\ref{lem:layer-comparison} and~\ref{lem:partition-geometry} give
\[
\begin{aligned}
        |J|
        &\le7\cdot2^{-12}\Delta^2(\Delta+p-U_J)
        \le14\cdot2^{-12}\Delta^2(p-U_J)
        \le{\frac{p-U_J}{64}},                                        \\
        p-v
        &\ge p-U_J-|J|
        \ge{\frac{63}{64}}(p-U_J)
        \ge{\frac{1}{2}}(p-U_J),                                      \\
        p-v
        &\le p-U_J+|J|
        \le2(p-U_J)
        \qquad(v\in J).
\end{aligned}
\]
Therefore,
\begin{equation}
\begin{aligned}
        &\sum_{\substack{J\in\cP_m\\
        U_J\le p-\Delta^{\frac{1}{1+\lambda}}}}
        |J|(p-U_J)^{\lambda-1}                                    \qquad\le
        2\int_0^{\max\{0,p-\frac12\Delta^{\frac{1}{1+\lambda}}\}}
        (p-v)^{\lambda-1}\,\dd v
        \le{\frac{2}{\lambda}}.
\end{aligned}
\label{eq:hard-left-power-sum}
\end{equation}
It follows that
\[
\begin{aligned}
        &\sum_{\substack{J\in\cP_m\\
        U_J\le p-\Delta^{\frac{1}{1+\lambda}}}}
        |J|H_\lambda(p-U_J,y-p)                                   \qquad\le
        (y-p)^{\lambda-1}
        \sum_{\substack{J\in\cP_m\\
        U_J\le p-\Delta^{\frac{1}{1+\lambda}}}}
        |J|(p-U_J)^{\lambda-1}                                    \\
        &\qquad\le{\frac{2}{\lambda}}(y-p)^{\lambda-1}
        \le\CL^{1/2}(y-p)^{\lambda-1},
\end{aligned}
\]
which proves \eqref{eq:hard-discrete-mass-left}.

Equations~\eqref{eq:hard-kernel-pointwise},
\eqref{eq:hard-right-power-sum}, and
\eqref{eq:hard-left-power-sum} also give
\[
\begin{aligned}
        &\sum_{\substack{J,K\in\cP_m\\
        U_J\le p-\Delta^{\frac{1}{1+\lambda}}\\
        U_K\ge p+\Delta^{\frac{1}{1+\lambda}}}}
        |J||K|H_\lambda(p-U_J,U_K-p)                              \\
        &\qquad\le
        \left(
        \sum_{\substack{J\in\cP_m\\
        U_J\le p-\Delta^{\frac{1}{1+\lambda}}}}
        |J|(p-U_J)^{\lambda-1}
        \right)
        \left(
        \sum_{\substack{K\in\cP_m\\
        U_K\ge p+\Delta^{\frac{1}{1+\lambda}}}}
        |K|(U_K-p)^{\lambda-1}
        \right)                                                    \\
        &\qquad\le{\frac{4}{\lambda^2}}
        \le\CL^{1/2}.
\end{aligned}
\]
Finally, \eqref{eq:hard-kernel-one-sided},
\eqref{eq:hard-left-power-sum}, and
\eqref{eq:hard-right-power-sum} yield
\[
\begin{aligned}
        \sum_{\substack{J\in\cP_m\\
        U_J\le p-\Delta^{\frac{1}{1+\lambda}}}}
        |J|K_p^S(p-U_J)
        &\le{\frac{1}{\lambda}}
        \sum_{\substack{J\in\cP_m\\
        U_J\le p-\Delta^{\frac{1}{1+\lambda}}}}
        |J|(p-U_J)^{\lambda-1}                                    \\
        &\le{\frac{2}{\lambda^2}}
        \le\CL^{1/2},                                             \\
        \sum_{\substack{K\in\cP_m\\
        U_K\ge p+\Delta^{\frac{1}{1+\lambda}}}}
        |K|K_p^B(U_K-p)
        &\le{\frac{1}{\lambda}}
        \sum_{\substack{K\in\cP_m\\
        U_K\ge p+\Delta^{\frac{1}{1+\lambda}}}}
        |K|(U_K-p)^{\lambda-1}                                    \\
        &\le{\frac{2}{\lambda^2}}
        \le\CL^{1/2}.
\end{aligned}
\]
\end{proof}

\begin{lemma}[Bounds on sums of squares over $\cP_m$]
\label{lem:quadratic-envelope}
Let $m\in\{0,\ldots,M_T\}$ and $\delta\in(0,1)$. Assume that
$H_m$ holds, set
\[
        \Delta=\Delta_m,
        \qquad
        \ell=\ell(\Delta,\delta),
\]
and let $\cP_m$ be the sampling partition constructed by
Algorithm~\ref{alg:estimate} with confidence $\delta$. For
$p\in N_m$ and $J\in\cP_m$, define
\begin{equation}
        a_J^-(p):=
        \begin{cases}
        \displaystyle
        |J|\sup_{\substack{x\in J\\
        x\le p-\Delta^{\frac{1}{1+\lambda}}}}
        (p-x)^{\lambda-1},
        &\begin{array}{l}
        \text{if }p\ge\Delta^{\frac{1}{1+\lambda}}
        \text{ and}\\[-1mm]
        J\cap[0,p-\Delta^{\frac{1}{1+\lambda}}]\neq\varnothing,
        \end{array}\smallskip\\
        0,&\text{otherwise,}
        \end{cases}
\label{eq:hard-def-a-minus}
\end{equation}
and
\begin{equation}
        a_J^+(p):=
        \begin{cases}
        \displaystyle
        |J|\sup_{\substack{x\in J\\
        x\ge p+\Delta^{\frac{1}{1+\lambda}}}}
        (x-p)^{\lambda-1},
        &\begin{array}{l}
        \text{if }p\le1-\Delta^{\frac{1}{1+\lambda}}
        \text{ and}\\[-1mm]
        J\cap[p+\Delta^{\frac{1}{1+\lambda}},1]\neq\varnothing,
        \end{array}\smallskip\\
        0,&\text{otherwise.}
        \end{cases}
\label{eq:hard-def-a-plus}
\end{equation}
Then
\begin{equation}
        \sum_{J\in\cP_m}(a_J^-(p))^2
        \le{\frac{32\Delta^2}{\CL^{1/2}\ell}},
        \qquad
        \sum_{J\in\cP_m}(a_J^+(p))^2
        \le{\frac{32\Delta^2}{\CL^{1/2}\ell}}.
\label{eq:hard-quadratic-envelope}
\end{equation}
\end{lemma}

\begin{proof}
If $p<\Delta^{\frac{1}{1+\lambda}}$, then
$a_J^-(p)=0$ for all $J$. Assume
$p\ge\Delta^{\frac{1}{1+\lambda}}$, and define
\[
        \cP_m^-(p)
        :=
        \left\{
        J\in\cP_m:
        J\cap[0,p-\Delta^{\frac{1}{1+\lambda}}]
        \neq\varnothing
        \right\}.
\]
Let $J\in\cP_m^-(p)$, and choose
$x_J\in J\cap[0,p-\Delta^{\frac{1}{1+\lambda}}]$. For every
$x\in J$, Lemmas~\ref{lem:layer-comparison} and
\ref{lem:partition-geometry} give
\[
\begin{aligned}
        |J|
        &\le{\frac{16\Delta^2}{\CL\ell}}
        \begin{cases}
        \Delta_j,&J\subset R_j,\ j<m,\\
        \Delta,&J\subset A_m,
        \end{cases}                                                \qquad \le{\frac{7\cdot16\Delta^2}{\CL\ell}}
        (\Delta+|p-x|),                                            \\
        |J|
        &\le{\frac{7\cdot16\Delta^2}{\CL\ell}}
        (\Delta+p-x_J)                                             \qquad\le{\frac{14\cdot16\Delta^2}{\CL\ell}}(p-x_J)
        \qquad
        \le{\frac{p-x_J}{64}}.
\end{aligned}
\]
Consequently, for every $x,y\in J$,
\[
\begin{aligned}
        p-x
        &\ge p-x_J-|J|
        \ge{\frac{63}{64}}(p-x_J)
        \ge{\frac{1}{2}}(p-x_J),                                      \\
        p-x
        &\le p-x_J+|J|
        \le{\frac{65}{64}}(p-x_J),                                    \\
        \sup_{y\in J}(p-y)^{2\lambda-2}
        &\le4{\frac{1}{|J|}}\int_J(p-x)^{2\lambda-2}\,\dd x.
\end{aligned}
\]
Thus, in particular, the cells in $\cP_m^-(p)$ are contained in
$[0,p-\frac12\Delta^{\frac{1}{1+\lambda}}]$. Therefore,
\[
\begin{aligned}
        \sum_{J\in\cP_m}(a_J^-(p))^2
        &\le4\sum_{J\in\cP_m^-(p)}|J|
        \int_J(p-x)^{2\lambda-2}\,\dd x                           \\
        &\le{\frac{2^9\Delta^2}{\CL\ell}}
        \sum_{J\in\cP_m^-(p)}\int_J
        (\Delta+p-x)(p-x)^{2\lambda-2}\,\dd x                    \\
        &\le{\frac{2^9\Delta^2}{\CL\ell}}
        \int_{\frac12\Delta^{\frac{1}{1+\lambda}}}^1
        (\Delta+u)u^{2\lambda-2}\,\dd u                           \\
        &= {\frac{2^9\Delta^2}{\CL\ell}}\left[
        \Delta\int_{\frac12\Delta^{\frac{1}{1+\lambda}}}^1
        u^{2\lambda-2}\,\dd u
        +\int_{\frac12\Delta^{\frac{1}{1+\lambda}}}^1
        u^{2\lambda-1}\,\dd u
        \right]                                                     \\
        &\le {\frac{2^9\Delta^2}{\CL\ell}}\left[
        \Delta\int_{\frac12\Delta^{\frac{1}{1+\lambda}}}^1
        u^{-2}\,\dd u
        +{\frac{1}{2\lambda}}
        \right]                                                     \\
        &\le {\frac{2^9\Delta^2}{\CL\ell}}\left[
        2\Delta^{\frac{\lambda}{1+\lambda}}
        +{\frac{1}{2\lambda}}
        \right]                                                     \\
        &\le {\frac{2^9}{\CL\ell}}
        \left(2+{\frac{1}{2\lambda}}\right)\Delta^2                \\
        &\le {\frac{32\Delta^2}{\CL^{1/2}\ell}},
\end{aligned}
\]
where the last inequality follows from
\eqref{eq:hard-constant-bounds}.

If $p>1-\Delta^{\frac{1}{1+\lambda}}$, then
$a_J^+(p)=0$ for all $J$. Otherwise, define
\[
        \cP_m^+(p)
        :=
        \left\{
        J\in\cP_m:
        J\cap[p+\Delta^{\frac{1}{1+\lambda}},1]
        \neq\varnothing
        \right\}.
\]
Choosing $x_J\in J\cap[p+\Delta^{\frac{1}{1+\lambda}},1]$ for
$J\in\cP_m^+(p)$, and replacing $p-x$ by $x-p$ in the
preceding computation gives
\[
\begin{aligned}
        \sum_{J\in\cP_m}(a_J^+(p))^2
        &\le {\frac{2^9\Delta^2}{\CL\ell}}
        \int_{\frac12\Delta^{\frac{1}{1+\lambda}}}^1
        (\Delta+u)u^{2\lambda-2}\,\dd u                           \le {\frac{2^9}{\CL\ell}}
        \left(2+{\frac{1}{2\lambda}}\right)\Delta^2                \le {\frac{32\Delta^2}{\CL^{1/2}\ell}}.
\end{aligned}
\]
\end{proof}

Before providing the concentration in Lemma \ref{lem:algorithmic-concentration}, we show that the individual components of the estimator $\widehat{G}_{m}(p)$ are unbiased for their targets.

\begin{lemma}[Unbiasedness and deterministic bounds]
\label{lem:algorithmic-unbiasedness}
Fix $m\in\{0,\ldots,M_T\}$
and $\delta\in(0,1)$, condition on a realization of
$\cF_{m-1}$ for which $H_m$ holds, and set $\Delta=\Delta_m$. For every $p\in N_m$,
\begin{equation}
\begin{gathered}
        \E[\widehat F_S(p)\mid\cF_{m-1}]=F_S(p),
        \qquad
        \E[\widehat F_B^\circ(p)\mid\cF_{m-1}]=F_B^\circ(p),
        \\
        \E[\widehat A^S(p)\mid\cF_{m-1}]=A^S(p),
        \qquad
        \E[\widehat A^B(p)\mid\cF_{m-1}]=A^B(p).
\end{gathered}
\label{eq:hard-unbiased-local-A}
\end{equation}
For every $p\in N_m$,
\begin{equation}
        \E[\widehat L_\Delta^S(p)\mid\cF_{m-1}]=L^S_\Delta(p),
\label{eq:hard-unbiased-LS}
\end{equation}
\begin{equation}
        \E[\widehat L_\Delta^B(p)\mid\cF_{m-1}]=L^B_\Delta(p),
\label{eq:hard-unbiased-LB}
\end{equation}
\begin{equation}
        \E[\widehat C_\Delta(p)\mid\cF_{m-1}]=C_\Delta(p)
\label{eq:hard-unbiased-C}
\end{equation}
and
\begin{equation}
        \E[\widehat{\mathcal I}_{\lambda,\Delta}(p)\mid\cF_{m-1}]
        =\mathcal I_{\lambda,\Delta}(p).
\label{eq:hard-Ihat-unbiased}
\end{equation}
Moreover, for every $p\in[0,1]$,
\begin{equation}
        J(p),
        \ |L^S_\Delta(p)|,
        \ |L^B_\Delta(p)|,
        \ |C_\Delta(p)|
        \le\CL^{1/2}.
\label{eq:hard-deterministic-component-bounds}
\end{equation}
\end{lemma}

\begin{proof}
Recall the expressions of $\widehat F_S(p)$ and
$\widehat F_B^\circ(p)$ in \eqref{eq:hatFS:FB}, and set
\[
        q=q_\Delta(\delta),
        \qquad
        \ell=\ell(\Delta,\delta).
\]
For every $a\in\{1,\ldots,q\}$,
\[
        \E[Y_{p,a}\mid\cF_{m-1}]=F_S(p),
        \qquad
        \E[Z_{p,a}\mid\cF_{m-1}]=F_B^\circ(p).
\]
Therefore,
\[
\begin{aligned}
        \E[\widehat F_S(p)\mid\cF_{m-1}]
        &={\frac{1}{q}}\sum_{a=1}^q\E[Y_{p,a}\mid\cF_{m-1}]
        =F_S(p),                                                    \\
        \E[\widehat F_B^\circ(p)\mid\cF_{m-1}]
        &={\frac{1}{q}}\sum_{a=1}^q\E[Z_{p,a}\mid\cF_{m-1}]
        =F_B^\circ(p).
\end{aligned}
\]
For $\widehat A^S(p)$, recalling \eqref{eq:hatAS},
\[
\begin{aligned}
        \E[\widehat A^S(p)\mid\cF_{m-1}]
        &=\sum_{J\in\cP_m}|J|
        \E[\mathbb I\{U_J\le p\}Y_J\mid\cF_{m-1}]              \\
        &=\sum_{J\in\cP_m}
        \int_J\mathbb I\{v\le p\}F_S(v)\,\dd v                 \\
        &=\int_0^pF_S(v)\,\dd v
        =A^S(p).
\end{aligned}
\]
Similarly, recalling \eqref{eq:hatAB},
\[
\begin{aligned}
        \E[\widehat A^B(p)\mid\cF_{m-1}]
        &=\sum_{J\in\cP_m}
        \int_J\mathbb I\{w\ge p\}F_B^\circ(w)\,\dd w           \\
        &=\int_p^1F_B^\circ(w)\,\dd w
        =A^B(p).
\end{aligned}
\]

If $p<\Delta^{\frac{1}{1+\lambda}}$, then
$\widehat L_\Delta^S(p)=L^S_\Delta(p)=0$. Assume
$p\ge\Delta^{\frac{1}{1+\lambda}}$. Conditionally on
$\cF_{m-1}$,
\[
        (Y_{p,a})_{1\le a\le q}
        \quad\text{and}\quad
        (U_J,Y_J)_{J\in\cP_m}
\]
are independent. Therefore,
\[
\begin{aligned}
        \E[\widehat L_\Delta^S(p)\mid\cF_{m-1}]
        &=\sum_{J\in\cP_m}|J|
        \E\left[
        \mathbb I\{U_J\le p-\Delta^{\frac{1}{1+\lambda}}\}
        (Y_J-\widehat F_S(p))
        K_p^S(p-U_J)
        \middle|\cF_{m-1}\right]                                  \\
        &=\sum_{J\in\cP_m}
        \int_{J\cap[0,p-\Delta^{\frac{1}{1+\lambda}}]}
        (F_S(v)-F_S(p))K_p^S(p-v)\,\dd v                           \\
        &=L^S_\Delta(p).
\end{aligned}
\]
If $p>1-\Delta^{\frac{1}{1+\lambda}}$, then
$\widehat L_\Delta^B(p)=L^B_\Delta(p)=0$. Assume
$p\le1-\Delta^{\frac{1}{1+\lambda}}$. Conditionally on
$\cF_{m-1}$,
\[
        (Z_{p,a})_{1\le a\le q}
        \quad\text{and}\quad
        (U_J,Z_J)_{J\in\cP_m}
\]
are independent. Therefore,
\[
\begin{aligned}
        \E[\widehat L_\Delta^B(p)\mid\cF_{m-1}]
        &=\sum_{J\in\cP_m}
        \int_{J\cap[p+\Delta^{\frac{1}{1+\lambda}},1]}
        (F_B^\circ(w)-F_B^\circ(p))K_p^B(w-p)\,\dd w               \\
        &=L^B_\Delta(p).
\end{aligned}
\]

If $\Delta^{\frac{1}{1+\lambda}}>1/2$, or if
$p<\Delta^{\frac{1}{1+\lambda}}$, or if
$p>1-\Delta^{\frac{1}{1+\lambda}}$, then
$\widehat C_\Delta(p)=C_\Delta(p)=0$. Assume
$\Delta^{\frac{1}{1+\lambda}}\le1/2$ and
$\Delta^{\frac{1}{1+\lambda}}\le p
\le1-\Delta^{\frac{1}{1+\lambda}}$.
By Lemma~\ref{lem:partition-geometry}, every $J\in\cP_m$ satisfies
\begin{equation}
\begin{aligned}
        |J|
        &\le {\frac{16\Delta^2}{\CL\ell}}
        \le\Delta^2
        \le\Delta^{\frac{1}{1+\lambda}}.
\end{aligned}
\label{eq:hard-no-cell-both-sides}
\end{equation}
Consequently, no cell $J\in\cP_m$ can intersect both
$[0,p-\Delta^{\frac{1}{1+\lambda}}]$ and
$[p+\Delta^{\frac{1}{1+\lambda}},1]$. In particular, for
$J=K$, the corresponding summand in \eqref{eq:hard-def-Chat} is
zero, and
\[
\begin{aligned}
        &\int_{J\cap[0,p-\Delta^{\frac{1}{1+\lambda}}]}
        \int_{J\cap[p+\Delta^{\frac{1}{1+\lambda}},1]}
        H_\lambda(p-v,w-p)\,\dd w\,\dd v
        =0.
\end{aligned}
\]
For $J\ne K$, the triples $(U_J,Y_J,Z_J)$ and
$(U_K,Y_K,Z_K)$ are independent conditionally on
$\cF_{m-1}$. They are also independent of the local samples $(Y_{p,a},Z_{p,a})$ at
$p$. Therefore,
\[
\begin{aligned}
        \E[\widehat C_\Delta(p)\mid\cF_{m-1}]
        &=\sum_{\substack{J,K\in\cP_m\\J\ne K}}|J||K|          \\
        &\quad\times
        \E\left[
        \mathbb I\{U_J\le p-\Delta^{\frac{1}{1+\lambda}}\}
        \mathbb I\{U_K\ge p+\Delta^{\frac{1}{1+\lambda}}\}
        \right.                                                     \\
        &\hspace{17mm}\left.
        (Y_J-\widehat F_S(p))
        (Z_K-\widehat F_B^\circ(p))
        H_\lambda(p-U_J,U_K-p)
        \middle|\cF_{m-1}\right]                                  \\
        &=\sum_{\substack{J,K\in\cP_m\\J\ne K}}
        \int_{J\cap[0,p-\Delta^{\frac{1}{1+\lambda}}]}
        \int_{K\cap[p+\Delta^{\frac{1}{1+\lambda}},1]}
        (F_S(v)-F_S(p))(F_B^\circ(w)-F_B^\circ(p))                \\
        &\hspace{35mm}\times
        H_\lambda(p-v,w-p)\,\dd w\,\dd v                         \\
        &=C_\Delta(p),
\end{aligned}
\]
where the last equality follows from the vanishing of the diagonal
integrals proved above.

Conditionally on $\cF_{m-1}$, the arrays
\[
        (U_J,Y_J,Z_J)_{J\in\cP_m},
        \qquad
        (Y_{p,a})_{1\le a\le q},
        \qquad
        (Z_{p,a})_{1\le a\le q}
\]
are mutually independent, except for the dependence between $Y_J$
and $Z_J$ within the same partition triple. Moreover, the last two
arrays are independent because seller and buyer valuations are
independent. The dependence within a partition triple does not enter
the products above because the diagonal terms vanish. Hence,
\[
\begin{aligned}
        \E[\widehat{\mathcal I}_{\lambda,\Delta}(p)\mid\cF_{m-1}]
        & =
        \E[\widehat F_S(p)\widehat F_B^\circ(p)\mid\cF_{m-1}]J(p)\\
        &\quad+
        \E[\widehat F_B^\circ(p)\widehat L_\Delta^S(p)\mid\cF_{m-1}]
        +\E[\widehat F_S(p)\widehat L_\Delta^B(p)\mid\cF_{m-1}]  \\
        &\quad+
        \E[\widehat C_\Delta(p)\mid\cF_{m-1}]                    \\
        &=F_S(p)F_B^\circ(p)J(p)
        +F_B^\circ(p)L^S_\Delta(p)
        +F_S(p)L^B_\Delta(p)
        +C_\Delta(p)                                               \\
        &=\mathcal I_{\lambda,\Delta}(p).
\end{aligned}
\]
Finally, \eqref{eq:hard-kernel-double} gives
\[
\begin{aligned}
        J(p)
        &\le\int_0^1\int_0^1H_\lambda(u,z)\,\dd z\,\dd u
        \le{\frac{1}{\lambda^2}},                                    \\
        |L^S_\Delta(p)|
        &\le\int_0^pK_p^S(p-v)\,\dd v
        =J(p),                                                       \\
        |L^B_\Delta(p)|
        &\le\int_p^1K_p^B(w-p)\,\dd w
        =J(p),                                                       \\
        |C_\Delta(p)|
        &\le\int_0^p\int_p^1H_\lambda(p-v,w-p)\,\dd w\,\dd v
        =J(p)                                                        \\
        &\le{\frac{1}{\lambda^2}}
        \le\CL^{1/2}.
\end{aligned}
\]
\end{proof}

We can finally provide the announced concentration of the individual components of the estimator $\widehat{G}_{m}(p)$.

\begin{lemma}[Concentration of the estimators]
\label{lem:algorithmic-concentration}
Fix $m\in\{0,\ldots,M_T\}$ and $\delta\in(0,1)$, condition on a realization of
$\cF_{m-1}$ for which $H_m$ holds, and set
\[
        \Delta=\Delta_m,
        \qquad
        q=q_\Delta(\delta),
        \qquad
        \ell=\ell(\Delta,\delta).
\]
For every fixed $p\in N_m$,
\begin{equation}
\begin{aligned}
        \Prob\left(
        |\widehat F_S(p)-F_S(p)|>\Delta
        \mid\cF_{m-1}\right)
        &\le2e^{-\ell},                                             \\
        \Prob\left(
        |\widehat F_B^\circ(p)-F_B^\circ(p)|>\Delta
        \mid\cF_{m-1}\right)
        &\le2e^{-\ell},                                             \\
        \Prob\left(
        |\widehat A^S(p)-A^S(p)|>\Delta
        \mid\cF_{m-1}\right)
        &\le2e^{-4\ell},                                            \\
        \Prob\left(
        |\widehat A^B(p)-A^B(p)|>\Delta
        \mid\cF_{m-1}\right)
        &\le2e^{-4\ell}.
\end{aligned}
\label{eq:hard-concentration-local-A}
\end{equation}
For every $p\in N_m$,
\begin{equation}
        \Prob\left(
        |\widehat L_\Delta^S(p)-L^S_\Delta(p)|>4\CL^{1/2}\Delta
        \mid\cF_{m-1}\right)
        \le2e^{-4\ell},
\label{eq:hard-concentration-LS}
\end{equation}
\begin{equation}
        \Prob\left(
        |\widehat L_\Delta^B(p)-L^B_\Delta(p)|>4\CL^{1/2}\Delta
        \mid\cF_{m-1}\right)
        \le2e^{-4\ell},
\label{eq:hard-concentration-LB}
\end{equation}
and
\begin{equation}
        \Prob\left(
        |\widehat C_\Delta(p)-C_\Delta(p)|>8\CL^{1/2}\Delta
        \mid\cF_{m-1}\right)
        \le2e^{-4\ell}.
\label{eq:hard-concentration-C}
\end{equation}
\end{lemma}

\begin{proof}
All probabilities and expectations below are conditional on
$\cF_{m-1}$. For fixed $p$, the independent variables used by the
estimators are
\[
        (U_J,Y_J,Z_J)_{J\in\cP_m},
        \qquad
        (Y_{p,a})_{1\le a\le q},
        \qquad
        (Z_{p,a})_{1\le a\le q},
\]
where the partition triples are independent across cells. The last two
arrays are mutually independent and are independent of the partition
triples because seller and buyer valuations are independent.

Since $q\ge\ell/(2\Delta^2)$, Hoeffding's inequality gives
\[
\begin{aligned}
        \Prob(|\widehat F_S(p)-F_S(p)|>\Delta)
        &\le2\exp(-2q\Delta^2)
        \le2e^{-\ell},                                              \\
        \Prob(|\widehat F_B^\circ(p)-F_B^\circ(p)|>\Delta)
        &\le2\exp(-2q\Delta^2)
        \le2e^{-\ell}.
\end{aligned}
\]

For $\widehat A^S(p)$, the random variables
\[
        |J|\mathbb I\{U_J\le p\}Y_J,
        \qquad J\in\cP_m,
\]
are independent and take values in intervals of length $|J|$. By
Lemma~\ref{lem:partition-geometry},
\[
\begin{aligned}
        \sum_{J\in\cP_m}|J|^2
        &\le\max_{J\in\cP_m}|J|
        \sum_{J\in\cP_m}|J|                                      \le{\frac{16\Delta^2}{\CL\ell}}.
\end{aligned}
\]
Hence, by Hoeffding's inequality,
\[
\begin{aligned}
        \Prob(|\widehat A^S(p)-A^S(p)|>\Delta)
        &\le2\exp\left(
        -{\frac{2\Delta^2}{16\Delta^2/(\CL\ell)}}
        \right)                                                     =2\exp\left(-{\frac{\CL\ell}{8}}\right)
        \le2e^{-4\ell}.
\end{aligned}
\]
The same calculation gives the bound for $\widehat A^B(p)$.

If $p<\Delta^{\frac{1}{1+\lambda}}$, then
$\widehat L_\Delta^S(p)=L^S_\Delta(p)=0$, and
\eqref{eq:hard-concentration-LS} is immediate. Assume
$p\ge\Delta^{\frac{1}{1+\lambda}}$. Recall
\eqref{eq:hard-def-LShat}. If the triple $(U_J,Y_J,Z_J)$ changes
from $(u,y,z)$ to $(u',y',z')$, then
$\widehat L_\Delta^S(p)$ changes by at most
\[
\begin{aligned}
        |J|\Bigg[&
        \mathbb I\!\left\{u\le p-\Delta^{\frac{1}{1+\lambda}}\right\}
        |y-\widehat F_S(p)|K_p^S(p-u)                              \\
        &+
        \mathbb I\!\left\{u'\le p-\Delta^{\frac{1}{1+\lambda}}\right\}
        |y'-\widehat F_S(p)|K_p^S(p-u')
        \Bigg]
        \le {\frac{2}{\lambda}}a_J^-(p).
\end{aligned}
\]
Also, changing one local variable $Y_{p,a}$ changes
$\widehat L_\Delta^S(p)$ by at most
\[
\begin{aligned}
        {\frac{1}{q}}
        \sum_{\substack{J\in\cP_m\\
        U_J\le p-\Delta^{\frac{1}{1+\lambda}}}}
        |J|K_p^S(p-U_J)
        \le{\frac{\CL^{1/2}}{q}},
\end{aligned}
\]
by \eqref{eq:hard-discrete-KS-mass}. Therefore, by
\eqref{eq:hard-quadratic-envelope}, the sum of the squares of these
upper bounds on the changes is bounded as
\[
\begin{aligned}
        &\sum_{J\in\cP_m}
        \left({\frac{2a_J^-(p)}{\lambda}}\right)^2
        +\sum_{a=1}^q\left({\frac{\CL^{1/2}}{q}}\right)^2           \qquad\le
        {\frac{4}{\lambda^2}}
        {\frac{32\Delta^2}{\CL^{1/2}\ell}}
        +{\frac{\CL}{q}}                                             \\
        &\qquad\le
        {\frac{32\Delta^2}{\ell}}
        +{\frac{2\CL\Delta^2}{\ell}}                                \qquad\le{\frac{4\CL\Delta^2}{\ell}},
\end{aligned}
\]
where we used \eqref{eq:hard-constant-bounds} and
$q\ge\ell/(2\Delta^2)$. McDiarmid's inequality
(\cite{mcdiarmid1989bounded}; see also
\cite[Chapter~6]{boucheron2013concentration}) and
\eqref{eq:hard-unbiased-LS} give
\[
\begin{aligned}
        \Prob\left(
        |\widehat L_\Delta^S(p)-L^S_\Delta(p)|>4\CL^{1/2}\Delta
        \right)
        &\le2\exp\left(
        -{\frac{2(4\CL^{1/2}\Delta)^2
        }{4\CL\Delta^2/\ell}}
        \right)                                                     \\
        &\le2e^{-8\ell}
        \le2e^{-4\ell}.
\end{aligned}
\]

If $p>1-\Delta^{\frac{1}{1+\lambda}}$, then
$\widehat L_\Delta^B(p)=L^B_\Delta(p)=0$, and
\eqref{eq:hard-concentration-LB} is immediate. Assume
$p\le1-\Delta^{\frac{1}{1+\lambda}}$. Proceeding similarly from
\eqref{eq:hard-def-LBhat}, changing one triple
$(U_J,Y_J,Z_J)$ changes $\widehat L_\Delta^B(p)$ by at most
$2a_J^+(p)/\lambda$, and changing one $Z_{p,a}$ changes it by at
most $\CL^{1/2}/q$, by \eqref{eq:hard-discrete-KB-mass}. Hence,
\[
\begin{aligned}
        &\sum_{J\in\cP_m}
        \left({\frac{2a_J^+(p)}{\lambda}}\right)^2
        +\sum_{a=1}^q\left({\frac{\CL^{1/2}}{q}}\right)^2           \\
        &\qquad\le{\frac{4\CL\Delta^2}{\ell}}.
\end{aligned}
\]
McDiarmid's inequality therefore gives
\[
\begin{aligned}
        \Prob\left(
        |\widehat L_\Delta^B(p)-L^B_\Delta(p)|>4\CL^{1/2}\Delta
        \right)
        &\le2e^{-8\ell}
        \le2e^{-4\ell}.
\end{aligned}
\]

If $\Delta^{\frac{1}{1+\lambda}}>1/2$, or if
$p<\Delta^{\frac{1}{1+\lambda}}$, or if
$p>1-\Delta^{\frac{1}{1+\lambda}}$, then
$\widehat C_\Delta(p)=C_\Delta(p)=0$, and
\eqref{eq:hard-concentration-C} is immediate. Assume
$\Delta^{\frac{1}{1+\lambda}}\le1/2$ and
$\Delta^{\frac{1}{1+\lambda}}\le p
\le1-\Delta^{\frac{1}{1+\lambda}}$.
As shown in \eqref{eq:hard-no-cell-both-sides}, no cell in $\cP_m$
intersects both truncated integration ranges. Hence, changing one
triple $(U_J,Y_J,Z_J)$ can affect either the row indexed by $J$ or
the column indexed by $J$ in \eqref{eq:hard-def-Chat}, but not both.
If $J$ intersects the left truncated range, the change is at most
\[
\begin{aligned}
        |J|\sum_{K\in\cP_m}|K|
        &\mathbb I\!\left\{U_K\ge p+\Delta^{\frac{1}{1+\lambda}}\right\}
        |Z_K-\widehat F_B^\circ(p)|                               \\
        &\times\Bigg[
        \mathbb I\!\left\{u\le p-\Delta^{\frac{1}{1+\lambda}}\right\}
        |y-\widehat F_S(p)|H_\lambda(p-u,U_K-p)                  \\
        &\hspace{12mm}+
        \mathbb I\!\left\{u'\le p-\Delta^{\frac{1}{1+\lambda}}\right\}
        |y'-\widehat F_S(p)|H_\lambda(p-u',U_K-p)
        \Bigg]                                                      \\
        &\le2\CL^{1/2}a_J^-(p),
\end{aligned}
\]
by \eqref{eq:hard-discrete-mass-right}. If $J$ intersects the right
truncated range, the symmetric argument and
\eqref{eq:hard-discrete-mass-left} give the upper bound
$2\CL^{1/2}a_J^+(p)$. Changing one $Y_{p,a}$ changes
$\widehat C_\Delta(p)$ by at most
\[
\begin{aligned}
        {\frac{1}{q}}
        \sum_{\substack{J,K\in\cP_m\\
        U_J\le p-\Delta^{\frac{1}{1+\lambda}}\\
        U_K\ge p+\Delta^{\frac{1}{1+\lambda}}}}
        |J||K|H_\lambda(p-U_J,U_K-p)
        \le {\frac{\CL^{1/2}}{q}},
\end{aligned}
\]
by \eqref{eq:hard-discrete-rectangular-mass}. The same bound holds when
one $Z_{p,a}$ changes. Thus, by
\eqref{eq:hard-quadratic-envelope}, the sum of the squares of these
upper bounds on the changes is bounded as
\[
\begin{aligned}
        &4\CL\left[
        \sum_{J\in\cP_m}(a_J^-(p))^2+
        \sum_{J\in\cP_m}(a_J^+(p))^2
        \right]
        +{\frac{2\CL}{q}}                                            \\
        &\qquad\le
        {\frac{2^8\CL^{1/2}\Delta^2}{\ell}}
        +{\frac{4\CL\Delta^2}{\ell}}                                \\
        &\qquad\le{\frac{8\CL\Delta^2}{\ell}},
\end{aligned}
\]
where the last inequality uses $\CL^{1/2}\ge2^8$. McDiarmid's
inequality and \eqref{eq:hard-unbiased-C} give
\[
\begin{aligned}
        \Prob\left(
        |\widehat C_\Delta(p)-C_\Delta(p)|>8\CL^{1/2}\Delta
        \right)
        &\le2\exp\left(
        -{\frac{2(8\CL^{1/2}\Delta)^2
        }{8\CL\Delta^2/\ell}}
        \right)                                                     \le2e^{-16\ell}
        \le2e^{-4\ell}.
\end{aligned}
\]
\end{proof}

We can now state our main technical lemma. Its proof will follow from Lemma~\ref{lem:algorithmic-concentration}. Assume $H_m$ holds. Then, with high conditional probability, $B_m$ holds. Intuitively, if $H_m$ occurs, Algorithm~\ref{alg:main} enjoys a good enough partitioning of its action space at the beginning of its $m$-th epoch, which ensures, with high conditional probability, $B_m$, meaning that it enjoys good enough estimations $(\widehat G_m(p))_{p\in N_m}$ at the end of the $m$-th epoch.

\begin{lemma}[Conditional estimation]
\label{lem:estimation}
For every $m=0,\ldots,M_T$,
\begin{equation}
        \mathbb I_{H_m}
        \Prob(B_m^c\mid\cF_{m-1})
        \le\delta_m\mathbb I_{H_m}.
\label{eq:hard-Bm-conditional}
\end{equation}
\end{lemma}

\begin{proof}
Fix a realization of $\cF_{m-1}$ for which $H_m$ holds, and write
\[
        \Delta=\Delta_m,
        \qquad
        \ell=\ell(\Delta,\delta_m).
\]
By construction,
\[
        |N_m|=|\mathsf A_m|\le\Delta^{-1}.
\]
If $p<\Delta^{\frac{1}{1+\lambda}}$, then
$\widehat L_\Delta^S(p)=L^S_\Delta(p)=0$. If
$p>1-\Delta^{\frac{1}{1+\lambda}}$, then
$\widehat L_\Delta^B(p)=L^B_\Delta(p)=0$. If $p$ does not belong to
$[\Delta^{\frac{1}{1+\lambda}},1-\Delta^{\frac{1}{1+\lambda}}]$, then
$\widehat C_\Delta(p)=C_\Delta(p)=0$. Therefore, Lemma~\ref{lem:algorithmic-concentration} and a union bound imply that, with conditional failure probability at most
\[
\begin{aligned}
        4\Delta^{-1}e^{-\ell}
        +10\Delta^{-1}e^{-4\ell}
        &\le4\Delta^{-1}e^{-\ell}
        +10\Delta^{-1}e^{-\ell}                                   ={\frac{14}{64}}\delta_m
        \le{\frac{\delta_m}{2}},
\end{aligned}
\]
the following inequalities hold simultaneously for all $p\in N_m$:
\begin{equation}
\begin{gathered}
        |\widehat F_S(p)-F_S(p)|\le\Delta,
        \qquad
        |\widehat F_B^\circ(p)-F_B^\circ(p)|\le\Delta,            \\
        |\widehat A^S(p)-A^S(p)|\le\Delta,
        \qquad
        |\widehat A^B(p)-A^B(p)|\le\Delta,                         \\
        |\widehat L_\Delta^S(p)-L^S_\Delta(p)|\le4\CL^{1/2}\Delta,
        \qquad
        |\widehat L_\Delta^B(p)-L^B_\Delta(p)|\le4\CL^{1/2}\Delta,      \\
        |\widehat C_\Delta(p)-C_\Delta(p)|\le8\CL^{1/2}\Delta.
\end{gathered}
\label{eq:hard-simultaneous-component-errors}
\end{equation}
Here $e^{-\ell}=\Delta\delta_m/64$.

On the event \eqref{eq:hard-simultaneous-component-errors}, using
$A^S(p),A^B(p),\widehat A^S(p),\widehat A^B(p)\in[0,1]$,
\[
\begin{aligned}
        |2^{1-1/\lambda}(\widehat G_1(p)-G_1(p))|
        &\le{\frac{1}{2}}
        \left|
        \widehat F_B^\circ(p)\widehat A^S(p)
        -F_B^\circ(p)A^S(p)
        \right|                                                     \\
        &\quad+{\frac{1}{2}}
        \left|
        \widehat F_S(p)\widehat A^B(p)
        -F_S(p)A^B(p)
        \right|                                                     \\
        &\le{\frac{1}{2}}(2\Delta+2\Delta)
        =2\Delta.
\end{aligned}
\]
By \eqref{eq:hard-deterministic-component-bounds},
\begin{align*}
&
    |\widehat{\mathcal I}_{\lambda,\Delta}(p)
    -\mathcal I_{\lambda,\Delta}(p)|
\\
& \qquad \le 
    J(p)\left|
    \widehat F_S(p)\widehat F_B^\circ(p)
    -F_S(p)F_B^\circ(p)
    \right|
    +
    \left|
    \widehat F_B^\circ(p)\widehat L_\Delta^S(p)
    -F_B^\circ(p)L^S_\Delta(p)
    \right|
\\
& \qquad \qquad
    +\left|
    \widehat F_S(p)\widehat L_\Delta^B(p)
    -F_S(p)L^B_\Delta(p)
    \right|
    +|\widehat C_\Delta(p)-C_\Delta(p)|
\\
& \qquad \le
    2\CL^{1/2}\Delta
    +\left(4\CL^{1/2}\Delta
    +\Delta|L^S_\Delta(p)|\right)                     +
    \left(4\CL^{1/2}\Delta
    +\Delta|L^B_\Delta(p)|\right)
    +8\CL^{1/2}\Delta
\\
& \qquad
    \le20\CL^{1/2}\Delta.
\end{align*}
Combining this inequality with \eqref{eq:hard-truncation-I} gives
\[
\begin{aligned}
        |\widehat G_m(p)-G_\lambda(p)|
        &\le
        |2^{1-1/\lambda}(\widehat G_1(p)-G_1(p))|                  +
        |\widehat{\mathcal I}_{\lambda,\Delta}(p)
        -\mathcal I_{\lambda,\Delta}(p)|
        +|\mathcal I_{\lambda,\Delta}(p)-\mathcal I_\lambda(p)| \\
        &\le2\Delta+21\CL^{1/2}\Delta                            \le32\CL^{1/2}\Delta
        \le{\frac{\CL}{16}}\Delta.
\end{aligned}
\]
Thus $B_m$ holds outside an event of conditional probability at most
$\delta_m/2\le\delta_m$, proving \eqref{eq:hard-Bm-conditional}.
\end{proof}

Finally, we can show that the good event $\mathcal E_{M_T}$ holds with high probability. Furthermore, this good event implies all the events $H_0,\dots,H_{M_T+1}$ that the partitions of the action space are accurate.

\begin{lemma}[The good event and the active and rejected regions]
\label{lem:induction}
For every $m=0,\ldots,M_T+1$,
\begin{equation}
        \mathcal E_{m-1}\Longrightarrow H_m.
\label{eq:hard-Em-implies-Hm}
\end{equation}
Moreover,
\begin{equation}
        \Prob(\mathcal E_{M_T}^c)
        \le{\frac{\pi^2}{6T^2}}.
\label{eq:hard-good-event-probability}
\end{equation}
\end{lemma}

\begin{proof}
At epoch $0$, $A_0=[0,1]$, there are no rejected layers, and
\eqref{eq:hard-G-bounded} together with $3\CL\ge1$ gives
\[
        \mathcal E_{-1}=\Omega\Longrightarrow H_0.
\]
Assume that $m\le M_T$ and
$\mathcal E_{m-1}\Longrightarrow H_m$. Work on
$\mathcal E_m=\mathcal E_{m-1}\cap B_m$.
Let $p^\star$ be a maximizer of $G_\lambda$.
Let $I^\star\in\mathsf A_m$ contain $p^\star$, and set
$p_m^\star=c(I^\star)$. Then
\[
\begin{aligned}
        G_\lambda^\star-G_\lambda(p_m^\star)
        &\le{\frac{\CL}{16}}|p^\star-p_m^\star|
        \le{\frac{\CL}{32}}\Delta_m,                                 \\
        \widehat G_m(p_m^\star)
        &\ge G_\lambda^\star-{\frac{3\CL}{32}}\Delta_m,             \\
        \widehat M_m
        &\le G_\lambda^\star+{\frac{\CL}{16}}\Delta_m,              \\
        \widehat M_m-\widehat G_m(p_m^\star)
        &\le{\frac{5\CL}{32}}\Delta_m
        <\CL\Delta_m.
\end{aligned}
\]
Thus every interval containing a maximizer is retained, and every maximizer belongs to $A_{m+1}$.

Let $q\in A_{m+1}$, let $I\in\mathsf A_m^+$ be its parent interval, and set $p=c(I)$. Since $|q-p|\le\Delta_m/2$,
\[
\begin{aligned}
        G_\lambda^\star-G_\lambda(q)
        &\le
        G_\lambda^\star-\widehat G_m(p_m^\star)
        +\widehat G_m(p_m^\star)-\widehat M_m                     \\
        &\quad+
        \widehat M_m-\widehat G_m(p)
        +\widehat G_m(p)-G_\lambda(p)
        +G_\lambda(p)-G_\lambda(q)                                \\
        &\le
        {\frac{3\CL}{32}}\Delta_m
        +\CL\Delta_m
        +{\frac{\CL}{16}}\Delta_m
        +{\frac{\CL}{32}}\Delta_m                                    \\
        &={\frac{19\CL}{16}}\Delta_m
        \le3\CL\Delta_{m+1}.
\end{aligned}
\]
Hence $(\mathrm{Act}_{m+1})$ holds.

Let $x\in R_m$, let $I\in\mathsf R_m$ contain $x$, and set
$p=c(I)$. Since $I$ is rejected and $|x-p|\le\Delta_m/2$,
\[
\begin{aligned}
        G_\lambda^\star-G_\lambda(x)
        &=G_\lambda^\star-\widehat M_m
        +\widehat M_m-\widehat G_m(p)                              \\
        &\quad+
        \widehat G_m(p)-G_\lambda(p)
        +G_\lambda(p)-G_\lambda(x)                                \\
        &>{\frac{-\CL}{16}}\Delta_m
        +\CL\Delta_m
        -{\frac{\CL}{16}}\Delta_m
        -{\frac{\CL}{32}}\Delta_m                                    \\
        &={\frac{27\CL}{32}}\Delta_m
        \ge{\frac{\CL}{2}}\Delta_m.
\end{aligned}
\]
Moreover, $R_m\subset A_m$, so $H_m$ gives
\[
        G_\lambda^\star-G_\lambda(x)
        \le3\CL\Delta_m.
\]
Thus $(\mathrm{Rej}_{<m+1})$ holds. We have proved
\[
        H_m\cap B_m\Longrightarrow H_{m+1},
\]
and therefore
\[
        \mathcal E_m\Longrightarrow H_{m+1}.
\]
This proves \eqref{eq:hard-Em-implies-Hm} by induction.

Since $\mathcal E_{m-1}$ and $H_m$ are
$\cF_{m-1}$-measurable and
$\mathcal E_{m-1}\Longrightarrow H_m$,
\[
\begin{aligned}
        \Prob(\mathcal E_{m-1}\cap B_m^c)
        &=\E\left[
        \mathbb I_{\mathcal E_{m-1}}
        \Prob(B_m^c\mid\cF_{m-1})
        \right]                                                     \le\E\left[
        \mathbb I_{H_m}
        \Prob(B_m^c\mid\cF_{m-1})
        \right]                                                     \le\delta_m.
\end{aligned}
\]
Consequently,
\[
\begin{aligned}
        \Prob(\mathcal E_{M_T}^c)
        &\le
        \sum_{m=0}^{M_T}
        \Prob(\mathcal E_{m-1}\cap B_m^c)                         \le\sum_{m=0}^{M_T}\delta_m
        \le{\frac{1}{T^2}}\sum_{m=0}^\infty{\frac{1}{(m+1)^2}}
        ={\frac{\pi^2}{6T^2}}.
\end{aligned}
\]
\end{proof}

\subsection{Regret analysis under the good event}
\label{section:regret:under:good:event}

We first bound the realized (pseudo-) regret when Algorithm~\ref{alg:main} is run under $\mathcal E_{M_T}$. Our analysis will use Lemma~\ref{lem:induction}, which ensures that $H_m$ holds true for every $m=0,\ldots,M_T+1$ on this event.

We now come to the realized (pseudo-) regret. Define the random (pseudo-) regret
\[
\mathcal R_T 
=
\sum_{t=1}^T 
\left( 
G_{\lambda}^\star
-
G_{\lambda}(P_t)
\right).
\]
It follows that $\mathbb E[\mathcal R_T]=R_T:=R_\lambda^\nu(\alpha,T)$, where $G_\lambda^\star-G_\lambda(P_t)$ intuitively measures the (pseudo) regret suffered by Algorithm~\ref{alg:main} in the $t$th round. We need the following lemma. It intuitively upper-bounds the total regret suffered in rounds belonging to the estimation epochs of Algorithm~\ref{alg:main}.

\begin{lemma}[Exploration budget and exploration regret]
\label{lem:exploration-budget-regret}
Assume $\overline\Delta_T\le1$. The epochs
$0,\ldots,M_T$ use at most $T/2$ rounds (deterministically, regardless of whether $\mathcal E_{M_T}$ holds). On
$\mathcal E_{M_T}$, the regret incurred during these epochs is at most
\begin{equation}
        {\frac{2^{19}}{3}}\CL^2\log(64T^5)\Delta_{M_T}^{-2}.
\label{eq:hard-exploration-regret}
\end{equation}
\end{lemma}

\begin{proof}
We first prove the deterministic bound on the number of exploration rounds.
For $m\le M_T$,
$\Delta_m\ge\overline\Delta_T\ge T^{-1}$ and $m+1\le T$. Hence
\begin{equation}
\begin{aligned}
        \ell(\Delta_m,\delta_m)
        &=\log\left({\frac{64T^2(m+1)^2}{\Delta_m}}\right)
        \le\log(64T^5).
\end{aligned}
\label{eq:hard-ell-LT}
\end{equation}
Using Lemma~\ref{lem:partition-geometry} and
\eqref{eq:hard-rq-upper-bounds}, the number of rounds used at epoch $m$ satisfies
\[
\begin{aligned}
        |\cP_m|
        &\le{\frac{\CL}{8}}\Delta_m^{-3}
        \ell(\Delta_m,\delta_m)                                   \le{\frac{\CL}{8}}\Delta_m^{-3}\log(64T^5),                \\
        |N_m|q_{\Delta_m}(\delta_m)
        &\le\Delta_m^{-1}\Delta_m^{-2}
        \ell(\Delta_m,\delta_m)                                   \le{\frac{\CL}{8}}\Delta_m^{-3}\log(64T^5).
\end{aligned}
\]
Therefore,
\[
\begin{aligned}
        \sum_{m=0}^{M_T}
        \left[|\cP_m|+|N_m|q_{\Delta_m}(\delta_m)\right]
        &\le
        {\frac{\CL}{4}}\log(64T^5)
        \sum_{m=0}^{M_T}\Delta_m^{-3}                              \\
        &\le{\frac{2}{7}}\CL\log(64T^5)\Delta_{M_T}^{-3}            \\
        &\le{\frac{2}{7}}\CL\log(64T^5)\overline\Delta_T^{-3}      \\
        &={\frac{T}{14}}
        <{\frac{T}{2}}.
\end{aligned}
\]

We next bound the regret at epoch $m$. We work under
$\mathcal E_{M_T}$ from now on. By the definition of $n_m$ and
\eqref{eq:hard-constant-bounds},
\begin{equation}
\begin{aligned}
        n_m
        &=\left\lceil
        {\frac{\CL\ell(\Delta_m,\delta_m)}{16\Delta_m^2}}
        \right\rceil
        \le{\frac{\CL\ell(\Delta_m,\delta_m)}{8\Delta_m^2}}.
\end{aligned}
\label{eq:hard-nm-upper}
\end{equation}
The partition samples contained in active intervals incur regret at most
\[
\begin{aligned}
        {\frac{|A_m|}{\Delta_m}}
        n_m\,3\CL\Delta_m
        =3\CL|A_m|n_m.
\end{aligned}
\]
For $j<m$, the partition samples contained in $R_j$ incur regret at
most
\[
\begin{aligned}
        {\frac{|R_j|}{\Delta_j}}
        n_m\,3\CL\Delta_j
        =3\CL|R_j|n_m.
\end{aligned}
\]
Using \eqref{eq:hard-layer-partition} and \eqref{eq:hard-nm-upper}, the
regret of all partition samples at epoch $m$ is therefore at most
\[
\begin{aligned}
        3\CL n_m
        \left(|A_m|+\sum_{j=0}^{m-1}|R_j|\right)
        &=3\CL n_m                                                   \le{\frac{3}{8}}\CL^2\Delta_m^{-2}
        \ell(\Delta_m,\delta_m).
\end{aligned}
\]
The local samples are posted at points of $N_m\subset A_m$, and hence
their regret is at most
\[
\begin{aligned}
        3\CL\Delta_m
        |N_m|q_{\Delta_m}(\delta_m)
        &\le3\CL\Delta_m^{-2}\ell(\Delta_m,\delta_m)             \le3\CL^2\Delta_m^{-2}\ell(\Delta_m,\delta_m).
\end{aligned}
\]
Thus, using \eqref{eq:hard-ell-LT}, the total exploration regret is at
most
\[
\begin{aligned}
        \sum_{m=0}^{M_T}
        4\CL^2\Delta_m^{-2}\ell(\Delta_m,\delta_m)
        &\le4\CL^2\log(64T^5)
        \sum_{m=0}^{M_T}\Delta_m^{-2}                              \\
        &\le{\frac{16}{3}}\CL^2\log(64T^5)\Delta_{M_T}^{-2}          \\
        &\le{\frac{2^{19}}{3}}\CL^2
        \log(64T^5)\Delta_{M_T}^{-2},
\end{aligned}
\]
which proves \eqref{eq:hard-exploration-regret}.
\end{proof}

The next lemma bounds the sum of regrets suffered during the times $t$ corresponding to the exploitation rounds of Algorithm \ref{alg:main} (Line~\ref{line:exploit}).

\begin{lemma}[Exploitation regret]
\label{lem:exploitation-regret}
On $\mathcal E_{M_T}$, every exploitation round has regret at most
\begin{equation}
        3\CL\overline\Delta_T,
\label{eq:hard-one-round-exploitation}
\end{equation}
and the total exploitation regret is at most
\begin{equation}
        3T\CL\overline\Delta_T.
\label{eq:hard-total-exploitation}
\end{equation}
\end{lemma}

\begin{proof}
On $\mathcal E_{M_T}$, Lemma~\ref{lem:induction} gives $H_{M_T+1}$. Since the exploitation prices belong to $A_{M_T+1}$ and
$\Delta_{M_T+1}<\overline\Delta_T$,
\[
        G_\lambda^\star-G_\lambda(p)
        \le3\CL\Delta_{M_T+1}
        \le3\CL\overline\Delta_T
        \qquad(p\in A_{M_T+1}).
\]
Multiplication by the number of remaining rounds, which is at most $T$, gives \eqref{eq:hard-total-exploitation}.
\end{proof}

We can now bound deterministically the realized pseudo-regret $\mathcal R_T$ under the good event. 

\begin{lemma}[Total realized regret on the good event]
\label{lem:total-good-regret}
Assume $\overline\Delta_T\le1$. On $\mathcal E_{M_T}$,
\begin{equation}
\begin{aligned}
        \mathcal R_T
        \le{}&
        {\frac{2^{19}}{3}}\CL^2(4\CL)^{-2/3}
        T^{2/3}\bigl(\log(64T^5)\bigr)^{1/3}                     +3\CL(4\CL)^{1/3}
        T^{2/3}\bigl(\log(64T^5)\bigr)^{1/3}.
\end{aligned}
\label{eq:hard-total-good-regret}
\end{equation}
\end{lemma}

\begin{proof}
By Lemmas~\ref{lem:exploration-budget-regret} and
\ref{lem:exploitation-regret},
\[
\begin{aligned}
        \mathcal R_T
        &\le
        {\frac{2^{19}}{3}}\CL^2\log(64T^5)\Delta_{M_T}^{-2}
        +3T\CL\overline\Delta_T                                  \\
        &\le
        {\frac{2^{19}}{3}}\CL^2\log(64T^5)\overline\Delta_T^{-2}
        +3T\CL\overline\Delta_T                                  \\
        &=
        {\frac{2^{19}}{3}}\CL^2(4\CL)^{-2/3}
        T^{2/3}\bigl(\log(64T^5)\bigr)^{1/3}                      +
        3\CL(4\CL)^{1/3}
        T^{2/3}\bigl(\log(64T^5)\bigr)^{1/3}.
\end{aligned}
\]
\end{proof}

\subsection{Proof of Theorem~\ref{thm:hard-main-statement}}
\label{subsection:conclusion:theorem}

\begin{proof}[Proof of Theorem~\ref{thm:hard-main-statement}]
Assume
\[
        \overline\Delta_T
        =\left({\frac{4\CL\log(64T^5)}{T}}\right)^{1/3}
        \le1.
\]
Lemma~\ref{lem:exploration-budget-regret} shows that the estimation epochs end before time $T/2$. By
Lemma~\ref{lem:total-good-regret}, on $\mathcal E_{M_T}$,
\[
\begin{aligned}
        \mathcal R_T
        \le
        \left[
        {\frac{2^{19}}{3}}\CL^2(4\CL)^{-2/3}
        +3\CL(4\CL)^{1/3}
        \right]
        T^{2/3}\bigl(\log(64T^5)\bigr)^{1/3}.
\end{aligned}
\]
On $\mathcal E_{M_T}^c$, $\mathcal R_T\le T$, and
\eqref{eq:hard-good-event-probability} gives
\[
\begin{aligned}
        R_T
        &=\E[\mathcal R_T]                                         \\
        &\le
        \left[
        {\frac{2^{19}}{3}}\CL^2(4\CL)^{-2/3}
        +3\CL(4\CL)^{1/3}
        \right]
        T^{2/3}\bigl(\log(64T^5)\bigr)^{1/3}
        +T\Prob(\mathcal E_{M_T}^c)                                \\
        &\le
        \left[
        {\frac{2^{19}}{3}}\CL^2(4\CL)^{-2/3}
        +3\CL(4\CL)^{1/3}
        \right]
        T^{2/3}\bigl(\log(64T^5)\bigr)^{1/3}
        +{\frac{\pi^2}{6T}}.
\end{aligned}
\]
\end{proof}

\section{Proofs for Section~\ref{subsection:lambda:half:one}}
\label{app:easy-upper}

We first provide bounds on the kernel $H_\lambda$ for
$1/2\le\lambda < 1$. 

\begin{lemma}[Kernel bounds]
\label{lem:easy-kernel-bounds}
Let $1/2\le\lambda < 1$. For $0<u\le z\le1$,
\begin{equation}
        H_\lambda(u,z)
        \le
        (1-\lambda)2^{-1/\lambda}
        u^{\lambda-1}z^{-\lambda},
\label{eq:easy-kernel-u-small}
\end{equation}
and, for $0<z\le u\le1$,
\begin{equation}
        H_\lambda(u,z)
        \le
        (1-\lambda)2^{-1/\lambda}
        u^{-\lambda}z^{\lambda-1}.
\label{eq:easy-kernel-z-small}
\end{equation}
Consequently, for every $u,z\in(0,1]$,
\begin{equation}
        \int_0^1H_\lambda(u,z)\,\dd z
        \le u^{\lambda-1},
        \qquad
        \int_0^1H_\lambda(u,z)\,\dd u
        \le z^{\lambda-1}.
\label{eq:easy-one-sided-integrals}
\end{equation}
\end{lemma}

\begin{proof}
If $u\le z$, then
\[
        M_\lambda(u,z)
        =
        \left(\frac{u^\lambda+z^\lambda}{2}\right)^{1/\lambda}
        \ge
        2^{-1/\lambda}z.
\]
Since $1-2\lambda\le0$, Lemma~\ref{lem:main-reconstruction} gives
\[
\begin{aligned}
        H_\lambda(u,z)
        &=
        \frac{1-\lambda}{4}
        M_\lambda(u,z)^{1-2\lambda}
        u^{\lambda-1}z^{\lambda-1}\\
        &\le
        \frac{1-\lambda}{4}
        \left(2^{-1/\lambda}z\right)^{1-2\lambda}
        u^{\lambda-1}z^{\lambda-1}\\
        &=
        (1-\lambda)2^{-1/\lambda}
        u^{\lambda-1}z^{-\lambda},
\end{aligned}
\]
which proves \eqref{eq:easy-kernel-u-small}. The proof of
\eqref{eq:easy-kernel-z-small} is symmetric.

Set $c_\lambda:=(1-\lambda)2^{-1/\lambda}$. By
\eqref{eq:easy-kernel-u-small} and
\eqref{eq:easy-kernel-z-small},
\[
\begin{aligned}
        \int_0^1H_\lambda(u,z)\,\dd z
        &\le
        c_\lambda
        \left[
        u^{-\lambda}\int_0^u z^{\lambda-1}\,\dd z
        +
        u^{\lambda-1}\int_u^1z^{-\lambda}\,\dd z
        \right]\\
        &=
        2^{-1/\lambda}
        \left[
        \frac{1-\lambda}{\lambda}
        +
        u^{\lambda-1}-1
        \right]\\
        &\le
        u^{\lambda-1},
\end{aligned}
\]
where the last inequality follows from
$(1-\lambda)/\lambda\le1$ and $2^{-1/\lambda}\le1$.
The second inequality in \eqref{eq:easy-one-sided-integrals} follows
by symmetry.
\end{proof}

The first inequality in Lemma~\ref{lem:hard-lipschitz} was proved
using only $0 < \lambda \le 1$. Therefore, the same proof gives, for
$1/2\le\lambda\le1$,
\begin{equation}
        |G_\lambda(p)-G_\lambda(q)|
        \le
        L\left(2+\frac1\lambda\right)|p-q|
        \le4L|p-q|,
        \qquad p,q\in[0,1].
\label{eq:easy-G-lipschitz}
\end{equation}
Also, recall that $0\le G_\lambda(p)\le1$.

For $q\in[0,1]$, define the truncated version of $\mathcal I_{\lambda}(q)$,
\begin{equation}
\mathcal I_{\lambda,n}(q)
:=
\begin{cases}
\displaystyle
\int_0^{q-1/n}\int_{q+1/n}^1
H_\lambda(q-v,w-q)F_S(v)F_B^\circ(w)
\,\dd w\,\dd v,
&
\text{if }\dfrac1n\le q\le1-\dfrac1n,
\smallskip\\
0,
&
\text{otherwise.}
\end{cases}
\label{eq:easy-I-truncated}
\end{equation}

The next lemma first bounds the truncation error. Then, it provides bounds on sums of squares that will then be used in the proof of Lemma \ref{lem:easy-scouting-concentration}. The latter lemma provides the concentration of 
$\widehat A^S(q)$,
$\widehat A^B(q)$
and $\widehat{\mathcal I}_{\lambda,n}(q)$, by means of Hoeffding's and McDiarmid's inequalities.

\begin{lemma}[Truncation bias and partition bounds]
\label{lem:easy-partition-bounds}
Let $n\ge1$, and let $\mathcal P_n$ be defined in
\eqref{eq:easy-partition}. For every $q\in[0,1]$,
\begin{equation}
        \left|
        \mathcal I_\lambda(q)-\mathcal I_{\lambda,n}(q)
        \right|
        \le\frac4{\sqrt n}.
\label{eq:easy-truncation-bound}
\end{equation}

For $J\in\mathcal P_n$, set
\begin{equation}
a_J^-(q)
:=
\begin{cases}
\displaystyle
|J|
\sup_{\substack{x\in J\\x\le q-1/n}}
(q-x)^{\lambda-1},
&
\text{if }J\cap[0,q-1/n]\ne\emptyset,
\smallskip\\
0,
&
\text{otherwise,}
\end{cases}
\label{eq:easy-a-minus}
\end{equation}
and
\begin{equation}
a_J^+(q)
:=
\begin{cases}
\displaystyle
|J|
\sup_{\substack{x\in J\\x\ge q+1/n}}
(x-q)^{\lambda-1},
&
\text{if }J\cap[q+1/n,1]\ne\emptyset,
\smallskip\\
0,
&
\text{otherwise.}
\end{cases}
\label{eq:easy-a-plus}
\end{equation}
Then
\begin{equation}
        \sum_{J\in\mathcal P_n}
        \left(a_J^-(q)\right)^2
        \le
        \frac{2(1+\log n)}n,
        \qquad
        \sum_{J\in\mathcal P_n}
        \left(a_J^+(q)\right)^2
        \le
        \frac{2(1+\log n)}n.
\label{eq:easy-a-square-bounds}
\end{equation}

Moreover, for arbitrary points $u_J\in J$,
$J\in\mathcal P_n$, if $u_J\le q-1/n$, then
\begin{equation}
        \sum_{J'\in\mathcal P_n}
        |J'|\mathbb I\{u_{J'}\ge q+1/n\}
        H_\lambda(q-u_J,u_{J'}-q)
        \le
        2(q-u_J)^{\lambda-1}.
\label{eq:easy-discrete-right-mass}
\end{equation}
Similarly, if $u_J\ge q+1/n$, then
\begin{equation}
        \sum_{J'\in\mathcal P_n}
        |J'|\mathbb I\{u_{J'}\le q-1/n\}
        H_\lambda(q-u_{J'},u_J-q)
        \le
        2(u_J-q)^{\lambda-1}.
\label{eq:easy-discrete-left-mass}
\end{equation}
\end{lemma}

\begin{proof}
We first prove \eqref{eq:easy-truncation-bound}. If $q<1/n$, then
$\mathcal I_{\lambda,n}(q)=0$, and
\[
\begin{aligned}
        \mathcal I_\lambda(q)
        &\le
        \int_0^q\int_0^1
        H_\lambda(u,z)\,\dd z\,\dd u\le
        \int_0^q u^{\lambda-1}\,\dd u=
        \frac{q^\lambda}{\lambda}\le
        \frac{n^{-\lambda}}\lambda\le
        \frac2{\sqrt n}.
\end{aligned}
\]
The same argument gives
$\mathcal I_\lambda(q)\le2/\sqrt n$ when
$q>1-1/n$. If $1/n\le q\le1-1/n$, then
\[
\begin{aligned}
        0
        &\le
        \mathcal I_\lambda(q)-\mathcal I_{\lambda,n}(q)\le
        \int_0^{1/n}\int_0^1
        H_\lambda(u,z)\,\dd z\,\dd u
        +
        \int_0^1\int_0^{1/n}
        H_\lambda(u,z)\,\dd z\,\dd u\\
        &\le
        2\int_0^{1/n}u^{\lambda-1}\,\dd u=
        \frac{2n^{-\lambda}}\lambda\le
        \frac4{\sqrt n}.
\end{aligned}
\]
This proves \eqref{eq:easy-truncation-bound}.

We prove the first inequality in \eqref{eq:easy-a-square-bounds};
the second is symmetric. The cells intersecting $[0,q-1/n]$ can
be grouped according to their distance from $q$. For every integer
$k\ge1$, at most two cells contain points $x$ such that
\[
        \frac{k}{n}\le q-x<\frac{k+1}{n}.
\]
Consequently,
\[
\begin{aligned}
        \sum_{J\in\mathcal P_n}
        \left(a_J^-(q)\right)^2
        &\le
        \frac2{n^2}
        \sum_{k=1}^n
        \left(\frac{k}{n}\right)^{2\lambda-2}\\
        &=
        2n^{-2\lambda}
        \sum_{k=1}^n k^{2\lambda-2}.
\end{aligned}
\]
Set $\beta:=2\lambda-1\in[0,1)$. If $\beta=0$, then
\[
        \sum_{k=1}^n k^{2\lambda-2}
        =
        \sum_{k=1}^n\frac1k
        \le1+\log n.
\]
If $\beta>0$, then
\[
\begin{aligned}
        n^{-\beta}
        \sum_{k=1}^n k^{\beta-1}
        &\le
        n^{-\beta}
        \left[
        1+\int_1^n x^{\beta-1}\,\dd x
        \right]\\
        &=
        n^{-\beta}
        +
        \frac{1-n^{-\beta}}\beta\\
        &\le
        1+\log n,
\end{aligned}
\]
where the last inequality follows from $1-e^{-\beta\log n}\le \beta\log n$. Since
$n^{-2\lambda}=n^{-1}n^{-\beta}$, the first inequality in
\eqref{eq:easy-a-square-bounds} follows.

It remains to prove \eqref{eq:easy-discrete-right-mass}. Fix
$u:=q-u_J\ge1/n$, and set
\[
\psi_u(z)
:=
(1-\lambda)2^{-1/\lambda}
\left[
        u^{-\lambda}z^{\lambda-1}\mathbb I\{z\le u\}
        +
        u^{\lambda-1}z^{-\lambda}\mathbb I\{z>u\}
\right].
\]
By \eqref{eq:easy-kernel-u-small} and
\eqref{eq:easy-kernel-z-small},
\[
        H_\lambda(u,z)\le\psi_u(z),
        \qquad z>0.
\]
The function $\psi_u$ is non-increasing. For each point
$z_{J'}:=u_{J'}-q\ge1/n$, consider the interval
$[z_{J'}-1/n,z_{J'}]$. Since there is one point $u_{J'}$ in
each cell of length $1/n$, these intervals overlap at most twice,
up to their endpoints. Hence,
\[
\begin{aligned}
        &\sum_{J'\in\mathcal P_n}
        |J'|\mathbb I\{u_{J'}\ge q+1/n\}
        H_\lambda(u,u_{J'}-q)\\
        &\qquad\le
        \frac1n\sum_{J':\,z_{J'}\ge1/n}\psi_u(z_{J'})\\
        &\qquad\le
        \sum_{J':\,z_{J'}\ge1/n}
        \int_{z_{J'}-1/n}^{z_{J'}}\psi_u(z)\,\dd z\\
        &\qquad\le
        2\int_0^1\psi_u(z)\,\dd z\\
        &\qquad\le
        2u^{\lambda-1},
\end{aligned}
\]
where the last inequality follows from the computation in the proof
of Lemma~\ref{lem:easy-kernel-bounds}. This proves
\eqref{eq:easy-discrete-right-mass}; the proof of
\eqref{eq:easy-discrete-left-mass} is symmetric.
\end{proof}

We can now provide the concentration of 
$\widehat A^S(q)$,
$\widehat A^B(q)$
and $\widehat{\mathcal I}_{\lambda,n}(q)$. 

\begin{lemma}[Uniform accuracy of the scouting quantities]
\label{lem:easy-scouting-concentration}
Let
\begin{equation}
        b_T:=\log(8KT^2),
        \qquad
        \varepsilon_A
        :=
        \sqrt{\frac{b_T}{2n}},
        \qquad
        \varepsilon_I
        :=
        8\sqrt{
        \frac{(1+\log n)b_T}{n}
        }.
\label{eq:easy-epsilons}
\end{equation}
With probability at least $1-T^{-2}$, simultaneously for every
$q\in\mathcal G_K$,
\begin{equation}
        |\widehat A^S(q)-A^S(q)|
        \le\varepsilon_A,
        \qquad
        |\widehat A^B(q)-A^B(q)|
        \le\varepsilon_A,
\label{eq:easy-A-concentration}
\end{equation}
and
\begin{equation}
        \left|
        \widehat{\mathcal I}_{\lambda,n}(q)
        -
        \mathcal I_{\lambda,n}(q)
        \right|
        \le\varepsilon_I.
\label{eq:easy-I-concentration}
\end{equation}
\end{lemma}

\begin{proof}
For every $q\in\mathcal G_K$, independence across the cells and
uniform sampling inside each cell give
\[
        \E[\widehat A^S(q)]
        =
        \sum_{J\in\mathcal P_n}
        |J|
        \E[
        \mathbb I\{U_J\le q\}Y_J
        ]=
        \sum_{J\in\mathcal P_n}
        \int_J
        \mathbb I\{v\le q\}F_S(v)\,\dd v=
        \int_0^qF_S(v)\,\dd v=
        A^S(q).
\]
Similarly,
\[
        \E[\widehat A^B(q)]=A^B(q).
\]
Every summand in $\widehat A^S(q)$ and
$\widehat A^B(q)$ belongs to an interval of length $1/n$, and
\[
\begin{aligned}
        \sum_{J\in\mathcal P_n}\left(\frac1n\right)^2
        &=
        \frac1n.
\end{aligned}
\]
Hoeffding's inequality therefore gives
\[
\begin{aligned}
        \Prob\left(
        |\widehat A^S(q)-A^S(q)|>\varepsilon_A
        \right)
        &\le
        2\exp\left(-2n\varepsilon_A^2\right)
        =
        2e^{-b_T},\\
        \Prob\left(
        |\widehat A^B(q)-A^B(q)|>\varepsilon_A
        \right)
        &\le
        2e^{-b_T}.
\end{aligned}
\]

Let us now turn to $ \widehat{\mathcal I}_{\lambda,n}(q)
        -
        \mathcal I_{\lambda,n}(q)$.
If $q<1/n$ or $q>1-1/n$, both
$\widehat{\mathcal I}_{\lambda,n}(q)$ and
$\mathcal I_{\lambda,n}(q)$ are zero. Assume now that
$1/n\le q\le1-1/n$. If $J=J'$, the corresponding summand in
\eqref{eq:easy-Ihat} is zero, because $U_J$ cannot satisfy both
\[
        U_J\le q-\frac1n
        \qquad\text{and}\qquad
        U_J\ge q+\frac1n.
\]
The corresponding integral
in \eqref{eq:easy-I-truncated}
over $J\times J$ is also zero. Indeed,
if $v,w\in J$, then $|w-v|\le1/n$, whereas
$v\le q-1/n$ and $w\ge q+1/n$ would imply
$w-v\ge2/n$.

For $J\ne J'$, the observations obtained in the two cells come
from independent rounds. Consequently,
\[
\begin{aligned}
        \E[
        \widehat{\mathcal I}_{\lambda,n}(q)
        ]
        &=
        \sum_{\substack{J,J'\in\mathcal P_n\\J\ne J'}}
        |J||J'|
        \E\Bigg[
        \mathbb I\left\{U_J\le q-\frac1n\right\}
        \mathbb I\left\{U_{J'}\ge q+\frac1n\right\}
        Y_JZ_{J'}\\
        &\hspace{4.2cm}\times
        H_\lambda(q-U_J,U_{J'}-q)
        \Bigg]\\
        &=
        \sum_{\substack{J,J'\in\mathcal P_n\\J\ne J'}}
        \int_J\int_{J'}
        \mathbb I\left\{v\le q-\frac1n\right\}
        \mathbb I\left\{w\ge q+\frac1n\right\}\\
        &\hspace{3.2cm}\times
        H_\lambda(q-v,w-q)F_S(v)F_B^\circ(w)
        \,\dd w\,\dd v\\
        &=
        \int_0^{q-1/n}\int_{q+1/n}^1
        H_\lambda(q-v,w-q)F_S(v)F_B^\circ(w)
        \,\dd w\,\dd v\\
        &=
        \mathcal I_{\lambda,n}(q).
\end{aligned}
\]

We apply McDiarmid's inequality to
$\widehat{\mathcal I}_{\lambda,n}(q)$, viewed as a function of the
independent triples $(U_J,Y_J,Z_J)$, $J\in\mathcal P_n$. Fix a
cell $J$. Since the diameter of $J$ is $1/n$, the cell cannot
intersect both $[0,q-1/n]$ and $[q+1/n,1]$. Therefore, when the
triple associated with $J$ is changed, the terms in
\eqref{eq:easy-Ihat} can change only through the first index $J$,
or only through the second index $J$, but not through both.

Suppose first that $J\cap[0,q-1/n]\ne\emptyset$. For every possible
value $U_J\in J$ such that $U_J\le q-1/n$, the absolute value of
the corresponding row of \eqref{eq:easy-Ihat} is bounded by
\[
\begin{aligned}
        &|J|Y_J
        \sum_{J'\in\mathcal P_n}
        |J'|\mathbb I\left\{U_{J'}\ge q+\frac1n\right\}Z_{J'}
        H_\lambda(q-U_J,U_{J'}-q)\\
        &\qquad\le
        |J|
        \sum_{J'\in\mathcal P_n}
        |J'|\mathbb I\left\{U_{J'}\ge q+\frac1n\right\}
        H_\lambda(q-U_J,U_{J'}-q)\\
        &\qquad\le
        2|J|(q-U_J)^{\lambda-1}\\
        &\qquad\le
        2a_J^-(q),
\end{aligned}
\]
where the penultimate inequality follows from
\eqref{eq:easy-discrete-right-mass}. Thus changing the triple
associated with $J$ changes the estimator by at most
$4a_J^-(q)$. The symmetric argument shows that, if
$J\cap[q+1/n,1]\ne\emptyset$, the estimator changes by at most
$4a_J^+(q)$. If $J$ intersects neither set, it does not affect the
estimator. Hence the sum of the squares of the bounded-difference
constants is at most
\[
\begin{aligned}
        16\sum_{J\in\mathcal P_n}
        \left[
        (a_J^-(q))^2+(a_J^+(q))^2
        \right]
        &\le
        \frac{64(1+\log n)}n,
\end{aligned}
\]
by \eqref{eq:easy-a-square-bounds}. McDiarmid's inequality therefore
gives
\[
\begin{aligned}
        \Prob\left(
        \left|
        \widehat{\mathcal I}_{\lambda,n}(q)
        -
        \mathcal I_{\lambda,n}(q)
        \right|>\varepsilon_I
        \right) \le
        2\exp\left(
        -\frac{2\varepsilon_I^2}
        {64(1+\log n)/n}
        \right) =
        2e^{-2b_T}.
\end{aligned}
\]
A union bound over the $K$ grid prices and the three preceding
inequalities gives a failure probability at most
\[
        4Ke^{-b_T}+2Ke^{-2b_T}
        \le
        \frac1{T^2}.
\]
\end{proof}

On the event of Lemma~\ref{lem:easy-scouting-concentration}, set
\begin{equation}
        \varepsilon_T
        :=
        \varepsilon_A+\varepsilon_I+\frac4{\sqrt n}.
\label{eq:easy-total-error}
\end{equation}
Conditionally on the scouting observations, when a price
$q\in\mathcal G_K$ is posted after round $n$, the expectation of
\eqref{eq:easy-Ghat} is
\begin{equation}
\begin{aligned}
        \overline G_\lambda(q)
        &:=%
        2^{1-1/\lambda}\,
        \frac12
        \left[
        F_B^\circ(q)\widehat A^S(q)
        +
        F_S(q)\widehat A^B(q)
        \right]
        +
        \widetilde{\mathcal I}_{\lambda,n}(q).
\end{aligned}
\label{eq:easy-surrogate-mean}
\end{equation}
By Lemma~\ref{lem:main-reconstruction},
\eqref{eq:easy-A-concentration},
\eqref{eq:easy-I-concentration}, and
\eqref{eq:easy-truncation-bound},
\begin{align}
        |\overline G_\lambda(q)-G_\lambda(q)|
        &\le
        2^{1-1/\lambda}\,
        \frac12
        \left[
        F_B^\circ(q)\varepsilon_A
        +
        F_S(q)\varepsilon_A
        \right]
        +
        \left|
        \widetilde{\mathcal I}_{\lambda,n}(q)
        -
        \mathcal I_\lambda(q)
        \right|                                                   \notag\\
        &\le
        \varepsilon_A
        +
        \left|
        \widehat{\mathcal I}_{\lambda,n}(q)
        -
        \mathcal I_{\lambda,n}(q)
        \right|
        +
        \left|
        \mathcal I_{\lambda,n}(q)
        -
        \mathcal I_\lambda(q)
        \right|                                                   \notag\\
        &\le
        \varepsilon_A+\varepsilon_I+\frac4{\sqrt n}\\
        &=
        \varepsilon_T,
\label{eq:easy-surrogate-accuracy}
\end{align}
where we used that projection onto $[0,1]$
of $  \widehat{\mathcal I}_{\lambda,n}(q) \ge 0$
cannot increase the
distance to $\mathcal I_\lambda(q)\in[0,1]$.

Finally, we can provide the proof of Theorem~\ref{thm:easy-upper}.

\begin{proof}[Proof of Theorem~\ref{thm:easy-upper}]
Let $p^\star$ be a maximizer of $G_\lambda$, and let
$q^\star\in\mathcal G_K$ be a closest grid point. By
\eqref{eq:easy-G-lipschitz},
\begin{equation}
        G_\lambda^\star-G_\lambda(q^\star)
        \le
        \frac{4L}{K+1}.
\label{eq:easy-grid-error}
\end{equation}

The regret incurred during the first $n$ rounds is at most $n$.
Let $\mathcal E$ denote the event in
Lemma~\ref{lem:easy-scouting-concentration}.
Let $\mathcal{H}_n = (P_t,Y_t,Z_t)_{t=1}^n$, the history vector of the scouting phase.
On $\mathcal E$,
conditionally on $\mathcal{H}_n$, the rewards
$\widehat G_{\lambda,t}(q)$ belong to $[0,2]$, and their means are
given by $\overline G_\lambda(q)$. Therefore,
\eqref{eq:easy-bandit-guarantee} and
\eqref{eq:easy-surrogate-accuracy} give
\[
\begin{aligned}
        &\E\left[
        \sum_{t=n+1}^T
        \left(
        G_\lambda(q^\star)-G_\lambda(P_t)
        \right)
        \,\middle|\,
 \mathcal{H}_n
        \right]\\
        &\qquad\le
        \E\left[
        \sum_{t=n+1}^T
        \left(
        \overline G_\lambda(q^\star)
        -
        \overline G_\lambda(P_t)
        \right)
        \,\middle|\,
     \mathcal{H}_n
        \right]
        +
        2(T-n)\varepsilon_T\\
        &\qquad\le
        6\sqrt{K(T-n)\log(K)}
        +
        2T\varepsilon_T.
\end{aligned}
\]

Since the regret is at most $T$ on $\mathcal E^c$, we obtain
\begin{equation}
\begin{aligned}
        R_\lambda^\nu(\alpha,T)
        &\le
        n
        +
        \frac{4LT}{K+1}
        +
        6\sqrt{KT\log K}
        +
        2T\varepsilon_T
        +
        T\Prob(\mathcal E^c)\\
        &\le
        n
        +
        \frac{4LT}{K+1}
        +
        6\sqrt{KT\log K}
        +
        2T\varepsilon_T
        +
        \frac1T.
\end{aligned}
\label{eq:easy-regret-decomposition}
\end{equation}

We now bound the terms in \eqref{eq:easy-regret-decomposition}. For
$T\ge1$,
\begin{equation}
        n
        \le
        2T^{2/3}\log(eT),
        \qquad
        K
        \le
        2T^{1/3}.
\label{eq:easy-n-bound}
\end{equation}
Moreover, using $n\le T$,
\[
\begin{aligned}
        n
        &\ge
        T^{2/3}\log(eT),
        &
        b_T
        &=
        \log(8KT^2)
        \le
        4\log(eT),
        &
        1+\log n
        &\le
        \log(eT).
\end{aligned}
\]
It follows from \eqref{eq:easy-epsilons} that
\[
\begin{aligned}
        \varepsilon_A
        &\le
        \sqrt{
        \frac{4\log(eT)}
        {2T^{2/3}\log(eT)}
        }
        =
        \sqrt2\,T^{-1/3},\\
        \varepsilon_I
        &\le
        8\sqrt{
        \frac{4(\log(eT))^2}
        {T^{2/3}\log(eT)}
        }
        =
        16T^{-1/3}\sqrt{\log(eT)},\\
        \frac4{\sqrt n}
        &\le
        \frac4{T^{1/3}\sqrt{\log(eT)}}
        \le
        4T^{-1/3}.
\end{aligned}
\]
Since $\log(eT)\ge1$, we obtain
\begin{equation}
        \varepsilon_T
        \le
        32T^{-1/3}\sqrt{\log(eT)}.
\label{eq:easy-epsilon-final}
\end{equation}
Also, using
$\log K\le\log(2T^{1/3})\le2\log(eT)$,
\[
\begin{aligned}
        \frac{4LT}{K+1}
        &\le
        4LT^{2/3},\\
        6\sqrt{KT\log K}
        &\le
        6\sqrt{
        2T^{4/3}\,2\log(eT)
        }
        =
        12T^{2/3}\sqrt{\log(eT)},\\
        2T\varepsilon_T
        &\le
        64T^{2/3}\sqrt{\log(eT)}.
\end{aligned}
\]
Substituting these inequalities and \eqref{eq:easy-n-bound} into
\eqref{eq:easy-regret-decomposition}, and using
$\sqrt{\log(eT)}\le\log(eT)$, gives
\[
\begin{aligned}
        R_\lambda^\nu(\alpha,T)
        &\le
        \left[
        2+4L+12+64+1
        \right]
        T^{2/3}\log(eT)\\
        &\le
        2^8(1+L)T^{2/3}\log(eT),
\end{aligned}
\]

which proves \eqref{eq:easy-regret-bound}.
\end{proof}

\section{Generic revealing-action lower-bound template}\label{app:generic-template}

We use the following template, reproduced from Appendix~G of \cite{bachoc2026repeated}.

\begin{assumption}[Generic revealing-action structure]\label{ass:generic-revealing-template} Fix positive constants $\eps_0,c_{\mathrm{opt}},c_{\mathrm{gap}},c_{\mathrm{info}}$.
For every $\eps\in(0,\eps_0]$, there are two environments $+\eps$ and $-\eps$, and a partition $[0,1]=\Xminus\cup\Xinfo\cup\Xplus$, satisfying the following. (A1) Under $+\eps$, the best action lies in $\Xplus$, and under $-\eps$, the best action lies in $\Xminus$.  (A2) The wrong side is worse by at least $c_{\mathrm{opt}}\eps$. (A3) Every action in $\Xinfo$ is worse than the optimum by at least $c_{\mathrm{gap}}$.  (A4) Outside $\Xinfo$, the feedback law is identical in the two environments. (A5) Inside $\Xinfo$, one observation has KL divergence at most $c_{\mathrm{info}}\eps^2$.
\end{assumption}

\begin{theorem}[Generic $T^{2/3}$ lower-bound template]\label{thm:generic-lower-template}
Under Assumption~\ref{ass:generic-revealing-template}, if
\[
        T\ge
        \max\left\{
        {\frac{c_{\mathrm{gap}}}{4c_{\mathrm{info}}c_{\mathrm{opt}}\eps_0^3}},
        {\frac{2c_{\mathrm{opt}}^2}{c_{\mathrm{info}}c_{\mathrm{gap}}^2}}
        \right\},
\]
then every algorithm suffers, on at least one of the two environments with
\[
        \eps_T=\left({\frac{c_{\mathrm{gap}}}{4c_{\mathrm{info}}c_{\mathrm{opt}}T}}\right)^{1/3},
\]
regret at least
\[
        {\frac{c_{\mathrm{gap}}^{1/3}c_{\mathrm{opt}}^{2/3}}{4^{4/3}c_{\mathrm{info}}^{1/3}}}T^{2/3}.
\]
\end{theorem}

\begin{proof}
This is Theorem 37 of \cite{bachoc2026repeated}.  We invoke it as a black-box information-theoretic template in Appendix~\ref{app:smooth-lower}; all constants are substituted explicitly there.
\end{proof}

\section{Smooth five-tower revealing-action lower bound}\label{app:smooth-lower}

\begin{theorem}[Smooth five-tower lower bound]
\label{thm:smooth-five-tower-lower}
Fix any H\"older order $\lambda\in[0,1]$.
Let $L\ge 20$.
For every horizon
\[
        T\ge 32
\]
and every randomized posted-price policy $\alpha$, there exists an independent i.i.d. bilateral-trade instance whose seller and buyer marginal densities are bounded by $20$ and hence by $L$, such that
\[
        R_\lambda(\alpha,T)
        \ge
        \frac{1}{12800}T^{2/3}.
\]
Equivalently,
\[
        \inf_{\alpha}\sup_{\nu\in\mathcal D_L}
        R_\lambda^\nu(\alpha,T)
        \ge
        \frac{1}{12800}T^{2/3},
        \qquad L\ge20,\quad T\ge32.
\]
The constants are uniform in $\lambda\in[0,1]$.
\end{theorem}

\begin{remark}[What ``smooth'' means here]
We now explain how to make the densities used below continuous (or even $C^\infty$, if desired). First note that in fact, the densities are piecewise constant, and moreover, are everywhere bounded by $20$. As a consequence, the resulting cdfs are $20$-Lipschitz. If $C^\infty$ densities are desired, it suffices to replace each of the towers used below with a $C^\infty$-smooth, compactly supported mollified version of the original tower, keeping the same mass and its support inside the original interval, and having height at most $21$. For each fixed $\eps\in(0,1/50]$, the four reward-margin inequalities in Lemma~\ref{lem:reward-margins} continue to hold when the smoothing radius is sufficiently small, possibly depending on $\eps$: the proof gives strict slack uniformly in $\lambda\in[0,1]$, and bounded rewards make the change in each expected reward uniformly small as the product distributions converge in total variation. The proof below uses the original, piecewise-constant densities, which are allowed in the model with bounded densities considered by \cite{cesa2023smoothT34COLT}.
\end{remark}

\subsection{The five-tower construction}

For $\lambda\in[0,1]$, let
\[
M_0(x,y)=\sqrt{xy},
    \qquad
M_\lambda(x,y)=\left(\frac{x^\lambda+y^\lambda}{2}\right)^{1/\lambda}
    \quad (0<\lambda\le1)
\]
where $M_\lambda$ is the H\"older mean, and with the convention $0^\lambda=0$ for $\lambda>0$. The realized reward is
\[
    \gft_\lambda(p,s,b)
    =\mathbb{I}\{s\le p\le b\}M_\lambda(p-s,b-p).
\]
We set
\[
a=\frac{1}{10},\qquad c=\frac15,
    \qquad d=\frac{3}{10},
    \qquad r=\frac35,
    \qquad \eps_0=\frac1{50}
\]
and replace the five atoms
\[
    0,\frac1{10},\frac15,\frac3{10},\frac35
\]
on the seller side by five disjoint intervals, given by
\[
\begin{aligned}
J_0&=\left[0,\frac1{400}\right]=:[s_1^-,s_1^+],\\
J_a&=\left[\frac{79}{800},\frac{81}{800}\right]=:[s_2^-,s_2^+],\\
J_c&=\left[\frac{159}{800},\frac{161}{800}\right]=:[s_3^-,s_3^+],\\
J_d&=\left[\frac{239}{800},\frac{241}{800}\right]=:[s_4^-,s_4^+],\\
J_r&=\left[\frac{463}{800},\frac{497}{800}\right]=:[s_5^-,s_5^+].
\end{aligned}
\]
The five intervals on the buyer side are the reflected version of these intervals:
\[
\begin{aligned}
K_{2/5}&=\left[\frac{303}{800},\frac{337}{800}\right]=:[b_1^-,b_1^+],\\
K_{7/10}&=\left[\frac{559}{800},\frac{561}{800}\right]=:[b_2^-,b_2^+],\\
K_{4/5}&=\left[\frac{639}{800},\frac{641}{800}\right]=:[b_3^-,b_3^+],\\
K_{9/10}&=\left[\frac{719}{800},\frac{721}{800}\right]=:[b_4^-,b_4^+],\\
K_1&=\left[\frac{399}{400},1\right]=:[b_5^-,b_5^+].
\end{aligned}
\]
Now define the seller-side density by
\[
\begin{aligned}
f_{S,\eps}(s)
    &:=20\,\mathbb{I}_{J_0}(s)
    +(10-400\eps)\,\mathbb{I}_{J_a}(s)
    +(10+400\eps)\,\mathbb{I}_{J_c}(s)\\
    &\hspace{2cm}+20\,\mathbb{I}_{J_d}(s)
    +20\,\mathbb{I}_{J_r}(s)
\end{aligned}
\]
for $\eps\in[-1/50,1/50]$. We have
\[
\begin{aligned}
\int_{J_0}f_{S,\eps}&=\frac1{20},\qquad
\int_{J_a}f_{S,\eps}=\frac1{40}-\eps,\qquad
\int_{J_c}f_{S,\eps}=\frac1{40}+\eps,\\
\int_{J_d}f_{S,\eps}&=\frac1{20},\qquad
\int_{J_r}f_{S,\eps}=\frac{17}{20}
,
\end{aligned}
\]
and thus (by $|\eps|\le1/50$) the two perturbed heights satisfy
\[
    10-400\eps\in[2,18],
    \qquad
    10+400\eps\in[2,18],
\]
which in particular implies $0\le f_{S,\eps}\le20$. The buyer-side density is kept fixed to
\[
f_B(b)
    :=20\,\mathbb{I}_{K_{2/5}}(b)
    +20\,\mathbb{I}_{K_{7/10}}(b)
    +10\,\mathbb{I}_{K_{4/5}}(b)
    +10\,\mathbb{I}_{K_{9/10}}(b)
    +20\,\mathbb{I}_{K_1}(b),
\]
the reflected version of the unperturbed seller-side density, which ensures $0\le f_B\le20$. We take $S_\eps\perp B$. Now let us write
\[
G_{\lambda,\eps}(p)
    :=\mathbb E\bigl[\gft_\lambda(p,S_\eps,B)\bigr]
\]
for each $p\in[0,1]$. The two hard environments used in Theorem~\ref{thm:generic-lower-template} are $+\eps$ and $-\eps$.

The smoothed five-tower construction is shown in Figure~\ref{fig:five-tower-instance}. The perturbation only affects two of the five seller towers (centered at $a=1/10$ and $c=1/5$, respectively), keeping the buyer-side density fixed. The shaded interval is the revealing region $\Xinfo$.

\begin{figure}[t]
\centering
\includegraphics[width=0.88\textwidth]{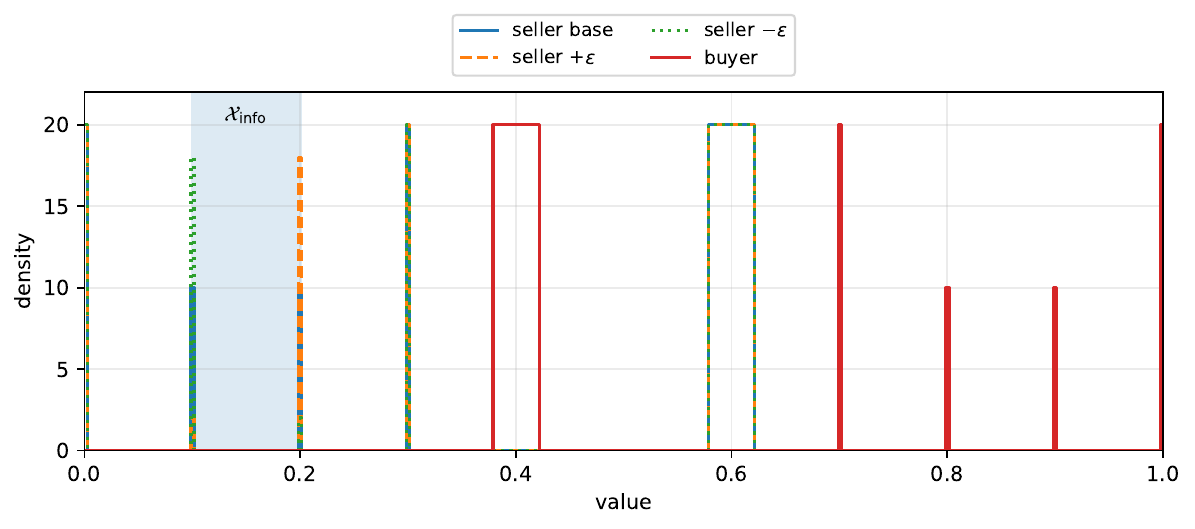}
\caption{The smooth five-tower lower-bound instance.  The plot shows the base seller density, the two perturbed seller densities $f_{S,\pm\eps_0}$ with $\eps_0=1/50$, and the fixed buyer density $f_B$.  Only the two seller towers around $a=1/10$ and $c=1/5$ are perturbed.  The shaded interval is the revealing region $\Xinfo=[79/800,161/800]$.}
\label{fig:five-tower-instance}
\end{figure}

\subsection{Action partition}

Let
\[
    \Xinfo:=\left[\frac{79}{800},\frac{161}{800}\right],
    \qquad
    \Xminus:=\left[0,\frac12\right)\setminus \Xinfo,
    \qquad
    \Xplus:=\left[\frac12,1\right].
\]
Then these three sets are disjoint and form a partition of $[0,1]$, where $\Xinfo$ is the smoothed version of the revealing interval, from the left edge of the tower around $a=1/10$ to the right edge of the tower around $c=1/5$.

In fact, Figure~\ref{fig:five-tower-curves} shows the geometric picture behind the lower bound. The top panel shows the perturbed reward curves $G_{\lambda,\pm\eps_0}$ for several values of $\lambda\in(0,1)$. The bottom panel shows the per-round information
\[
    \KL\bigl(P^p_{+\eps_0}\,\|\,P^p_{-\eps_0}\bigr),
\]
where $P^p_{\pm\eps_0}$ denotes the law of the threshold-feedback pair $(Y,Z)$ at price $p$ under the environments $\pm\eps_0$. Since neither marginal density depends on $\lambda$, this quantity does not depend on $\lambda$. Note that the KL divergence is positive only inside $\Xinfo$, yet this informative region stays separated from the maximizing peaks by a visible constant reward gap. This is the revealing-action geometry used in the proof.

\begin{figure}[t]
\centering
\includegraphics[width=0.88\textwidth]{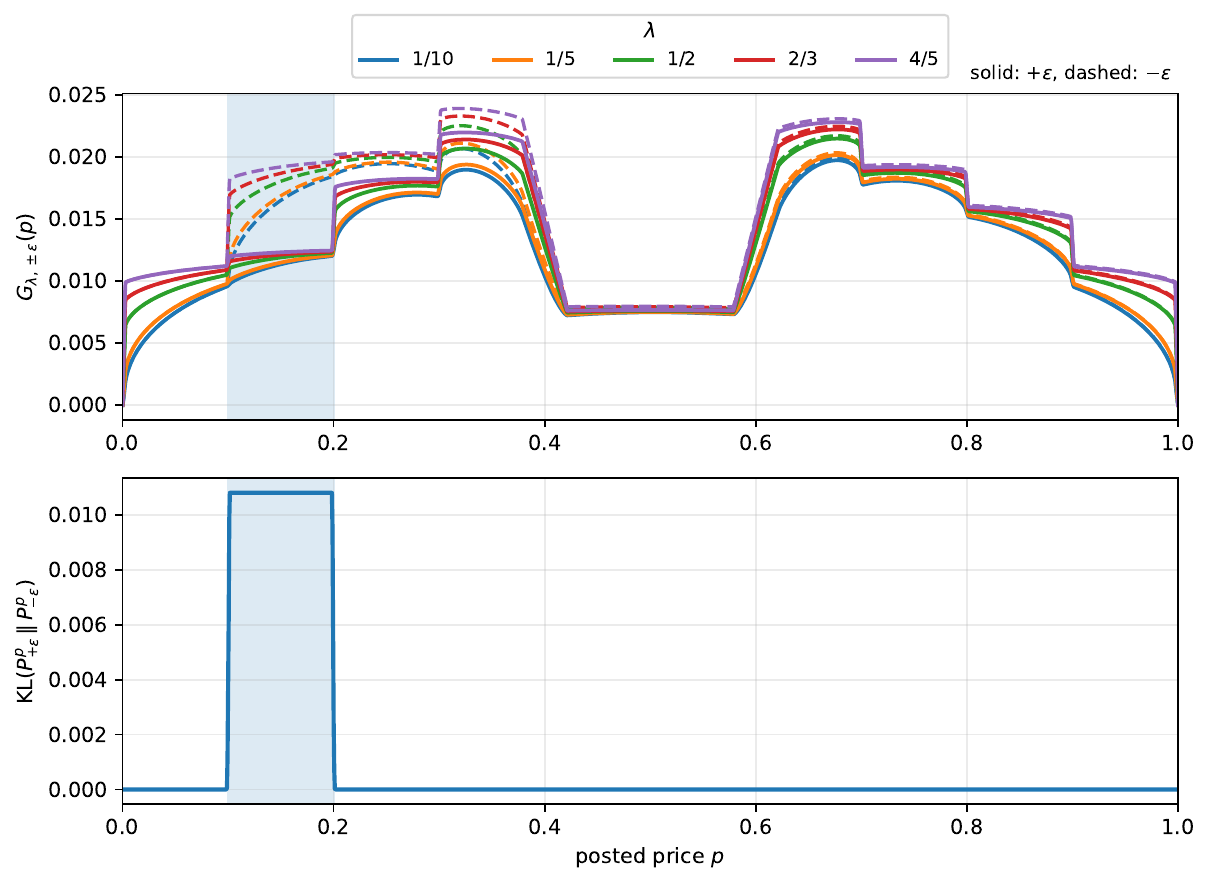}
\caption{Top: reward curves $G_{\lambda,\pm\eps_0}(p)$ for $\lambda\in\{1/10,1/5,1/2,2/3,4/5\}$, with solid lines for $+\eps_0$ and dashed lines for $-\eps_0$.  Bottom: the per-round KL divergence between the feedback laws under $+\eps_0$ and $-\eps_0$, which is independent of $\lambda$ because neither marginal density depends on $\lambda$.  Information is available only inside the shaded region $\Xinfo$, and prices in this region pay a visible $\Omega(1)$ reward gap compared with the maximizing peaks in the top panel.}
\label{fig:five-tower-curves}
\end{figure}

\subsection{The deterministic reward inequalities}

The next lemma is the deterministic core of the construction.
It is where the five-tower modification is used.

\begin{lemma}[Explicit reward margins]
\label{lem:reward-margins}
For every $\lambda\in[0,1]$ and every $\eps\in(0,1/50]$, the following four inequalities hold:
\begin{align}
    \sup_{p\in\Xplus}G_{\lambda,+\eps}(p)
    -\sup_{p\in\Xminus}G_{\lambda,+\eps}(p)
    &\ge \frac{1}{200}\eps,\label{eq:plus-side-margin}\\
    \sup_{p\in\Xminus}G_{\lambda,-\eps}(p)
    -\sup_{p\in\Xplus}G_{\lambda,-\eps}(p)
    &\ge \frac{1}{200}\eps,\label{eq:minus-side-margin}\\
    \sup_{p\in\Xplus}G_{\lambda,+\eps}(p)
    -\sup_{p\in\Xinfo}G_{\lambda,+\eps}(p)
    &\ge \frac{1}{2000},\label{eq:plus-info-margin}\\
    \sup_{p\in\Xminus}G_{\lambda,-\eps}(p)
    -\sup_{p\in\Xinfo}G_{\lambda,-\eps}(p)
    &\ge \frac{1}{2000}.\label{eq:minus-info-margin}
\end{align}
\end{lemma}

\begin{proof}
Write
\[
    e:=\frac1{50},\qquad q_0:=\frac{161}{800},\qquad
    p_0:=\frac{79}{200},\qquad
    p_-:=\frac13,\qquad p_+:=\frac23.
\]
We use a signed perturbation $\delta\in[-e,e]$. The unperturbed seller density and the buyer density are reflections of one another, so
\[
    G_{\lambda,0}(p)=G_{\lambda,0}(1-p).
\]
Let $U$ be uniform on $J_a$, independently of $B$. Since $U+1/10$ is uniform on $J_c$, define
\[
    D_\lambda(p):=
    \E\!\left[
    \gft_\lambda(p,U,B)-\gft_\lambda(p,U+1/10,B)
    \right].
\]
Increasing the seller's value cannot increase the reward, hence $D_\lambda(p)\ge0$. The dependence on the perturbation is exactly
\begin{equation}
    G_{\lambda,\delta}(p)
    =G_{\lambda,0}(p)-\delta D_\lambda(p).
\label{eq:margin-affine-perturbation}
\end{equation}
In particular, every fixed-price reward is affine in $\delta$.

\paragraph{Elementary endpoint bounds.}
We will only need explicit evaluations at orders $0$, $1/2$, and $1$. Recall that power means are nondecreasing in their order and that
\[
    G_{1/2,\delta}(p)
    =\frac12\bigl(G_{0,\delta}(p)+G_{1,\delta}(p)\bigr).
\]
At order zero, independence gives
\[
    G_{0,\delta}(p)=Q_S^\delta(p)Q_B(p),\qquad
    Q_S^\delta(p):=\E\sqrt{(p-S_\delta)_+},\qquad
    Q_B(p):=\E\sqrt{(B-p)_+}.
\]
Writing the seller intervals as $[s_i^-,s_i^+]$ with heights $h_i(\delta)$ and the buyer intervals as $[b_j^-,b_j^+]$ with heights $k_j$, their exact expressions are
\begin{align*}
    Q_S^\delta(p)
    &=\frac23\sum_{i=1}^5 h_i(\delta)
       \left[(p-s_i^-)_+^{3/2}-(p-s_i^+)_+^{3/2}\right],\\
    Q_B(p)
    &=\frac23\sum_{j=1}^5 k_j
       \left[(b_j^+-p)_+^{3/2}-(b_j^--p)_+^{3/2}\right].
\end{align*}
Direct calculation gives the following bounds. Each row means
$a/10^6<G_{0,\delta}(p)<b/10^6$:
\begin{center}
\begin{tabular}{cccc}
\toprule
$p$ & $\delta$ & $a$ & $b$\\
\midrule
$p_-$ & $-e$ & $20018$ & $20019$\\
$p_-$ & $e$  & $18483$ & $18484$\\
$p_+$ & $-e$ & $19333$ & $19334$\\
$p_+$ & $e$  & $19168$ & $19169$\\
$q_0$ & $-e$ & $18384$ & $18385$\\
$q_0$ & $0$  & $15542$ & $15543$\\
$q_0$ & $e$  & $12699$ & $12700$\\
\bottomrule
\end{tabular}
\end{center}

At order one, direct integration, or \eqref{eq:Gun}, gives
\begin{align*}
    G_{1,\delta}(q_0)&=\frac{6269}{320000}-\frac{\delta}{20},\\
    G_{1,\delta}(p_-)&=\frac{15217}{640000}-\frac{\delta}{20},
    &G_{1,\delta}(p_+)&=\frac{15217}{640000}-\frac{3\delta}{400},\\
    G_{1,\delta}(p_0)&=\frac{46113}{2560000}-\frac{27\delta}{800},
    &G_{1,\delta}(1-p_0)&=\frac{46113}{2560000}-\frac{3\delta}{400}.
\end{align*}
The preceding bounds imply, for every $\delta\in[-e,e]$,
\begin{align}
    G_{0,\delta}(p_-)-G_{1/2,\delta}(q_0)&>\frac1{2000},\notag\\
    G_{1/2,\delta}(p_-)-G_{1,\delta}(q_0)&>\frac1{2000},\notag\\
    G_{0,\delta}(p_-)-G_{1,\delta}(p_0)&>\frac1{1000},\notag\\
    G_{0,\delta}(p_+)-G_{1,\delta}(1-p_0)&>\frac1{1000}.
\label{eq:margin-endpoint-checks}
\end{align}
Indeed, each difference is affine in $\delta$, so it suffices to substitute $\delta=-e,e$. They also imply
\begin{align}
    G_{0,e}(p_+)-G_{1/2,0}(q_0)-\frac1{10000}&>\frac1{1000},\notag\\
    G_{1/2,e}(p_+)-G_{1,0}(q_0)-\frac1{10000}&>\frac1{1000}.
\label{eq:margin-right-endpoint-checks}
\end{align}

\paragraph{Localization of the best price on each side.}
First, $G_{0,\delta}$ and $G_{1,\delta}$ are nondecreasing on $[0,q_0]$. To check the order-zero assertion, for $0<p\le q_0$,
\[
    (Q_S^\delta)'(p)\ge\frac{Q_S^\delta(p)}{2p}.
\]
The buyer has mass $17/20$ in $K_{2/5}$ and mass $3/20$ in its other towers. In this price range, the respective distances from $p$ are at least $71/400$ and $199/400$. Since their square roots exceed $21/50$ and $7/10$, respectively,
\[
    Q_B(p)\ge\frac{231}{500},\qquad
    -2Q_B'(p)\le\frac{47}{21},\qquad
    \frac{47/21}{231/500}<\frac{800}{161}\le\frac1p.
\]
Thus $(Q_S^\delta Q_B)'(p)\ge0$. At order one, all buyers accept throughout $[0,q_0]$, and raising the price only adds nonnegative trade contributions. The identity for $G_{1/2,\delta}$ then gives its monotonicity as well. Consequently, using the first two inequalities of \eqref{eq:margin-endpoint-checks} for $\lambda\in[0,1/2]$ and $\lambda\in[1/2,1]$, respectively,
\begin{equation}
    G_{\lambda,\delta}(p_-)
    -\sup_{p\in[0,q_0]}G_{\lambda,\delta}(p)
    >\frac1{2000}.
\label{eq:margin-low-price-separation}
\end{equation}

There are no seller values in $[p_0,1/2]$, so $G_{1,\delta}$ is nonincreasing there. There are no buyer values in $[1/2,1-p_0]$, so $G_{1,\delta}$ is nondecreasing there. Since $G_{0,\delta}\le G_{\lambda,\delta}\le G_{1,\delta}$, the last two inequalities of \eqref{eq:margin-endpoint-checks} exclude these two central intervals from the respective side maxima.

It remains to exclude $p\ge1-q_0$ from the right-side maximum. For $x\ge y>0$ and $0\le\lambda\le1$,
\[
    \partial_x M_\lambda(x,y)
    =\frac12 M_\lambda(1,y/x)^{1-\lambda}\le\frac12,
\]
where the formula at $\lambda=0$ follows directly from the geometric mean. In the difference defining $D_\lambda(p)$, for $p\ge1-q_0$ all intermediate seller gains exceed the buyer gain. Also $\Prob(B\ge p)\le1/10$. Hence
\[
    0\le D_\lambda(p)\le\frac1{10}\cdot\frac12\cdot\frac1{10}
    =\frac1{200}.
\]
By reflection and \eqref{eq:margin-affine-perturbation},
\[
    G_{\lambda,\delta}(p)
    \le G_{\lambda,0}(1-p)+\frac1{10000}.
\]
For $\lambda\in[0,1/2]$, the right-hand side is at most $G_{1/2,0}(q_0)+1/10000$; for $\lambda\in[1/2,1]$, it is at most $G_{1,0}(q_0)+1/10000$. On the other hand, $D_\lambda\ge0$ implies $G_{\lambda,\delta}(p_+)\ge G_{\lambda,e}(p_+)$. The two inequalities of \eqref{eq:margin-right-endpoint-checks} therefore exclude $[1-q_0,1]$.

The rewards are continuous in the price, since the marginals have densities. The strict exclusions just proved show that the two side suprema are attained at prices in
\begin{equation}
    I_-:=(q_0,p_0)\subset\Xminus,
    \qquad
    I_+:=(1-p_0,1-q_0)\subset\Xplus,
\label{eq:margin-localized-peaks}
\end{equation}
respectively, for every $\lambda\in[0,1]$ and $\delta\in[-e,e]$.

\paragraph{Comparison of reflected prices.}
For fixed $s,b$ and $s+1/10<p<b$,
\[
    M_\lambda(p-s,b-p)-M_\lambda(p-s-1/10,b-p)
    =\int_0^{1/10}\partial_xM_\lambda(p-s-t,b-p)\,\dd t.
\]
The derivative in the integrand is
$\frac12 M_\lambda(1,(b-p)/(p-s-t))^{1-\lambda}$, which is nonincreasing in $p$. After $p$ passes $b$, the trade contribution drops to zero. Thus each buyer-specific contribution to $D_\lambda(p)$ is nonincreasing for $p>q_0$. For every $p\in I_-$, this gives
\[
    D_\lambda(p)-D_\lambda(1-p)
    \ge D_\lambda(p_0)-D_\lambda(1-p_0).
\]
All buyer-specific differences on the right are nonnegative. Buyers in $K_{2/5}$ never trade at $1-p_0$. Retaining only their contributions at $p_0$ therefore gives a lower bound. For $x\ge y>0$, the inequality $M_\lambda(x,y)\ge\sqrt{xy}$ implies
\[
    \partial_xM_\lambda(x,y)
    \ge\frac12\left(\frac yx\right)^{(1-\lambda)/2}
    \ge\frac12\sqrt{\frac yx}.
\]
Here $x=p_0-U-t\in[155/800,237/800]$ for $0\le t\le1/10$, while $y=b-p_0\in[0,21/800]$ for the contributing buyers in $K_{2/5}$, so $x\ge y$. Using their density $20$, we obtain
\begin{align*}
    D_\lambda(p_0)-D_\lambda(1-p_0)
    &\ge
    20\cdot\frac1{10}\cdot\frac1{2\sqrt{237/800}}
    \int_0^{21/800}\sqrt y\,\dd y\\
    &=\frac7{400}\sqrt{\frac7{79}}
    >\frac1{200}.
\end{align*}
The last inequality follows by squaring, since $343>316$. Hence
\begin{equation}
    D_\lambda(p)-D_\lambda(1-p)>\frac1{200}
    \qquad(p\in I_-).
\label{eq:margin-reflected-gap}
\end{equation}

Fix $\eps\in(0,e]$. Under $+\eps$, take a price attaining the left-side maximum, which lies in $I_-$ by \eqref{eq:margin-localized-peaks}. Its reflection improves its reward by
$\eps(D_\lambda(p)-D_\lambda(1-p))>\eps/200$.
Under $-\eps$, reflect a price attaining the right-side maximum instead; the improvement is again greater than $\eps/200$.
This proves \eqref{eq:plus-side-margin} and \eqref{eq:minus-side-margin}.
Finally, \eqref{eq:margin-low-price-separation} separates $p_-\in\Xminus$ from every price in $\Xinfo\subset[0,q_0]$ by more than $1/2000$ in both environments. It gives \eqref{eq:minus-info-margin} directly, and together with \eqref{eq:plus-side-margin} gives \eqref{eq:plus-info-margin}.
\end{proof}

\subsection{Verification of the template assumptions}

We now verify Assumption~\ref{ass:generic-revealing-template} of Theorem~\ref{thm:generic-lower-template} from Appendix~G of \cite{bachoc2026repeated} with
\[
        \eps_0=\frac1{50},
        \qquad
        c_{\mathrm{opt}}=\frac1{200},
        \qquad
        c_{\mathrm{gap}}=\frac1{2000},
        \qquad
        c_{\mathrm{info}}=100.
\]

\subsubsection{Optimal side and side separation}

By \eqref{eq:plus-side-margin}, under $+\eps$ the best action in $\Xplus$ beats the best action in $\Xminus$ by at least $\eps/200$.
By \eqref{eq:plus-info-margin}, the best action in $\Xplus$ also beats every action in $\Xinfo$ by at least $1/2000$.
Therefore
\[
    \sup_{p\in\Xplus}G_{\lambda,+\eps}(p)
    =\sup_{p\in[0,1]}G_{\lambda,+\eps}(p),
\]
and
\[
    \sup_{p\in\Xplus}G_{\lambda,+\eps}(p)
    -\sup_{p\in\Xminus}G_{\lambda,+\eps}(p)
    \ge \frac{1}{200}\eps.
\]
This proves the $+\eps$ parts of assumptions (A1) and (A2) with $c_{\mathrm{opt}}=1/200$.
The $-\eps$ parts follow symmetrically from \eqref{eq:minus-side-margin} and \eqref{eq:minus-info-margin}.

\subsubsection{Info actions are uniformly worse}

For $p\in\Xinfo$, \eqref{eq:plus-info-margin} gives
\[
    G_{\lambda,+\eps}^\star-G_{\lambda,+\eps}(p)
    \ge \frac{1}{2000},
\]
and \eqref{eq:minus-info-margin} gives
\[
    G_{\lambda,-\eps}^\star-G_{\lambda,-\eps}(p)
    \ge \frac{1}{2000}.
\]
Thus assumption (A3) holds with $c_{\mathrm{gap}}=1/2000$.

\subsubsection{Outside info, the feedback is sign-blind}

For a posted price $p$, the feedback is
\[
    F_t^\eps(p)=\bigl(\mathbb{I}\{S_t^\eps\le p\},\mathbb{I}\{p\le B_t\}\bigr).
\]
The buyer distribution is identical under $+\eps$ and $-\eps$.
The only difference between the seller densities is
\[
    f_{S,+\eps}(s)-f_{S,-\eps}(s)
    = -800\eps\,\mathbb{I}_{J_a}(s)+800\eps\,\mathbb{I}_{J_c}(s).
\]
Since $J_a$ and $J_c$ have the same length $1/400$, this signed density has total integral zero.  Hence the seller cdfs agree for every
\[
    p\notin\left[\frac{79}{800},\frac{161}{800}\right]=\Xinfo.
\]
Therefore the law of the feedback is the same under $+\eps$ and $-\eps$ for all $p\notin\Xinfo$.  Assumption (A4) holds.

\subsubsection{Info actions carry limited information}

It remains to verify assumption (A5).  Fix $p\in\Xinfo$ and let
\[
    q_+(p)=\mathbb P(S_{+\eps}\le p),
    \qquad
    q_-(p)=\mathbb P(S_{-\eps}\le p).
\]
The buyer bit has the same law under $+\eps$ and $-\eps$, and it is independent of the seller bit.  Hence
\[
    \KL\bigl(\mathcal L(F_t^{+\eps}(p))\,\|\,\mathcal L(F_t^{-\eps}(p))\bigr)
    =\KL\bigl(\Ber(q_+(p))\,\|\,\Ber(q_-(p))\bigr).
\]
For Bernoulli variables,
\[
    \KL(\Ber(u)\|\Ber(v))
    \le \frac{(u-v)^2}{v(1-v)}.
\]
On $\Xinfo$, the cdf difference satisfies
\[
    |q_+(p)-q_-(p)|\le 2\eps.
\]
Also, for every $p\in\Xinfo$ and every $\eps\in(0,1/50]$,
\[
    q_-(p)\in\left[\frac1{20},\frac1{10}\right].
\]
Consequently,
\[
    q_-(p)(1-q_-(p))
    \ge \frac1{20}\cdot\frac9{10}
    =\frac9{200}.
\]
Therefore
\[
    \KL\bigl(\mathcal L(F_t^{+\eps}(p))\,\|\,\mathcal L(F_t^{-\eps}(p))\bigr)
    \le
    \frac{4\eps^2}{9/200}
    =\frac{800}{9}\eps^2
    \le 100\eps^2.
\]
Thus (A5) holds with $c_{\mathrm{info}}=100$.

\subsection{Application of Theorem~\ref{thm:generic-lower-template}}

Theorem~\ref{thm:generic-lower-template} applies with
\[
        \eps_0=\frac1{50},
        \qquad
        c_{\mathrm{opt}}=\frac1{200},
        \qquad
        c_{\mathrm{gap}}=\frac1{2000},
        \qquad
        c_{\mathrm{info}}=100.
\]
Its threshold is
\begin{align*}
    T_0
    &=\max\left\{
        \frac{c_{\mathrm{gap}}}{4c_{\mathrm{info}}c_{\mathrm{opt}}\eps_0^3},
        \frac{2c_{\mathrm{opt}}^2}{c_{\mathrm{info}}c_{\mathrm{gap}}^2}
    \right\}\\
    &=\max\left\{
        \frac{1/2000}{4\cdot100\cdot(1/200)\cdot(1/50)^3},
        \frac{2\cdot(1/200)^2}{100\cdot(1/2000)^2}
    \right\}\\
    &=\max\left\{\frac{125}{4},2\right\}
    =\frac{125}{4}.
\end{align*}
For every $T\ge32$, Theorem~\ref{thm:generic-lower-template} chooses
\[
    \eps_T
    =\left(\frac{c_{\mathrm{gap}}}{4c_{\mathrm{info}}c_{\mathrm{opt}}T}\right)^{1/3}
    =\left(\frac{1}{4000T}\right)^{1/3}.
\]
Since $T\ge32$, the theorem guarantees $\eps_T\le\eps_0=1/50$.
Moreover, for every randomized policy $\alpha$,
\begin{align*}
    \max\left\{R_{\lambda,+\eps_T}(\alpha,T),R_{\lambda,-\eps_T}(\alpha,T)\right\}
    &\ge
    \frac{c_{\mathrm{gap}}^{1/3}c_{\mathrm{opt}}^{2/3}}{4^{4/3}c_{\mathrm{info}}^{1/3}}T^{2/3}\\
    &=
    \frac{(1/2000)^{1/3}(1/200)^{2/3}}{4^{4/3}\,100^{1/3}}T^{2/3}\\
    &=
    \frac{1}{2000\cdot4^{4/3}}T^{2/3}\\
    &\ge
    \frac{1}{12800}T^{2/3}.
\end{align*}
where the last inequality uses $4^{4/3}<6.4$.
Thus one of the two bounded-density instances, corresponding to $+\eps_T$ or $-\eps_T$, gives regret at least $T^{2/3}/12800$.
This proves Theorem~\ref{thm:smooth-five-tower-lower}.

\section{Proof of Theorem~\ref{thm:linear-arbitrary}}
\label{app:no-density}

In this appendix we show how to derive a lower bound on the regret for randomized algorithms from the needle instance, bypassing a reduction from a lower bound for either the pure-exploration or the fixed-confidence objective. This proves Theorem~\ref{thm:linear-arbitrary}. The construction is an exact atomic needle: for each hidden location $\theta$, the unique maximizer is $p=\theta$, and every other price has a fixed $\lambda$-fair gap.

Let $\lambda \in (0,1]$. To simplify the formulas for trades in which one participant receives no gain, we write
\[
        b_\lambda:=2^{-1/\lambda}.
\]
Thus $M_\lambda(x,0)=b_\lambda x$ and $M_\lambda(0,y)=b_\lambda y$.  Set
\[
        d:=\frac14,
        \qquad
        \Theta:=\left[\frac13,\frac23\right].
\]
For each $\theta\in\Theta$ and $\eta\in(0,1)$, the seller's value will be either $\theta-d$ or $\theta$, and the buyer's value will be either $\theta$ or $\theta+d$. We draw their values independently according to the following distributions:
\begin{equation}\label{eq:atomic-needle-seller}
        \Prob_\theta(S=\theta-d)=1-\eta,
        \qquad
        \Prob_\theta(S=\theta)=\eta,
\end{equation}
and
\begin{equation}\label{eq:atomic-needle-buyer}
        \Prob_\theta(B=\theta)=\eta,
        \qquad
        \Prob_\theta(B=\theta+d)=1-\eta.
\end{equation}
All atoms lie in $[0,1]$, since $\theta\in[1/3,2/3]$ and $d=1/4$.  We denote the corresponding product distribution by $\nu_\theta$.

\subsection{A fixed gap away from the hidden needle}

For $u\in[0,1]$, define
\[
        A_\lambda(u):=M_\lambda(1+u,1-u),
        \qquad
        R_\lambda(u):=M_\lambda(u,1-u).
\]
Note that it is elementary to show that $A_\lambda(u)$ is strictly decreasing in $u \in [0,1]$ for $\lambda \in (0,1)$, while $A_1(u)=1$.
If $0<\lambda<1$, set
\begin{equation}\label{eq:rho-explicit}
        \rho_\lambda:=\left((3/2)^\lambda-1\right)^{1/\lambda}
\end{equation}
and
\begin{equation}\label{eq:m-explicit}
        m_\lambda:=1-M_\lambda(1+\rho_\lambda,1-\rho_\lambda)>0.
\end{equation}
Define
\begin{equation}\label{eq:eta-small-explicit}
        \eta_\lambda:=\min\left\{\frac12,\frac{m_\lambda}{2(1+m_\lambda)}\right\},
        \qquad 0<\lambda<1.
\end{equation}
If $\lambda= 1$, define
\begin{equation}\label{eq:eta-large-explicit}
        \eta_\lambda:=\frac34.
\end{equation}
Finally, set
\begin{equation}\label{eq:gamma-lambda-explicit}
        \gamma_\lambda:=
        \begin{cases}
        \displaystyle
        \frac{d(1-\eta_\lambda)}{2}\min\{\eta_\lambda b_\lambda,m_\lambda\},
        &0<\lambda<1,\\[8pt]
        \displaystyle
        \frac{b_\lambda}{32},
        &\lambda = 1.
        \end{cases}
\end{equation}
This is the constant used in Theorem~\ref{thm:linear-arbitrary}.

\begin{lemma}[Atomic needle gap]\label{lem:atomic-needle-gap-direct}
For every $\lambda \in (0,1]$, every $\theta\in\Theta$, and $\eta=\eta_\lambda$, the instance $\nu_\theta$ in \eqref{eq:atomic-needle-seller}-\eqref{eq:atomic-needle-buyer} satisfies
\[
        G_\lambda^\theta(\theta)-G_\lambda^\theta(p)\ge \gamma_\lambda
        \qquad\text{for every }p\in[0,1]\setminus\{\theta\}.
\]
In particular, $\theta$ is the unique maximizer of $G_\lambda^\theta$.
\end{lemma}

\begin{proof}
Fix $\theta\in\Theta$ and write $\eta=\eta_\lambda$.  At the price $p=\theta$, the pair $(S,B)=(\theta-d,\theta+d)$ gives reward $d$, while the two boundary pairs $(\theta-d,\theta)$ and $(\theta,\theta+d)$ each give reward $b_\lambda d$.  The pair $(\theta,\theta)$ gives zero.  Hence
\begin{equation}\label{eq:G-needle-at-theta}
        G_\lambda^\theta(\theta)
        =d(1-\eta)^2+2d\eta(1-\eta)b_\lambda.
\end{equation}
If $p\notin[\theta-d,\theta+d]$, no pair trades, and therefore $G_\lambda^\theta(p)=0$.  Since $\gamma_\lambda\le d(1-\eta)b_\lambda/2$ in both cases of \eqref{eq:gamma-lambda-explicit}, this already gives the required bound.

It remains to consider $p\in[\theta-d,\theta+d]\setminus\{\theta\}$.  By symmetry, write
\[
        p=\theta+\sigma d u,
        \qquad
        \sigma\in\{-1,+1\},
        \quad u\in(0,1].
\]
The centered pair $(\theta-d,\theta+d)$ contributes $dA_\lambda(u)$.  Exactly one boundary pair contributes $dR_\lambda(u)$, and the other boundary pair does not trade.  Thus
\begin{equation}\label{eq:G-needle-away}
        G_\lambda^\theta(\theta+\sigma du)
        =d(1-\eta)^2A_\lambda(u)+d\eta(1-\eta)R_\lambda(u).
\end{equation}
Subtracting \eqref{eq:G-needle-away} from \eqref{eq:G-needle-at-theta} gives
\begin{equation}\label{eq:needle-gap-bracket-direct}
\begin{split}
        &G_\lambda^\theta(\theta)-G_\lambda^\theta(\theta+\sigma du)\\
        &\quad=d(1-\eta)
        \left[(1-\eta)(1-A_\lambda(u))+\eta(2b_\lambda-R_\lambda(u))\right].
\end{split}
\end{equation}

Assume first that $0<\lambda<1$.  For $u\in[0,\rho_\lambda]$, we have $A_\lambda(u)\le1$.  Also,
\[
        R_\lambda(u)^\lambda
        ={\frac{u^\lambda+(1-u)^\lambda}{2}}
        \le {\frac{u^\lambda+1}{2}}
        \le {\frac{(3/2)^\lambda}{2}}
        =\left({\frac{3b_\lambda}{2}}\right)^\lambda,
\]
where the second inequality uses $u\le\rho_\lambda$ and \eqref{eq:rho-explicit}.  Hence $R_\lambda(u)\le 3b_\lambda/2$, and the bracket in \eqref{eq:needle-gap-bracket-direct} is at least $\eta b_\lambda/2$.

For $u\in[\rho_\lambda,1]$, \eqref{eq:m-explicit} gives $1-A_\lambda(u)\ge m_\lambda$.  The crude bound $R_\lambda(u)\le1$ gives
\[
\begin{split}
        &(1-\eta)(1-A_\lambda(u))+\eta(2b_\lambda-R_\lambda(u))\\
        &\quad\ge (1-\eta)m_\lambda-\eta
        =m_\lambda-\eta(1+m_\lambda)
        \ge {\frac{m_\lambda}{2}},
\end{split}
\]
by \eqref{eq:eta-small-explicit}.  Combining the two regions and using $d(1-\eta)/2$ times the minimum of $\eta b_\lambda$ and $m_\lambda$ proves the claim for $0<\lambda<1$.

Assume now that $\lambda = 1$.  
Then $A_\lambda(u)\le A_\lambda(1)=2b_\lambda$ and $R_\lambda(u)\le R_\lambda(0)=b_\lambda$.  With $\eta=3/4$, the bracket in \eqref{eq:needle-gap-bracket-direct} is at least
\[
        -{\frac{1}{4}}(2b_\lambda-1)+{\frac{3}{4}}b_\lambda
        ={\frac{1+b_\lambda}{4}}
        \ge {\frac{b_\lambda}{2}}.
\]
Consequently,
\[
        G_\lambda^\theta(\theta)-G_\lambda^\theta(\theta+\sigma du)
        \ge d(1-\eta){\frac{b_\lambda}{2}}
        ={\frac{1}{4}}\cdot {\frac{1}{4}}\cdot {\frac{b_\lambda}{2}}
        ={\frac{b_\lambda}{32}}
        =\gamma_\lambda.
\]
Thus every price other than $\theta$ has a gap of at least $\gamma_\lambda$, no matter how close it is to $\theta$.
\end{proof}

\subsection{Randomization cannot hit a non-atomic needle}

The impossibility is a consequence of the "needle in a continuum": if the true value of the hidden parameter is drawn according to some non-atomic distribution on $\Theta$ (e.g. the uniform distribution), independently of the algorithm's private randomness, any given randomized threshold-feedback algorithm will choose a particular (random) price in each finite round, which will coincide with the hidden parameter with probability zero.

\begin{lemma}[Zero probability of an exact hit]\label{lem:exact-hit-zero}
Put the normalized Lebesgue prior on $\Theta$, and then run any possibly randomized posted-price algorithm on the instance $\nu_\theta$.  For every finite round $t$,
\[
        \Prob(P_t=\theta)=0,
\]
where the probability is over the prior on $\theta$, the learner's private randomization, and the valuation/feedback randomness.
\end{lemma}

\begin{proof}
It is enough to consider algorithms whose private randomization is represented by independent uniform $[0,1]$-valued seeds $U_1,U_2,\ldots$; this is without loss of generality by the standard randomization representation.  Fix a round $t$.  For a seed prefix $u=(u_1,\ldots,u_t)\in[0,1]^t$ and a feedback string
\[
        h=(y_1,z_1,\ldots,y_{t-1},z_{t-1})\in\{0,1\}^{2(t-1)},
\]
let $p_t(u,h)\in[0,1]$ be the price that the algorithm posts at round $t$ when the previous feedback string is $h$ and the seed prefix is $u$.

For any fixed value of $u$, $h$, and $\theta$, let $L_{t-1}(\theta;u,h)$ be the probability that the feedback string before round $t$ equals $h$, under instance $\nu_\theta$, when the seed prefix is $u$.  This likelihood is a bounded measurable function of $\theta$ with values in $[0,1]$.  Indeed, recursively fixing $u$ and $h$ fixes all previous prices, and each feedback outcome probability is a finite sum of products of the atom masses multiplied by threshold indicators involving $\theta-d,\theta,\theta+d$.

Let $\pi$ be normalized Lebesgue measure on $\Theta$.  By Tonelli's theorem,
\[
\begin{split}
        \Prob(P_t=\theta)
        &=\sum_{h\in\{0,1\}^{2(t-1)}}
        \int_{[0,1]^t}\int_\Theta
        \mathbb{I}\{p_t(u,h)=\theta\}
        L_{t-1}(\theta;u,h)
        \pi(\dd\theta)\dd u.
\end{split}
\]
For every fixed $u$ and $h$, the set $\{\theta\in\Theta:p_t(u,h)=\theta\}$ contains at most one point, hence has $\pi$-measure zero.  The inner integral is therefore zero, and the claim follows.
\end{proof}

\begin{proof}[Proof of Theorem~\ref{thm:linear-arbitrary}]
Fix $\lambda \in (0,1]$ and use the needle construction above with $\eta=\eta_\lambda$.  Let $\pi$ be normalized Lebesgue measure on $\Theta$, and consider the joint experiment in which $\theta\sim\pi$ is first drawn and the learner then interacts with $\nu_\theta$.

By Lemma~\ref{lem:atomic-needle-gap-direct}, for every $p\in[0,1]$,
\[
        G_\lambda^{\theta,\star}-G_\lambda^\theta(p)
        \ge \gamma_\lambda\mathbb{I}\{p\ne\theta\}.
\]
Therefore, using Lemma~\ref{lem:exact-hit-zero} at each round,
\[
\begin{split}
        \int_\Theta R_\lambda^\theta(\alpha,T)\,\pi(\dd\theta)
        &=\sum_{t=1}^T
        \E_{\theta\sim\pi,\alpha,\nu_\theta}
        \left[G_\lambda^{\theta,\star}-G_\lambda^\theta(P_t)\right]\\
        &\ge \gamma_\lambda\sum_{t=1}^T
        \Prob(P_t\ne\theta)
        =\gamma_\lambda T.
\end{split}
\]
Consequently, there exists at least one $\theta\in\Theta$ such that
\[
        R_\lambda^\theta(\alpha,T)\ge\gamma_\lambda T.
\]
The corresponding law $\nu_\theta$ is a product of two atomic marginals, hence an independent i.i.d. bilateral-trade instance with arbitrary marginals.  This proves the theorem with $c_\lambda=\gamma_\lambda$.
\end{proof}


\begin{thebibliography}{32}
	\providecommand{\natexlab}[1]{#1}
	\providecommand{\url}[1]{\texttt{#1}}
	\expandafter\ifx\csname urlstyle\endcsname\relax
	\providecommand{\doi}[1]{doi: #1}\else
	\providecommand{\doi}{doi: \begingroup \urlstyle{rm}\Url}\fi
	
	\bibitem[Auer et~al.(2002)Auer, Cesa-Bianchi, Freund, and
	Schapire]{auer2002nonstochastic}
	Peter Auer, Nicol{\`o} Cesa-Bianchi, Yoav Freund, and Robert~E. Schapire.
	\newblock The nonstochastic multiarmed bandit problem.
	\newblock \emph{SIAM Journal on Computing}, 32\penalty0 (1):\penalty0 48--77,
	2002.
	\newblock \doi{10.1137/S0097539701398375}.
	
	\bibitem[Azar et~al.(2022)Azar, Fiat, and Fusco]{azar2022alphaNeurIPS}
	Yossi Azar, Amos Fiat, and Federico Fusco.
	\newblock An $\alpha$-regret analysis of adversarial bilateral trade.
	\newblock In \emph{Advances in Neural Information Processing Systems},
	volume~35, pages 1685--1697, 2022.
	
	\bibitem[Bacchiocchi et~al.(2025)Bacchiocchi, Castiglioni, Colomboni, and
	Marchesi]{bacchiocchi2025NeurIPS}
	Francesco Bacchiocchi, Matteo Castiglioni, Roberto Colomboni, and Alberto
	Marchesi.
	\newblock Online bilateral trade with minimal feedback: Don’t waste
	seller’s time.
	\newblock In \emph{The Thirty-ninth Annual Conference on Neural Information
		Processing Systems (NeurIPS)}, 2025.
	
	\bibitem[Bachoc et~al.(2024{\natexlab{a}})Bachoc, Cesa-Bianchi, Cesari, and
	Colomboni]{bachoc2024fairNeurIPS}
	Fran\c{c}ois Bachoc, Nicol\`{o} Cesa-Bianchi, Tommaso Cesari, and Roberto
	Colomboni.
	\newblock Fair online bilateral trade.
	\newblock In \emph{Proceedings of the 38th International Conference on Neural
		Information Processing Systems (NeurIPS)}, 2024{\natexlab{a}}.
	
	\bibitem[Bachoc et~al.(2025{\natexlab{a}})Bachoc, Cesari, and
	Colomboni]{bachoc2025parcnxtICML}
	Fran\c{c}ois Bachoc, Tommaso Cesari, and Roberto Colomboni.
	\newblock A parametric contextual online learning theory of brokerage.
	\newblock In \emph{Proceedings of the 42nd International Conference on Machine
		Learning (ICML)}, volume 267 of \emph{Proceedings of Machine Learning
		Research}, pages 2293--2309. PMLR, 2025{\natexlab{a}}.
	
	\bibitem[Bachoc et~al.(2021{\natexlab{a}})Bachoc, Cesari, and
	Gerchinovitz]{bachoc2021instance}
	Fran{\c{c}}ois Bachoc, Tom Cesari, and S{\'e}bastien Gerchinovitz.
	\newblock Instance-dependent bounds for zeroth-order lipschitz optimization
	with error certificates.
	\newblock \emph{Advances in Neural Information Processing Systems},
	34:\penalty0 24180--24192, 2021{\natexlab{a}}.
	
	\bibitem[Bachoc et~al.(2021{\natexlab{b}})Bachoc, Cesari, and
	Gerchinovitz]{bachoc2021sample}
	Fran{\c{c}}ois Bachoc, Tommaso Cesari, and S{\'e}bastien Gerchinovitz.
	\newblock The sample complexity of level set approximation.
	\newblock In \emph{International Conference on Artificial Intelligence and
		Statistics}, pages 424--432. PMLR, 2021{\natexlab{b}}.
	
	\bibitem[Bachoc et~al.(2024{\natexlab{b}})Bachoc, Cesari, Colomboni, and
	Paudice]{bachoc2024zerothTMLR}
	Fran{\c{c}}ois Bachoc, Tommaso Cesari, Roberto Colomboni, and Andrea Paudice.
	\newblock A theoretical framework for zeroth-order budget convex optimization.
	\newblock \emph{Transactions on Machine Learning Research}, 2024{\natexlab{b}}.
	
	\bibitem[Bachoc et~al.(2025{\natexlab{b}})Bachoc, Boli{\'c}, Cesari, and
	Colomboni]{bachoc2025online}
	Fran{\c{c}}ois Bachoc, Nata{\v{s}}a Boli{\'c}, Tommaso Cesari, and Roberto
	Colomboni.
	\newblock An online learning theory of brokerage.
	\newblock \emph{HAL preprint hal-04987905}, 2025{\natexlab{b}}.
	
	\bibitem[Bachoc et~al.(2025{\natexlab{c}})Bachoc, Cesari, and
	Colomboni]{bachoc2025nonparcnxtAISTATS}
	Fran{\c{c}}ois Bachoc, Tommaso Cesari, and Roberto Colomboni.
	\newblock A tight regret analysis of non-parametric repeated contextual
	brokerage.
	\newblock In \emph{Proceedings of The 28th International Conference on
		Artificial Intelligence and Statistics (AISTATS)}, volume 258 of
	\emph{Proceedings of Machine Learning Research}, pages 2836--2844. PMLR,
	2025{\natexlab{c}}.
	
	\bibitem[Bachoc et~al.(2026)Bachoc, Colomboni, and
	Kaufmann]{bachoc2026repeated}
	Fran{\c{c}}ois Bachoc, Roberto Colomboni, and Emilie Kaufmann.
	\newblock Repeated bilateral trade: The quest for fairness.
	\newblock \emph{arXiv preprint arXiv:2606.15369}, 2026.
	
	\bibitem[Bart{\'o}k et~al.(2014)Bart{\'o}k, Foster, P{\'a}l, Rakhlin, and
	Szepesv{\'a}ri]{bartok2014partial}
	G{\'a}bor Bart{\'o}k, Dean~P. Foster, D{\'a}vid P{\'a}l, Alexander Rakhlin, and
	Csaba Szepesv{\'a}ri.
	\newblock Partial monitoring---classification, regret bounds, and algorithms.
	\newblock \emph{Mathematics of Operations Research}, 39\penalty0 (4):\penalty0
	967--997, 2014.
	\newblock \doi{10.1287/moor.2014.0663}.
	
	\bibitem[Bernasconi et~al.(2024)Bernasconi, Castiglioni, Celli, and
	Fusco]{bernasconi2024globalSTOC}
	Martino Bernasconi, Matteo Castiglioni, Andrea Celli, and Federico Fusco.
	\newblock No-regret learning in bilateral trade via global budget balance.
	\newblock In \emph{Proceedings of the 56th Annual ACM Symposium on Theory of
		Computing (STOC)}, pages 247--258, 2024.
	
	\bibitem[Boli{\'c} et~al.(2024)Boli{\'c}, Cesari, and
	Colomboni]{bolic2024onlineAAMAS}
	Natasa Boli{\'c}, Tommaso Cesari, and Roberto Colomboni.
	\newblock An online learning theory of brokerage.
	\newblock In \emph{Proceedings of the 23rd International Conference on
		Autonomous Agents and Multiagent Systems (AAMAS)}, pages 216--224, 2024.
	
	\bibitem[Boli\'{c} et~al.(2025)Boli\'{c}, Cesari, Colomboni, and
	Paravalos]{bolic2025newsvendorNeurIPS}
	Natasa Boli\'{c}, Tommaso Cesari, Roberto Colomboni, and Christian Paravalos.
	\newblock Online learning in the repeated mediated newsvendor problem.
	\newblock In \emph{The Thirty-ninth Annual Conference on Neural Information
		Processing Systems (NeurIPS)}, 2025.
	
	\bibitem[Boucheron et~al.(2013)Boucheron, Lugosi, and
	Massart]{boucheron2013concentration}
	St{\'e}phane Boucheron, Gabor Lugosi, and Pascal Massart.
	\newblock \emph{{Concentration inequalities : a non asymptotic theory of
			independence}}.
	\newblock {Oxford University Press}, 2013.
	
	\bibitem[Bouttier et~al.(2020)Bouttier, Cesari, Ducoffe, and
	Gerchinovitz]{bouttier2020regret}
	Cl{\'e}ment Bouttier, Tommaso Cesari, M{\'e}lanie Ducoffe, and S{\'e}bastien
	Gerchinovitz.
	\newblock Regret analysis of the {Piyavskii-Shubert} algorithm for global
	{Lipschitz} optimization.
	\newblock \emph{arXiv preprint arXiv:2002.02390}, 2020.
	
	\bibitem[Bubeck et~al.(2011)Bubeck, Munos, Stoltz, and
	Szepesv{\'a}ri]{bubeck2011xarmed}
	S{\'e}bastien Bubeck, R{\'e}mi Munos, Gilles Stoltz, and Csaba Szepesv{\'a}ri.
	\newblock {X}-armed bandits.
	\newblock \emph{Journal of Machine Learning Research}, 12\penalty0
	(46):\penalty0 1655--1695, 2011.
	
	\bibitem[Cesa-Bianchi et~al.(2021)Cesa-Bianchi, Cesari, Colomboni, Fusco, and
	Leonardi]{cesa2021bilateralEC}
	Nicol{\`o} Cesa-Bianchi, Tommaso~R Cesari, Roberto Colomboni, Federico Fusco,
	and Stefano Leonardi.
	\newblock A regret analysis of bilateral trade.
	\newblock In \emph{Proceedings of the 22nd ACM Conference on Economics and
		Computation}, EC '21, pages 289--309. Association for Computing Machinery,
	2021.
	
	\bibitem[Cesa-Bianchi et~al.(2023)Cesa-Bianchi, Cesari, Colomboni, Fusco, and
	Leonardi]{cesa2023smoothT34COLT}
	Nicol{\`o} Cesa-Bianchi, Tommaso~R. Cesari, Roberto Colomboni, Federico Fusco,
	and Stefano Leonardi.
	\newblock Repeated bilateral trade against a smoothed adversary.
	\newblock In \emph{Proceedings of Thirty Sixth Conference on Learning Theory
		(COLT)}, volume 195 of \emph{Proceedings of Machine Learning Research}, pages
	1095--1130. PMLR, 2023.
	
	\bibitem[Cesa-Bianchi et~al.(2024{\natexlab{a}})Cesa-Bianchi, Cesari,
	Colomboni, Fusco, and Leonardi]{cesa2024bilateralMOR}
	Nicol\`{o} Cesa-Bianchi, Tommaso Cesari, Roberto Colomboni, Federico Fusco, and
	Stefano Leonardi.
	\newblock Bilateral trade: A regret minimization perspective.
	\newblock \emph{Mathematics of Operations Research}, 49\penalty0 (1):\penalty0
	171–203, 2024{\natexlab{a}}.
	
	\bibitem[Cesa-Bianchi et~al.(2024{\natexlab{b}})Cesa-Bianchi, Cesari,
	Colomboni, Fusco, and Leonardi]{cesa2024smoothT34JMLR}
	Nicol{{\`o}} Cesa-Bianchi, Tommaso Cesari, Roberto Colomboni, Federico Fusco,
	and Stefano Leonardi.
	\newblock Regret analysis of bilateral trade with a smoothed adversary.
	\newblock \emph{Journal of Machine Learning Research (JMLR)}, 25\penalty0
	(234):\penalty0 1--36, 2024{\natexlab{b}}.
	
	\bibitem[Cesa-Bianchi et~al.(2025)Cesa-Bianchi, Cesari, Colomboni, Foscari, and
	Pathak]{cesa2025marketCOLT}
	Nicol{\`o} Cesa-Bianchi, Tommaso Cesari, Roberto Colomboni, Luigi Foscari, and
	Vinayak Pathak.
	\newblock Market making without regret.
	\newblock In \emph{Proceedings of the Thirty-Eighth Conference on Learning
		Theory}, volume 291 of \emph{Proceedings of Machine Learning Research}, pages
	799--837, 2025.
	
	\bibitem[Cesari and Colomboni(2025)]{cesari2025volumeICLR}
	Tommaso Cesari and Roberto Colomboni.
	\newblock An online learning theory of trading-volume maximization.
	\newblock In \emph{International Conference on Learning Representations
		(ICLR)}, volume 2025, 2025.
	
	\bibitem[Cesa‐Bianchi et~al.(2025)Cesa‐Bianchi, Colomboni, and
	Kasy]{Cesa2025econometrica}
	Nicolò Cesa‐Bianchi, Roberto Colomboni, and Maximilian Kasy.
	\newblock Adaptive maximization of social welfare.
	\newblock \emph{Econometrica}, 93\penalty0 (3):\penalty0 1073--1104, 2025.
	
	\bibitem[Chen et~al.(2026)Chen, Jin, Lu, and Zhang]{chen2026tightICALP}
	Houshuang Chen, Yaonan Jin, Pinyan Lu, and Chihao Zhang.
	\newblock Tight regret bounds for fixed-price bilateral trade.
	\newblock In \emph{53rd International Colloquium on Automata, Languages, and
		Programming (ICALP 2026)}, Leibniz International Proceedings in Informatics
	(LIPIcs), 2026.
	
	\bibitem[Cohn et~al.(1996)Cohn, Ghahramani, and Jordan]{cohn1996active}
	David~A. Cohn, Zoubin Ghahramani, and Michael~I. Jordan.
	\newblock Active learning with statistical models.
	\newblock \emph{Journal of Artificial Intelligence Research}, 4:\penalty0
	129--145, 1996.
	
	\bibitem[De~Freitas et~al.(2012)De~Freitas, Smola, and
	Zoghi]{de2012exponential}
	Nando De~Freitas, Alex Smola, and Masrour Zoghi.
	\newblock Exponential regret bounds for {Gaussian} process bandits with
	deterministic observations.
	\newblock \emph{International Conference on Machine Learning (ICML)}, 2012.
	
	\bibitem[He et~al.(2024)He, Jiang, Lam, and Fu]{he2024adaptive}
	Shengyi He, Guangxin Jiang, Henry Lam, and Michael~C. Fu.
	\newblock Adaptive importance sampling for efficient stochastic root finding
	and quantile estimation.
	\newblock \emph{Operations Research}, 72\penalty0 (6):\penalty0 2612--2630,
	2024.
	\newblock \doi{10.1287/opre.2023.2484}.
	
	\bibitem[Horvitz and Thompson(1952)]{horvitz1952generalization}
	Daniel~G. Horvitz and Donovan~J. Thompson.
	\newblock A generalization of sampling without replacement from a finite
	universe.
	\newblock \emph{Journal of the American Statistical Association}, 47\penalty0
	(260):\penalty0 663--685, 1952.
	
	\bibitem[McDiarmid(1989)]{mcdiarmid1989bounded}
	Colin McDiarmid.
	\newblock On the method of bounded differences.
	\newblock In J.~Siemons, editor, \emph{Surveys in Combinatorics, 1989}, volume
	141 of \emph{London Mathematical Society Lecture Note Series}, pages
	148--188. Cambridge University Press, 1989.
	\newblock \doi{10.1017/CBO9781107359949.008}.
	
	\bibitem[Munos(2011)]{munos2011optimistic}
	R{\'e}mi Munos.
	\newblock Optimistic optimization of a deterministic function without the
	knowledge of its smoothness.
	\newblock \emph{Advances in neural information processing systems}, 24, 2011.
	
\end{thebibliography}
\end{document}